\documentclass[12pt, a4]{amsart}

\usepackage{latexsym,amsmath,amssymb,amsthm,mathrsfs,color}
\usepackage[bbgreekl]{mathbbol}
\usepackage{bbm}
\usepackage{tikz-cd}
\usepackage[all]{xy}
\usepackage{stmaryrd}

\usepackage{hyperref} 
\hypersetup{colorlinks,%
    citecolor=brown,%
    filecolor=black,%
   linkcolor=blue,%
   urlcolor=black}

\usepackage[shortalphabetic]{amsrefs}
 \makeatletter
\def\append@label@year@{%
    \safe@set\@tempcnta\bib@year
    \edef\bib@citeyear{\the\@tempcnta}%
    \ifnum\bib@citeyear>9
      \append@to@stem{%
          \ifx\bib@year\@empty
          \else
            \@xp\year@short \bib@citeyear \@nil
          \fi
      }%
    \fi    
}
\makeatother

\let\oldtocsection=\tocsection
\renewcommand{\tocsection}[2]{\hspace{0em}\oldtocsection{#1}{#2}}

\usepackage[a4paper]{geometry}

\def\upddots{\mathinner{\mkern 1mu\raise 1pt \hbox{.}\mkern 2mu
\mkern 2mu \raise 4pt\hbox{.}\mkern 1mu \raise 7pt\vbox {\kern 7
pt\hbox{.}}} }

\begin{document}
   
\setlength{\unitlength}{2.5cm}

\newtheorem{thm}{Theorem}[section]
\newtheorem{lm}[thm]{Lemma}
\newtheorem{prop}[thm]{Proposition}
\newtheorem{cor}[thm]{Corollary}
\newtheorem{conj}[thm]{Conjecture}

\theoremstyle{definition}
\newtheorem{dfn}[thm]{Definition}
\newtheorem{eg}[thm]{Example}
\newtheorem{rmk}[thm]{Remark}

\newcommand{\F}{\mathbf{F}}
\newcommand{\N}{\mathbf{N}}
\newcommand{\R}{\mathbf{R}}
\newcommand{\C}{\mathbf{C}}
\newcommand{\Z}{\mathbf{Z}}
\newcommand{\Q}{\mathbf{Q}}

\newcommand{\Mp}{{\rm Mp}}
\newcommand{\Sp}{{\rm Sp}}
\newcommand{\GSp}{{\rm GSp}}
\newcommand{\GL}{{\rm GL}}
\newcommand{\PGL}{{\rm PGL}}
\newcommand{\SL}{{\rm SL}}
\newcommand{\SO}{{\rm SO}}
\newcommand{\Spin}{\text{Spin}}
\newcommand{\Ind}{{\rm Ind}}
\newcommand{\Res}{{\rm Res}}
\newcommand{\Hom}{{\rm Hom}}
\newcommand{\msc}[1]{\mathscr{#1}}
\newcommand{\mfr}[1]{\mathfrak{#1}}
\newcommand{\mca}[1]{\mathcal{#1}}
\newcommand{\mbf}[1]{{\bf #1}}
\newcommand{\mbm}[1]{\mathbbm{#1}}
\newcommand{\perm}{ {\rm Perm} }

\newcommand{\into}{\hookrightarrow}
\newcommand{\onto}{\twoheadrightarrow}

\newcommand{\s}{\mathbf{s}}
\newcommand{\cc}{\mathbf{c}}
\newcommand{\bfa}{\mathbf{a}}
\newcommand{\id}{{\rm id}}
\newcommand{\g}{\mathbf{g}_{\psi^{-1}}}
\newcommand{\w}{\mathbbm{w}}
\newcommand{\Ftn}{{\rm Ftn}}
\newcommand{\p}{\mathbf{p}}
\newcommand{\bq}{\mathbf{q}}
\newcommand{\WD}{\text{WD}}
\newcommand{\W}{\text{W}}
\newcommand{\Wh}{{\rm Wh}_\psi}
\newcommand{\ggma}{\pmb{\gamma}}
\newcommand{\sct}{\text{\rm sc}}
\newcommand{\OF}{\mca{O}^\digamma}
\newcommand{\vep}{ \varepsilon }
\newcommand{\gk}{c_{\sf gk}}
\newcommand{\Ima}{{\rm Im}}
\newcommand{\Ker}{{\rm Ker}}
\newcommand{\lrho}{ {}^L\rho }

\newcommand{\Tr}{ {\rm Tr} }

\newcommand{\fwchi}{{\chi^\flat}}

\newcommand{\cu}[1]{\textsc{\underline{#1}}}
\newcommand{\set}[1]{\left\{#1\right\}}
\newcommand{\ul}[1]{\underline{#1}}
\newcommand{\wt}[1]{\overline{#1}}
\newcommand{\wtsf}[1]{\wt{\sf #1}}
\newcommand{\angb}[2]{\left\langle #1, #2 \right\rangle}
\newcommand{\wm}[1]{\wt{\mbf{#1}}}
\newcommand{\elt}[1]{\pmb{\big[} #1\pmb{\big]} }
\newcommand{\ceil}[1]{\left\lceil #1 \right\rceil}
\newcommand{\val}[1]{\left| #1 \right|}

\newcommand{\uchi}{\underline{\chi}}

\newcommand{\Irr}{ {\rm Irr} }

\newcommand{\nequiv}{\not \equiv}
\newcommand{\half}{\frac{1}{2}}
\newcommand{\psii}{\widetilde{\psi}}
\newcommand{\ab} {|\!|}
\newcommand{\mb}{{\widetilde{B(\F)}}}

\title[R-group and Whittaker space]{R-group and Whittaker space of some genuine representations}

\author{Fan Gao}
\address{School of Mathematical Sciences, Yuquan Campus, Zhejiang University, 38 Zheda Road, Hangzhou, China 310027}
\email{gaofan.math@gmail.com}

\subjclass[2010]{Primary 11F70; Secondary 22E50}
\keywords{covering groups, R-groups, Whittaker functionals, scattering matrix, gamma factor, Plancherel measure}
\maketitle

\begin{abstract} 
For a unitary unramified genuine principal series representation of a covering group, we study the associated R-group. We prove a formula relating the R-group to the dimension of the Whittaker space for the irreducible constituents of such a principal series representation. Moreover, for saturated covers of a semisimple simply-connected group, we also propose a simpler conjectural formula for such dimensions. This latter conjectural formula is verified in several cases, including covers of the symplectic groups.
\end{abstract}
\tableofcontents

\section{Introduction}

In \cite{Ga6}, we proposed and proved partially a conjectural formula for the dimension of a certain relative Whittaker space of the irreducible constituents of a \emph{regular} unramified genuine principal series representation $I(\chi)$ of a covering group $\wt{G}$ over a non-archimedean field $F$. The formula in \cite{Ga6} relates certain Kazhdan-Lusztig representations of the Weyl group to the dimensions of such relative Whittaker spaces. 

This paper, as a companion to the above one, deals with the case where $I(\chi)$ is a \emph{unitary}  unramified genuine principal series,  i.e., $\chi$ is a unitary unramified genuine character of the center $Z(\wt{T})$ of the covering torus $\wt{T}\subset \wt{G}$. In this case, $I(\chi)$ is a semisimple $\wt{G}$-module, and has a decomposition
$$I(\chi)= \bigoplus_{\pi \in \Pi(\chi)} m_\pi \cdot \pi,$$
where $\pi$'s are the non-equivalent constituents of $I(\chi)$. These representations $\pi$ thus should constitute an $L$-packet $\Pi(\chi)$. The reducibility of $I(\chi)$ and the above decomposition is controlled by a certain Knapp--Stein R-group $R_\chi \subset W_\chi$, where $W_\chi \subset W$ is the stabilizer subgroup of $\chi$ inside the Weyl group $W$. In particular, there is a correspondence 
\begin{equation} \label{F:P-chi}
\text{Irr}(R_\chi) \longleftrightarrow   \Pi(\chi), \quad \sigma \leftrightarrow \pi_\sigma,
\end{equation}
between the irreducible representations of $R_\chi$ and elements in $\Pi(\chi)$ such that $m_{\pi_\sigma} = \dim (\sigma)$. Since $\chi$ is unramified, one can show that $R_\chi$ is abelian and therefore $m_{\pi_\sigma}=1$ for every $\sigma$.

It is well-known that genuine representations of covering groups could have high-dimensional Whittaker space (i.e., the space of Whittaker functionals), see the introductions of \cite{Ga6, GSS2} for brief literature reviews. In particular, $\dim \Wh(I(\chi))$ increases as the degree of covering increases. In view of the correspondence in \eqref{F:P-chi}, it is a natural question to ask:
\begin{enumerate}
\item[$\bullet$] How to determine $\dim \Wh(\pi_\sigma)$ in terms of $\sigma \in \Irr(R_\chi)$?
\end{enumerate}
Our goal is first to prove a formula for $\dim \Wh(\pi_\sigma)$ for general $\wt{G}$. Second, for saturated covers of a semisimple simply-connected $G$, we propose a simpler conjectural formula for $\dim \Wh(\pi_\sigma)$  in terms of the character pairing of $\sigma$ and a certain permutation representation $\sigma^\msc{X}$ of $R_\chi$. We will also verify several cases of this conjectural formula in this paper.

\subsection{Background and motivation}
We briefly recall some relevant works for linear algebraic groups which motivate our consideration in this paper. 

For a linear algebraic group $G$ and $\chi$ a character of $T$ (not necessarily unramified), the correspondence \eqref{F:P-chi} arises from the theory of Harish-Chandra and Knapp--Stein for the commuting algebra $\text{End}(I(\chi))$ of the principal series representation $I(\chi)$, especially from the algebra isomorphism
\begin{equation} \label{F:iso}
\C[R_\chi]  \simeq \text{End}(I(\chi)).
\end{equation}
In fact, for general parabolic induction (i.e., parabolically induced representation) for $G$, the theory of R-group was initiated in the work of Knapp--Stein for real groups \cite{KnSt2}. Based on Harish-Chandra's work \cite{Sil-B}, Silberger worked out the formulation for $p$-adic groups \cite{Sil1, Sil1C}; in particular, he showed that Harish-Chandra's commuting algebra theorem holds,  which then gives an analogue of the isomorphism \eqref{F:iso} above for general parabolic induction on linear algebraic groups. Here for simplicity in this introduction, we ignore the subtleties of the two-cocycle twisting in the algebra of the R-group for general parabolic inductions, see \cite{Art93} for details. For minimal parabolic induction of a Chevalley group, the group $R_\chi$ was computed explicitly by Keys \cite{Key1}; the study in the unramified case was furthered in \cite{Key2}. 

In the minimal parabolic case, the connection of the R-group  with the Langlands correspondence was elucidated in \cite{Key3, KS88}. For example, let $\phi_\chi$ be the L-parameter associated to the character $\chi$; Keys showed that the component group $\mca{S}_{\phi_\chi}$ is isomorphic to $R_\chi$, and thus elements inside the $L$-packet $\Pi(\chi)$ are also naturally parametrized by $\text{Irr}(\mca{S}_{\phi_\chi})$. 
Beyond the principal series case, it was conjectured by Arthur \cite{Art89} that the R-group associated with the parabolic induction $I(\sigma)$ from a discrete series $\sigma$ on the Levi subgroup is also isomorphic to the component group $\mca{S}_{\phi_\sigma}$, where $\phi_\sigma$ is the $L$-parameter associated to $I(\sigma)$. We refer to \cite{BZ05, Gol11, BG12} and the references therein for works in this direction.

Such a relation between $\Pi(\chi)$ and $\phi_\chi$ demonstrates a prototype of the general idea of the local Langlands parametrization for admissible representations of a local reductive group $G$. Let $W_F' = W_F \times \SL_2(\C)$ be the Weil-Deligne group of $F$, where $W_F$ is the local Weil group. Let ${}^L G$ be the $L$-group of $G$. For each parameter
$$\phi: W_F' \longrightarrow {}^LG,$$
the local conjecture of Langlands asserts that there is an $L$-packet $\Pi_\phi$ consisting of irreducible admissible representations of $G$ which satisfy certain desiderata, see \cite{Bor}. Members inside the same packet $\Pi_\phi$ are equipped with the same $L$-function and $\varepsilon$-factor. Moreover, as alluded to above, if $\phi$ is a tempered parameter (i.e., the image of $\phi|_{W_F}$ in the dual group of $G$ is relatively compact), then  the component group $\mca{S}_\phi$ conjecturally parametrizes elements in the $L$-packet $\Pi_\phi$ (see \cite{Art06}), i.e., there is a bijection
\begin{equation} \label{F:par}
\text{Irr}(\mca{S}_\phi) \longleftrightarrow  \Pi_\phi.
\end{equation}
The bijection in \eqref{F:par} originally manifests in the theory of endoscopic transfer, and in particular in the character identity relations matching orbital integrals arising from the transfers (cf. \cite{LL79, She79, She82, Art84, Art93}).  More generally, in order to deal with non-tempered representation in the global $L^2$ spectral decomposition, it was postulated by Arthur that one should consider a  parameter $\varphi$ of $W_F' \times \SL_2(\C)$ valued in the Langlands $L$-group ${}^LG$; the component group $\mca{S}_{\varphi}$ should parametrize certain Arthur-packet which plays a crucial role in the work of  Arthur in formulating his global multiplicity formula for the discrete spectrum of automorphic representations of a linear reductive group (see \cite{Art89, Art16}).

It should be noted that the bijection in \eqref{F:par} is not canonical, which from the geometric side depends on the normalization of the Langlands-Shelstad transfer factors, and from the representation-theoretic side on the normalization of intertwining operators (see \cite{KS88, Sha83, Sha3}). Indeed, one can always twist the bijection by a character of the component group $\mca{S}_\phi$. However, if $\phi$ is a tempered parameter, then it was conjectured in \cite{Sha3} that $\Pi_\phi$ always contains a unique generic element with respect to a fixed Whittaker datum of $G$. In particular, for a fixed Whittaker datum of $G$, there is a one-to-one correspondence between generic tempered representations and tempered $L$-packets. Granted with this tempered-packet conjecture, the correspondence  in \eqref{F:par} can then be normalized such that the unique generic element (with respect to the fixed Whittaker datum) is parametrized by $\mbm{1} \in \text{Irr}(\mca{S}_\phi)$. We refer the reader to \cite{Sha5} and the references therein for works in this direction.

It is to this last problem on genericty of representations inside $\Pi_\phi$, where $\phi=\phi_\chi$ is the parameter for a unitary unramified genuine principal series $I(\chi)$ of a covering group, that our object in this paper pertains. Fix a Whittaker datum $(\wt{B}= \wt{T}U, \psi)$ for a covering group $\wt{G}$. The Whittaker space $\Wh(I(\chi))$ is understood as a consequence of Rodier's heridity and the fact that $\wt{T}$ has only trivial unipotent subgroup. However, if $I(\chi)$ is reducible, then for any constituent $\pi$, it is a natural but delicate question to determine the dimension $\dim \Wh(\pi)$. 

For a unitary unramified genuine $\chi$ the correspondence \eqref{F:P-chi}  continues to hold. The proof is essentially the same as in the linear algebraic case, and relies on the covering analogue of \eqref{F:iso}, which follows from the recent work of W.-W. Li \cite{Li3, Li5} and C.-H. Luo \cite{Luo3}. In fact, we will also show the isomorphism $R_\chi \simeq \mca{S}_{\phi_\chi}$. However, since it is possible to have $\dim \Wh(\pi) > 0$ for every constituent $\pi$ of $I(\chi)$, there does not seem to be a preferred choice of $\pi \in \Pi(\chi)$ by using the genericity criterion. Thus, we choose the normalization for the correspondence in \eqref{F:P-chi} such that the unique unramified constituent $\pi_\chi^{un} \subset I(\chi)$ corresponds to  the trivial representation $\mbm{1}$ of $R_\chi$, i.e.,
$$\pi_{\mbm{1}} = \pi_\chi^{un}.$$
For unramified principal series of a linear algebraic group, this is the natural choice since $\pi_\chi^{un}$ is generic with respect to the fixed Whittaker datum, as a consequence of the Casselman--Shalika formula.

With this normalized correspondence $\sigma \leftrightarrow \pi_\sigma$, we want to determine $\dim \Wh(\pi_\sigma)$ in terms of $\sigma$ for every $\sigma \in \Irr(R_{\chi})$. We hope that the results in this paper will eventually find applications in the context of global automorphic representations for covering groups.

\subsection{Main conjecture}
Let $\chi$ be a unitary unramified genuine character of $Z(\wt{T})$. Associated to $\wt{T}$, there is a finite abelian group 
$$\msc{X}_{Q,n}:= Y/Y_{Q,n},$$
which is the quotient of the cocharacter lattice $Y$ of $G$ by a certain sublattice $Y_{Q,n}$. Here $\msc{X}_{Q,n}$ is the ``moduli space" of $\Wh(I(\chi))$; in particular,
$$\dim \Wh(I(\chi))= \val{\msc{X}_{Q,n}}.$$
The group $\msc{X}_{Q,n}$  is endowed with a natural twisted $W$-action which we denote by $\w[y]$. From this one has a permutation representation
$$\sigma^\msc{X}:  W \longrightarrow  \perm(\msc{X}_{Q,n})$$
given by $\sigma^\msc{X}(w)(y)= \w[y]$. Let $\mca{O}_\msc{X}$ be the set of all $W$-orbits (with respect to this twisted action) in $\msc{X}_{Q,n}$. Clearly, for each $W$-orbit $\mca{O}_y  \in \mca{O}_\msc{X} $, one has the permutation representation
$$\sigma^\msc{X}_{\mca{O}_y}:  W \longrightarrow  \perm(\mca{O}_y);$$
moreover, $\sigma^\msc{X}$ decomposes as a sum of all  $\sigma_{\mca{O}_y}^\msc{X}$, i.e.,
$$\sigma^\msc{X} = \bigoplus_{  \mca{O}_y \in \mca{O}_\msc{X} }  \sigma_{\mca{O}_y}^\msc{X}.$$
By restriction, $\sigma^\msc{X}_{\mca{O}_y}$ could be viewed as a permutation representation of $R_\chi \subset W_\chi \subset W$. 

For every $W$-orbit $\mca{O}_y \subset \msc{X}_{Q,n}$, there is also a natural subspace $\Wh(\pi_\sigma)_{\mca{O}_y} \subset \Wh(\pi_\sigma)$ (see \eqref{Oy-Wh}) such that 
$$\Wh(\pi_\sigma) = \bigoplus_{ \mca{O}_y \in \mca{O}_\msc{X}  }\Wh(\pi_\sigma)_{\mca{O}_y}.$$

\begin{conj}[{Conjecture \ref{MConj}}] \label{C:0}
Let $\wt{G}$ be a saturated $n$-fold cover (see Definition \ref{D:meta}) of a semisimple simply-connected group $G$. In the normalized correspondence ${\rm Irr}(R_\chi) \longleftrightarrow \Pi(\chi), \sigma \leftrightarrow \pi_\sigma$ such that $\pi_\mbm{1}= \pi_\chi^{un}$, one has
$$\dim \Wh(\pi_\sigma)_{\mca{O}_y} =\angb{\sigma}{ \sigma^\msc{X}_{\mca{O}_y}  }_{R_\chi}$$
for every orbit $\mca{O}_y \in \mca{O}_\msc{X}$, where $\angb{-}{-}_{R_\chi}$ denotes the pairing of two representations of $R_\chi$. Consequently,
$$\dim \Wh(\pi_\sigma) =\angb{\sigma}{ \sigma^\msc{X} }_{R_\chi}$$
for every $\sigma \in \Irr(R_\chi)$; in particular, $\dim \Wh(\pi_\chi^{un})$ is equal to the number of $R_\chi$-orbits in $\msc{X}_{Q,n}$.
\end{conj}

\subsection{Main results}
Prior to the formulation of the above conjecture, a substantial part of the paper is devoted to analyzing the group $R_\chi$ and proving for covers of general reductive groups an unconditional formula for $\dim \Wh(\pi_\sigma)_{\mca{O}_y}$ in terms of $\sigma$ and a certain representation $\sigma^{\rm Wh}_{\mca{O}_y}$ of $R_\chi$. We briefly outline the content of the paper and highlight some of our results. 

After a brief introduction on covering groups in \S \ref{S:cg}, we study in \S \ref{S:red-ps} the normlized intertwining operator between genuine principal series of $\wt{G}$. As in the case for linear algebraic groups, the normalization is given by the Langlands $L$-functions, and one important property is the cocycle relation of the normalized intertwining operators which does not depend on the length function of $W$.

In \S \ref{S:R-grp}, we analyze the group $R_\chi$ based on the work of Keys \cite{Key1, Key3}, W.-W. Li \cite{Li3, Li5} and C.-H. Luo \cite{Luo3}. In particular, it follows from \cite{Luo3} that for a unitary unramified genuine principal series $I(\chi)$, one has an algebra isomorphism $\C[R_\chi] \simeq \text{End}(I(\chi))$.
We show how to compute $R_\chi$ by relating it to another group $R_\chi^{sc}$, which is equal to the R-group of a certain unramified principal series of a simply-connected Chevalley group $H^{sc}$. The group $R_\chi^{sc}$ is explicitly determined in Keys' paper \cite{Key1, Key2} for principal series of simply-connected Chevalley groups. Moreover, by reducing to the linear algebraic case, we prove in Theorem \ref{T:R=S} the isomorphism $R_\chi \simeq \mca{S}_{\phi_\chi}$, where $\phi_\chi$ is the $L$-parameter of $I(\chi)$ valued in the $L$-group of $\wt{G}$ constructed by M. Weissman \cite{We6}. 

Denote by $\wt{G}^\vee$ (resp. ${}^L\wt{G}$) the dual group (resp. $L$-group) for the covering group $\wt{G}$. The following is an amalgam of Proposition \ref{P:bd}, Theorem \ref{T:R-abel} and Theorem \ref{T:R=S}.

\begin{thm}
Let $\wt{G}$ be an $n$-fold cover of a linear algebraic group $G$. Let $\chi$ be a unitary unramified genuine character of $Z(\wt{T})$. We have:
\begin{enumerate}
\item[$\bullet$] $R_\chi \subseteq R_\chi^{sc}$ with $R_\chi^{sc}$ being an abelian group, and if $\wt{G}$ is semisimple, then $[R_\chi^{sc}: R_\chi] \le \val{Z(\wt{G}^\vee)}$, where $Z(\wt{G}^\vee)$ is the center of $\wt{G}^\vee$;
\item[$\bullet$] $R_\chi \simeq \mca{S}_{\phi_\chi}$, where $\mca{S}_{\phi_\chi}$ is the component group of the parameter $\phi_\chi: W_F \to {}^L \wt{G}$.
\end{enumerate}
\end{thm}
We remark that since $\chi$ is unitary, the parameter $\phi_\chi$ is trivial on $\SL_2(\C) \subset W_F'$, and thus it suffices to consider the Weil group $W_F$ alone.

Section \S \ref{S:conj} is devoted to stating and investigating several aspects of the main Conjecture \ref{MConj}, i.e., Conjecture \ref{C:0} above. First, we show that for every $\mca{O}_y \in \mca{O}_\msc{X}$, the space $\Wh(I(\chi))_{\mca{O}_y}$ affords a natural representation
$$\sigma_{\mca{O}_y}^{\rm Wh}:  R_\chi \longrightarrow \GL(\C^{ \val{ \mca{O}_y } })$$ such that the following holds (cf. Theorem \ref{T:dW}):

\begin{thm} \label{T:unc}
Let $\wt{G}$ be an $n$-fold cover of a connected reductive group $G$. For every orbit $\mca{O}_y \in \mca{O}_\msc{X}$, one has
$$\dim \Wh(\pi_\sigma)_{\mca{O}_y} = \angb{\sigma}{ \sigma_{\mca{O}_y}^{\rm Wh} }_{R_\chi}.$$
Consequently, $\dim \Wh(\pi_\sigma) = \angb{\sigma}{ \sigma^{\rm Wh} }_{R_\chi}$.
\end{thm}
Here $\sigma_{ \mca{O}_y }^{\rm Wh}(\w)$ is represented by the matrix $\gamma(w, \chi) \cdot \mca{S}_{\mfr{R}}(w, i(\chi))$, where $\gamma(w, \chi)$ is the gamma-factor associated to $\w$ and $\mca{S}_{\mfr{R}}(w, i(\chi))$ a so-called scattering matrix. As an application of the above theorem, we show in \S \ref{SS:GSp} that a result of D. Szpruch \cite{Szp5} on double cover of $\GSp_{2r}$ could be recovered from it, see Theorem \ref{T:GSp}. Here Theorem \ref{T:unc} also implies that Conjecture \ref{C:0} is equivalent to the following (cf. Conjecture \ref{MConj2}):

\begin{conj}   \label{C:1}
Let $\wt{G}$ be a saturated $n$-fold cover of a  semisimple simply-connected $G$. Then for every orbit $\mca{O}_y \subset \msc{X}_{Q,n}$, one has $\sigma_{\mca{O}_y}^{\rm Wh} = \sigma_{\mca{O}_y}^\msc{X}$; or equivalently,
$$\Tr\big( \mca{S}_{\mfr{R}}(w, i(\chi))  \big) = \val{ (\mca{O}_y)^\w } \cdot \gamma(w, \chi)^{-1}$$
for every $\w \in R_\chi$.
\end{conj}

Using the formulation in Conjecture \ref{C:1},  we prove several results in  \S \ref{S:Wh-un}:
\begin{enumerate}
\item[$\bullet$] For a general reductive group $G$, we show that there is an exceptional set $\msc{X}_{Q,n}^{\rm exc} \subset (\msc{X}_{Q,n})^W$, which might be empty, such that $\sigma_{\mca{O}_y}^{\rm Wh} = \sigma_{\mca{O}_y}^\msc{X} = \mbm{1}_{R_\chi}$ for $y \in \msc{X}_{Q,n}^{\rm exc}$. It follows that $\dim \Wh(\pi_\chi^{un}) \ge \val{ \msc{X}_{Q,n}^{\rm exc}  }$; it also implies that Conjecture \ref{C:1} holds for such $\mca{O}_y$. This is the content of Theorem \ref{T:unWh}.
\item[$\bullet$] In \S \ref{S:UWf}, we consider the Whittaker space $\Wh (\pi_\chi^{un})$ from the perspective of unramified Whittaker functions. Using an analogue of Casselman--Shalika formula proved in \cite{GSS2}, we show in Theorem \ref{T:un-gen} a result on $\dim \Wh(\pi_\chi^{un})$, which is  compatible with Theorem \ref{T:unWh}.
\end{enumerate}

In \S \ref{S:eg}, we verify the following:
\begin{thm}[{Theorem \ref{T:Sp2r}}]
 Conjecture \ref{C:1} (and thus Conjecture \ref{C:0}) holds for the $n$-fold covers $\wt{\Sp}_{2r}^{(n)}$.
\end{thm}
We also prove in \S \ref{S:eg} that Conjecture \ref{C:0} holds for the double cover of $\SL_3$ by explicit computations. 

Lastly, in \S \ref{S:2rmk}, we consider $n$-fold covers of $\SO_3$ and the double cover of $\Spin_6 \simeq \SL_4$, and show that the naive analogue of Conjecture \ref{C:1} fails for such covers. Thus, the constraints on $G$ being simply-connected and  on $\wt{G}$ being saturated seem to be indispensable. For general reductive groups, a unified conjectural formula for $\Wh(\pi_\sigma)_{\mca{O}_y}$ in terms of $\sigma^\msc{X}_{\mca{O}_y}$ involves subtleties beyond the consideration in this paper, and its intimate relation with the R-group has yet to be unveiled in full generality.

\section{Covering group} \label{S:cg}
 
 Our exposition on covering group is essentially the same as in \cite[\S 2]{Ga6}. However, to ensure that the paper is self-contained, we will briefly recall again some summarized results on $\wt{G}$.
 
Let $F$ be a finite extension of $\Q_p$. Denote by $O \subset F$ the ring of integers of $F$ and $\varpi \in O$ a fixed uniformizer. 

\subsection{Covering group} \label{Sec:SF}
Let $\mbf{G}$ be a split connected linear algebraic group over $F$ with a maximal split torus $\mbf{T}$. Let
$$\set{ X,\ \Delta,  \ \Phi; \ Y, \ \Delta^\vee, \ \Phi^\vee  }$$
be the based root datum of $\mbf{G}$. Here $X$ (resp. $Y$) is the character lattice (resp. cocharacter lattice) for $(\mbf{G}, \mbf{T})$. Choose a set $\Delta\subseteq \Phi$ of simple roots from the set of roots $\Phi$, and let $\Delta^\vee$ be the corresponding simple coroots from $\Phi^\vee$. This gives us a choice of positive roots $\Phi_+$ and positive coroots $\Phi_+^\vee$. Write $Y^{\sct}\subseteq Y$ for the sublattice generated by $\Phi^\vee$. Let $\mbf{B} =\mbf{T} \mbf{U}$ be the Borel subgroup associated with $\Delta$. Denote by $\mbf{U}^- \subset \mbf{G}$ the unipotent subgroup opposite $\mbf{U}$.

Fix a Chevalley-Steinberg system of pinnings for $(\mbf{G}, \mbf{T})$. That is, we fix a set of compatible isomorphisms
$$\set{e_\alpha: \mbf{G}_\text{a} \to \mbf{U}_\alpha}_{\alpha\in \Phi},$$ 
where $\mbf{U}_\alpha \subseteq \mbf{G}$ is the root subgroup  associated with $\alpha$. In particular, for each $\alpha\in \Phi$, there is a unique homomorphism $\varphi_\alpha: \mbf{SL}_2 \to \mbf{G}$ which restricts to $e_{\pm \alpha}$ on the upper and lower triangular subgroup of unipotent matrices of $\mbf{SL}_2$, respectively.

Denote by $W$ the Weyl group of $(\mbf{G}, \mbf{T})$, which we identify with the Weyl group of the coroot system. In particular, $W$ is generated by simple reflections $\set{\w_\alpha: \alpha^\vee \in \Delta^\vee}$ in $Y \otimes \Q$. Let $l: W \to \N$ be the length function. Let $\w_G$ be the longest element in $W$.

Consider the algebro-geometric $\mbf{K}_2$-extension $\wm{G}$ of $\mbf{G}$ studied by Brylinski and Deligne \cite{BD}, which is categorically equivalent to the pairs $\set{(D, \eta)}$ (see \cite[\S 2.6]{GG}). Here $\eta: Y^{\sct} \to F^\times$ is a homomorphism. On the other hand, 
$$D: Y \times Y \to \Z$$ 
is a (not necessarily symmetric) bilinear form on $Y$ such that 
$$Q(y):=D(y, y)$$
is a Weyl-invariant integer-valued quadratic form on $Y$. We call $D$ a bisector. Let $B_Q$ be the Weyl-invariant bilinear form associated to $Q$ by
$$B_Q(y_1, y_2)=Q(y_1+y_2)-Q(y_1) -Q(y_2).$$
Clearly, $D(y_1, y_2) + D(y_2, y_1)=B_Q(y_1, y_2)$.  Any $\wm{G}$ is, up to isomorphism, incarnated by (i.e., categorically associated to) a pair $(D,\eta)$ for a bisector $D$ and $\eta$.

The couple $(D, \eta)$ play the following role for the structure of $\wm{G}$.
\begin{enumerate}
\item[$\bullet$] First, the group $\wm{G}$ splits canonically and uniquely over any unipotent subgroup of $\mbf{G}$. For $\alpha\in \Phi$ and $a\in \mbf{G}_a$, denote by $\wt{e}_\alpha(a) \in \wm{G}$  the canonical lifting of $e_\alpha(a) \in \mbf{G}$. For $\alpha\in \Phi$ and $a\in \mbf{G}_m$, define 
\begin{equation*}  \label{F:w}
w_\alpha(a):=e_{\alpha}(a) \cdot e_{-\alpha}(-a^{-1}) \cdot e_{\alpha}(a) \text{ and } \wt{w}_\alpha(a):=\wt{e}_{\alpha}(a) \cdot \wt{e}_{-\alpha}(-a^{-1}) \cdot \wt{e}_{\alpha}(a).
\end{equation*}
This gives natural representatives $w_\alpha:=w_\alpha(1)$ in $\mbf{G}$ and also $\wt{w}_\alpha:=\wt{w}_\alpha(1)$ in $\wm{G}$, of the Weyl element $\w_\alpha\in W$. Moreover, for $h_\alpha(a):=\alpha^\vee(a)\in \mbf{T}$, there is a natural lifting 
\begin{equation*}  \label{h-alpha}
\wt{h}_\alpha(a):=\wt{w}_\alpha(a)\cdot \wt{w}_\alpha(-1) \in \wm{T},
\end{equation*}
which depends only on the pinnings and the canonical unipotent splittings.
\item[$\bullet$] Second, there is a section $\s$ of $\wm{T}$ over $\mbf{T}$ such that the group law on $\wm{T}$ includes the relation
\begin{equation} \label{F:s}
\s(y_1(a)) \cdot \s(y_2(b)) = \set{a, b}^{D(y_1, y_2)} \cdot \s(y_1(a)\cdot y_2(b))
\end{equation}
for any $a, b\in \mbf{G}_m$. Moreover, for $\alpha\in \Delta$ and the natural lifting $\wt{h}_\alpha(a)$ of $h_\alpha(a)$ above, one has
\begin{equation*} 
\wt{h}_\alpha(a)=\set{\eta(\alpha^\vee), a} \cdot \s(h_\alpha(a)) \in \wm{T}.
\end{equation*}
\item[$\bullet$] Third, let $w_\alpha \in \mbf{G}$ be the above natural representative of $\w_\alpha\in W$. For every $\wt{y(a)} \in \wm{T}$ with $y\in Y$ and $a\in \mbf{G}_m$, one has
\begin{equation} \label{F:W-act}
w_\alpha \cdot \wt{y(a)} \cdot w_\alpha^{-1} = \wt{y(a)} \cdot \wt{h}_\alpha(a^{-\angb{y}{\alpha}}),
\end{equation}
where $\angb{-}{-}$ is the paring between $Y$ and $X$.
\end{enumerate}

If the derived subgroup of $\mbf{G}$ is simply-connected, then the isomorphism class of $\wm{G}$ is determined by the Weyl-invariant quadratic form $Q$. In particular, for such $\mbf{G}$, any extension $\wm{G}$ is incarnated by $(D, \eta=\mathbbm{1})$ for some bisector $D$, up to isomorphism.
In this paper, we assume that the composite
$$\eta_n: Y^{sc} \to F^\times \onto F^\times/(F^\times)^n$$ 
of $\eta$ with the obvious quotient is trivial. For some consequence of this assumption, see \S \ref{S:D-L} and the beginning of \S \ref{S:red-ps}.

Let $n\in \N$. We assume that $F$ contains the full group of $n$-th roots of unity, denoted by $\bbmu_n$. An $n$-fold cover of $\mbf{G}$, in the sense of \cite[Definition 1.2]{We6}, is just a pair $(n, \wm{G})$.  The $\mbf{K}_2$-extension $\wm{G}$ gives rise to an $n$-fold covering $\wt{G}$ as follows. Let
$$(-,-)_n:  F \times F \to \bbmu_n$$
be the $n$-th Hilbert symbol. The cover $\wt{G}$ arises from the central extension
 $$\begin{tikzcd}
\mbf{K}_2(F) \ar[r, hook] & \wm{G}(F) \ar[r, two heads, "\phi"] & \mbf{G}(F)
\end{tikzcd}$$
by push-out via the natural map $\mbf{K}_2(F) \to \bbmu_n$ given by $\set{a, b} \mapsto (a, b)_n$. This gives
$$\begin{tikzcd}
\bbmu_n \ar[r, hook] & \wt{G} \ar[r, two heads, "\phi"] & G.
\end{tikzcd}$$
We may write $\wt{G}^{(n)}$ for $\wt{G}$ to emphasize the degree of covering. 

For any subset $H \subset G$, denote $\wt{H}:=\phi^{-1}(H)$. The relations for $\wm{G}$ described above give rise to the corresponding relations for $\wt{G}$.  For example, inherited from \eqref{F:s} is the following relation on $\wt{T}$,
\begin{equation}
\s(y_1(a)) \cdot \s(y_2(b)) = (a, b)_n^{D(y_1, y_2)} \cdot \s(y_1(a)\cdot y_2(b)),
\end{equation}
where $y_i\in Y$ and $a, b\in F^\times$. The commutator $[\wt{t}_1, \wt{t}_2]:=\wt{t}_1 \wt{t}_2 \wt{t}_1^{-1} \wt{t}_2^{-1}$ of $\wt{T}$, which descends to a map $[-,-]: T \times T \to \bbmu_n$, is thus given by
$$[y_1(a), y_2(b)]=(a, b)_n^{B_Q(y_1, y_2)}.$$

A representation of $\wt{G}$ is called $\epsilon$-genuine (or simply genuine) if $\bbmu_n$ acts by a fixed embedding $\epsilon: \bbmu_n \into \C^\times$. We consider only genuine representations of a covering group in this paper.

Let $W' \subset \wt{G}$ be the group generated by $\wt{w}_\alpha$ for all $\alpha$. Then the map $\wt{w}_\alpha \mapsto \w_\alpha$ gives a surjective homomorphism
$$W' \onto W$$
with kernel being a finite group. For any $\w=\w_{\alpha_k} ... \w_{\alpha_2} \w_{\alpha_1} \in W$ in a minimal decomposition, we let
$$\wt{w}:=\wt{w}_{\alpha_k} ... \wt{w}_{\alpha_2} \wt{w}_{\alpha_1} \in W'$$
be its representative in $W'$, which is independent of the minimal decomposition (see \cite[Lemma 83 (b)]{Ste16}). In particular, we denote by $\wt{w}_G \in \wt{G}$ the above representative of the longest Weyl element $\w_G$. Note that we also have the natural representative 
$$w:=w_{\alpha_k} ... w_{\alpha_2} w_{\alpha_1} \in G$$
of $\w$. In particular, one has the representative $w_G \in G$ for $\w_G$, which is the image of $\wt{w}_G$ in $G$.

\subsection{Dual group and $L$-group} \label{S:D-L}
For a cover $(n, \wm{G})$ associated to $(D, \eta)$, with $Q$ and $B_Q$ arising from $D$, we define
\begin{equation} \label{YQn}
Y_{Q,n}:= Y\cap nY^*,
\end{equation}
where $Y^* \subset Y\otimes \Q$ is the dual lattice of $Y$ with respect to $B_Q$; more explicitly,
$$Y_{Q,n}= \set{y\in Y: B_Q(y, y')\in n\Z \text{ for all } y'\in Y} \subset Y.$$
For every $\alpha^\vee\in \Phi^\vee$, denote
$$n_\alpha:= \frac{n}{\text{gcd}(n, Q(\alpha^\vee))}$$
and
$$ \alpha_{Q,n}^\vee=n_\alpha \alpha^\vee, \quad \alpha_{Q,n}=\frac{\alpha}{n_\alpha} .$$
Let 
$$Y_{Q,n}^{sc} \subset Y_{Q,n}$$ 
be the sublattice generated by $\Phi_{Q,n}^\vee=\{\alpha_{Q,n}^\vee: \alpha^\vee \in \Phi^\vee \}$.  Denote $X_{Q,n}=\text{Hom}_\Z(Y_{Q,n}, \Z)$ and $\Phi_{Q,n}=\set{\alpha_{Q,n}: \alpha \in \Phi }$. We also write 
$$\Delta_{Q,n}^\vee=\{ \alpha_{Q,n}^\vee: \alpha^\vee \in \Delta^\vee \} \text{ and } \Delta_{Q,n}=\set{\alpha_{Q,n}: \alpha\in \Delta}.$$
Then
$$\big( Y_{Q,n}, \ \Phi_{Q,n}^\vee, \ \Delta_{Q,n}^\vee;\  X_{Q,n},\  \Phi_{Q,n}^\vee, \Delta_{Q,n} \big)$$
forms a root datum with a choice of simple roots $\Delta_{Q,n}$. It gives a unique (up to unique isomorphism) pinned reductive group $\wm{G}_{Q,n}^\vee$ over $\Z$, called the dual group of $(n, \wm{G})$. In particular, $Y_{Q,n}$ is the character lattice for $\wt{G}_{Q,n}^\vee$ and $\Delta_{Q,n}^\vee$ the set of simple roots. Let 
$$\wt{G}_{Q,n}^\vee:=\wm{G}_{Q,n}^\vee(\C)$$
be the associated complex dual group. For simplicity, we may also write $\wt{G}^\vee$ for $\wt{G}_{Q,n}^\vee$. One has
$$Z(\wt{G}_{Q,n}^\vee) = \Hom(Y/Y_{Q,n}, \C).$$

In \cite{We3, We6}, Weissman constructed the global $L$-group as well as the  local $L$-group extension
 $$\begin{tikzcd}
\wt{G}^\vee_{Q,n} \ar[r, hook] & {}^L\wt{G} \ar[r, two heads] & W_F,
\end{tikzcd}$$
which is compatible with the global $L$-group extension. (It may as well be an extension over the Weil-Deligne group. However, the Weil group $W_F$ suffices in this paper, since we eventually only consider unitary principal series.) His construction of $L$-group is functorial, and in particular it behaves well with respect to the restriction of $\wm{G}$ to parabolic subgroups. More precisely, let $\mbf{M} \subset \mbf{G}$ be a Levi subgroup. By restriction, one has the $n$-cover $\wt{M}$ of $M$. The $L$-groups ${}^L\wt{M}$ and ${}^L\wt{G}$ are compatible, i.e., there are natural homomorphisms of extensions:
$$\begin{tikzcd}
\wt{G}^\vee_{Q,n} \ar[r, hook] & {}^L\wt{G} \ar[r, two heads] & W_F \\
\wt{M}^\vee_{Q,n} \ar[r, hook] \ar[u, hook] & {}^L\wt{M} \ar[u, hook] \ar[r, two heads] & W_F  \ar[u, equal] .
\end{tikzcd}$$
This applies in particular to the case when $M=T$ is a torus.

The extension ${}^L\wt{G}$ does not split over $W_F$. However, if $\wt{G}_{Q,n}^\vee$ is of adjoint type, then we have a canonical isomorphism
\begin{equation*}
{}^L\wt{G} \simeq \wt{G}_{Q,n}^\vee \times W_F.
\end{equation*}
For general $\wt{G}$, under the assumption that $\eta_n=\mbm{1}$, there exists a so-called distinguished genuine character $\chi_\psi: Z(\wt{T}) \to \C^\times$ (see \cite[\S 6.4]{GG}), depending on a nontrivial additive character $\psi$ of $F$, such that $\chi_\psi$ gives rise to a splitting of ${}^L\wt{G}$ over $W_F$, with respect to which one has an isomorphism 
\begin{equation} \label{L-iso}
{}^L\wt{G} \simeq_{\chi_\psi} \wt{G}_{Q,n}^\vee \times W_F.
\end{equation}
For details on the construction and properties regarding the $L$-group, we refer the reader to \cite{We3, We6, GG}.

\subsection{Twisted Weyl action} \label{SS:tww}
For a group $H$ acting on a set $S$, we denote by 
$$\mca{O}^H_S$$
the set of all $H$-orbits in $S$. For every $z\in S$, denote by $\mca{O}^H_z  \in \mca{O}^H_S$ the $H$-orbit of $z$.
 
Denote by $\w(y)$ the natural Weyl group action on $Y$ and $Y\otimes \Q$ generated by the reflections $\w_\alpha$. The two lattices $Y_{Q,n}$ and $Y_{Q,n}^{sc}$ are both $W$-stable under this usual action.  Let $$\rho:= \frac{1}{2} \sum_{\alpha^\vee >0} \alpha^\vee$$
be the half sum of all positive coroots of $\mbf{G}$.  We consider the twisted Weyl-action
$$\w[y]:=\w(y-\rho)+ \rho.$$
It induces a well-defined twisted action of $W$ on 
$$\msc{X}_{Q,n}:=Y/Y_{Q,n}$$ given by
$\w[y + Y_{Q,n}]:=\w[y] + Y_{Q,n}$, since $W(Y_{Q,n}) = Y_{Q,n}$ as mentioned above. Thus, we have a permutation representation
$$\sigma^\msc{X}:  W \longrightarrow \perm(\msc{X}_{Q,n}),$$
which plays a pivotal role in the conjectural formulas on Whittaker space, in both \cite{Ga6} and this paper.

We note that the twisted Weyl-action on $Y/Y_{Q,n}^{sc}$ is also well-defined. For every $\alpha \in \Delta$, let $W_\alpha =\set{1, \w_\alpha} \subset W$. Arising from the surjection 
$$Y/Y_{Q,n}^{sc} \onto \msc{X}_{Q,n},$$
one has a map of sets
$$\phi_\alpha: (Y/Y_{Q,n}^{sc})^{W_\alpha} \onto (\msc{X}_{Q,n})^{W_\alpha}.$$
Recall the following definition from \cite{GSS2, Ga6}.

\begin{dfn}  \label{D:meta}
A covering group $\wt{G}$ of a connected linear reductive group $G$ is called saturated if $Y_{Q,n}^{sc} = Y_{Q,n}\cap Y^{sc}$. It is called of metaplectic type if there exists $\alpha\in \Delta$ such that $\phi_\alpha$ is not surjective.
\end{dfn}
If $G$ is semisimple simply-connected, then $\wt{G}$ is saturated if and only if $\wt{G}^\vee$ is of adjoint type.
On the other hand, covering groups of metaplectic type are rare. Indeed, it follows from \cite[\S 4.5]{GSS2} that if $G$ is almost simple, then $\wt{G}$ is of metaplectic type if and only if $\mbf{G}=\Sp_{2r}$ and $n_\alpha \equiv 2 \ (\text{mod}\ 4)$ for the unique short simple coroot $\alpha^\vee$ of $\Sp_{2r}$.  In particular, the classical double cover of $\Sp_{2r}$ is such an example.

Throughout the paper, we denote
$$y_\rho:=y-\rho \in Y\otimes \Q$$ 
for $y\in Y$. Clearly, 
$$\w[y]-y=\w(y_\rho) - y_\rho.$$
By Weyl action or Weyl orbits in $Y$ or $Y\otimes \Q$  we always refer to the ones with respect to the action $\w[y]$, unless specified otherwise.  For simplicity, we will abuse notation and denote by $y$ for an element in $\msc{X}_{Q,n}$. We will also write $\mca{O}_{\msc{X}}$ for the set of $W$-orbits  in $\msc{X}_{Q,n}$, and use the notation $\mca{O}_z :=\mca{O}^W_z$, whenever we consider $W$-orbits with respect to the twisted action.

\section{Unitary unramified principal series} \label{S:red-ps}
Henceforth, we assume that $|n|_F=1$. Let $K \subset G$ be the hyperspecial maximal compact subgroup generated by $\mbf{T}(O)$ and $e_\alpha(O)$ for all roots $\alpha$. With our assumption that $\eta_n$ is trivial, the group $\wt{G}$ splits over $K$ (see \cite[Theorem 4.2]{GG}); we fix such a splitting $s_K$. If no confusion arises, we will omit $s_K$ and write $K \subset \wt{G}$ instead. Call $\wt{G}$ an unramified covering group in this setting. 

A genuine representation $(\pi, V_\pi)$ is called unramified if $\dim V_\pi^{K}\ne 0$. Since $\wt{G}$ also splits uniquely over the unipotent subgroup $e_\alpha(O)$, we see that $\wt{h}_\alpha(u) \in s_K(K) \subset \wt{G}$ for every $u\in O^\times$.

\subsection{Principal series representation} \label{SS:ps}
As $\wt{G}$ splits canonically over the unipotent radical $U$ of the Borel subgroup $B$, we have $\wt{B}=\wt{T}U$. The covering torus $\wt{T}$ is a Heisenberg group. The center $Z(\wt{T})$ of the covering torus $\wt{T}$ is equal to $\phi^{-1}(\text{Im}(i_{Q,n}))$ (see \cite{We1}), where 
$$i_{Q,n}: Y_{Q,n}\otimes F^\times \to T$$
is the isogeny induced from the embedding $Y_{Q,n} \subset Y$.

Let $\chi \in \Hom_\epsilon(Z(\wt{T}), \C^\times)$ be a genuine character of $Z(\wt{T})$, write 
$$i(\chi):=\text{Ind}_{\wt{A}}^{\wt{T}} \ \chi'$$
 for the induced representation of $\wt{T}$, where $\wt{A}$ is a maximal abelian subgroup of $\wt{T}$, and $\chi'$ is an extension of $\chi$ to $\wt{A}$. By the Stone--von Neumann theorem (see \cite[Theorem 3.1]{We1}), the construction 
 $$\chi \mapsto i(\chi)$$
 gives a bijection between the isomorphism classes of genuine representations of $Z(\wt{T})$ and $\wt{T}$. Since we consider unramified covering group $\wt{G}$ in this paper, we take 
 $$\wt{A}:=Z(\wt{T})\cdot (K\cap T).$$
 
The left action of $w$ on $i(\chi)$ is given by ${}^w i(\chi) (\wt{t})= i(\chi)(w^{-1} \wt{t} w)$. The group $W$ does not act on $i(\chi)$, but only on its isomorphism classes. On the other hand, we have a well-defined action of $W$ on $\chi$:
$$({}^w \chi)(\wt{t}):= \chi(w^{-1} \wt{t} w).$$

Viewing $i(\chi)$ as a genuine representation of $\wt{B}$ by inflation from the quotient map $\wt{B} \to \wt{T}$, we denote by 
$$I(i(\chi)):=\text{Ind}_{\wt{B}}^{\wt{G}}\ i(\chi)$$ the normalized induced principal series representation of $\wt{G}$. For simplicity, we may also write $I(\chi)$ for $I(i(\chi))$. One knows that $I(\chi)$ is unramified (i.e., $I(\chi)^K\ne 0$) if and only if $\chi$ is unramified, i.e., $\chi$ is trivial on $Z(\wt{T})\cap K$.  Moreover, the Satake isomorphism for $\wt{G}$ implies that a genuine representation is unramified if and only if it is a subquotient of an unramified principal series.

Setting $\wt{Y}_{Q,n}:=Z(\wt{T})/(Z(\wt{T})\cap K)$, one has a natural abelian extension
\begin{equation} \label{Ext1}
\begin{tikzcd}
\bbmu_n \ar[r, hook] & \wt{Y}_{Q,n} \ar[r, two heads, "\varphi"] & Y_{Q,n}
\end{tikzcd}
\end{equation}
such that unramified genuine characters of $Z(\wt{T})$ correspond to genuine characters of $\wt{Y}_{Q,n}$. Since $\wt{A}/(T\cap K)\simeq \wt{Y}_{Q,n}$ as well, there is a canonical extension (also denoted by $\chi$) of an $\chi$ to $\wt{A}$, by composing it with $\wt{A} \twoheadrightarrow \wt{Y}_{Q,n}$. Therefore, we will identify $i(\chi)$ with $\text{Ind}_{\wt{A}}^{\wt{T}}\ \chi$ for this  canonical extension $\chi$.

\subsection{$\gamma$-function} \label{SS:Pg}

Let $\uchi: F^\times \to \C^\times$ be a linear character. Tate \cite{Tat}
defined a gamma factor $\gamma(s, \uchi, \psi), s\in \C$ which is essentially the ratio of two integrals of a test function and its Fourier transform (depending on a nontrivial character $\psi$). We have
$$\gamma(s, \uchi, \psi)= \varepsilon(s, \uchi, \psi) \cdot \frac{L(1-s, \uchi^{-1})}{L(s, \uchi)},$$
where $L(s, \uchi)$ is the $L$-function of $\uchi$. If $\uchi$ is unramified and the conductor of $\psi$ is $O$, then $$\varepsilon(s, \uchi, \psi)=1 \text{ and } L(s, \uchi)= (1-q^{-s} \uchi(\varpi))^{-1}.$$
In this case, we write
$$\gamma(s, \uchi):= \gamma(s, \uchi, \psi)= \frac{1-q^{-s} \uchi(\varpi)}{ 1- q^{-1+ s} \uchi(\varpi)^{-1}   }$$
with the $\psi$ omitted.

Let $\chi$ be a genuine unramified character of $Z(\wt{T})$. For every $\alpha \in \Phi$,
\begin{equation} \label{ud-chi}
\uchi_\alpha: F^\times \to \C^\times \text{ given by } \uchi_\alpha(a):= \chi(\wt{h}_\alpha(a^{n_\alpha}) )
\end{equation}
is an unramified character of $F^\times$.  We also use in this paper the shorthand notation
$$\chi_\alpha:=\uchi_\alpha(\varpi) \in \C^\times$$
for every root $\alpha$. For $\w=\w_\alpha$, we define the gamma factor $\gamma(w_\alpha, \chi)$ to be such that
$$\gamma(w_\alpha, \chi)^{-1}= \gamma(0, \uchi_\alpha)^{-1} = \frac{ 1-q^{-1} \chi(\wt{h}_\alpha (\varpi^{n_\alpha}))^{-1} }{ 1-\chi(\wt{h}_\alpha (\varpi^{n_\alpha}))   }.$$
For general $\w \in \W$, define
$$\gamma(w, \chi) = \prod_{\alpha \in \Phi_\w} \gamma(w_\alpha, \gamma),$$
where $\Phi_\w= \set{\alpha>0: \w(\alpha) <0}$.

\subsection{Intertwining operator}
For $\w \in W$, let $A(w,\chi): I(\chi) \to I({}^w \chi)$ be the intertwining operator defined by
$$A(w, \chi)(f)(\wt{g})=\int_{U_{w}} f(\wt{w}^{-1} u \wt{g}) du$$
whenever it is absolutely convergent. It can be meromorphically continued for all $\chi$ (see \cite[\S 7]{Mc1}). The operator $A(w, \chi)$ satisfies the cocycle relation as in the linear case.  Let $f_0\in I(\chi)$ and $f_0' \in I({}^w \chi)$ be the normalized unramified vectors. We have
$$A(w,\chi)(f_0)= \gk(w,\chi) f_0',$$
where
$$\gk(w,\chi)= \prod_{\alpha\in \Phi_\w}  \frac{ 1-q^{-1} \chi(\wt{h}_\alpha (\varpi^{n_\alpha})) }{ 1-\chi(\wt{h}_\alpha (\varpi^{n_\alpha}))   }.$$
The Plancherel measure $\mu(w, \chi)$ associated with $A(w,\chi)$ is such that 
\begin{equation} \label{E:Plan}
A(w^{-1}, {}^w \chi) \circ A(w, \chi)= \mu(w, \chi)^{-1} \cdot {\rm id};
\end{equation}
more explicitly, 
$$\mu(w, \chi)^{-1}=  \gk(w^{-1}, {}^w \chi) \cdot \gk(w, \chi).$$

To recall the interpretation of $\gk(w,\chi)$ in terms of $L$-functions, we first recall the set-up on the dual side. Consider the adjoint representation
$$Ad_{\wt{\mfr{u}}^\vee}: {}^L\wt{T} \to \GL(\wt{\mfr{u}}^\vee),$$
where $\wt{\mfr{u}}^\vee$ is the Lie algebra of unipotent radical $\wt{U}^\vee$ of the Borel subgroup $\wt{T}^\vee \wt{U}^\vee \subset \wt{G}^\vee$. It factors through $Ad^\C_{\wt{\mfr{u}}^\vee}$:
$$\begin{tikzcd}
 {}^L\wt{T}  \rar["{Ad_{\wt{\mfr{u}}^\vee}}"] \dar    & \GL(\wt{\mfr{u}}^\vee) \\
 \wt{T}^\vee/Z(\wt{G}^\vee) \ar[ru, "{Ad^\C_{\wt{\mfr{u}}^\vee}}"'],
\end{tikzcd}$$
as $Z(\wt{G}^\vee)$ acts trivially on $\wt{\mfr{u}}^\vee$.

Therefore, irreducible subspaces of $\wt{\mfr{u}}^\vee$ for $Ad_{\wt{\mfr{u}}^\vee}$ are in one-to-one correspondence with irreducible subspaces with respect to $Ad^\C_{\wt{\mfr{u}}^\vee}$, which are just the one-dimensional spaces associated to the positive roots of $\wt{G}^\vee$. More precisely, we have the decomposition of $Ad_{\wt{\mfr{u}}^\vee}$ into  irreducible ${}^L\wt{T}$-modules:
$$Ad_{\wt{\mfr{u}}^\vee}=\bigoplus_{\alpha > 0} Ad_\alpha,$$
where $(Ad_\alpha, V_\alpha) \subseteq \wt{\mfr{u}}^\vee$ is spanned by a basis  vector $E_{\alpha_{Q,n}^\vee}$ in the Lie algebra $\wt{\mfr{u}}^\vee_\alpha$ associated to the positive root $\alpha_{Q,n}^\vee$ of $\wt{G}^\vee$.

By the local Langlands correspondence for covering torus (see \cite[\S 10]{We6} or \cite[\S8]{GG}), associated to $i(\chi)$ we have a splitting 
$$\phi_\chi: W_F \longrightarrow {}^L\wt{T}$$ 
 of the L-group extension
\begin{equation} \label{L-T}
\begin{tikzcd}
\wt{T}^\vee \ar[r, hook] & {}^L\wt{T} \ar[r, two heads] & W_F,
\end{tikzcd}
\end{equation}
where $\wt{T}^\vee=\Hom (Y_{Q,n}, \C^\times)$ is the dual group of $\wt{T}$. Then $L(s, i(\chi), Ad_\alpha)$ is by definition equal to the local Artin $L$-function $L(s, Ad_\alpha \circ \phi_\chi)$ associated with $Ad_\alpha \circ \phi_\chi$, i.e., 
$$L(s, i(\chi), Ad_\alpha):=L(s, Ad_\alpha \circ \phi_\chi)=\text{det}\big(1-q^{-s} Ad_\alpha \circ \phi_\chi (\text{Frob})|_{V_\alpha^I} \big)^{-1}.$$
For unramified $i(\chi)$ (equivalently, unramified $\chi$), the inertia group $I$ acts trivially on $V_\alpha$. It follows that
$$L(s, i(\chi), Ad_\alpha)=\text{det}\big(1-q^{-s} Ad_\alpha \circ \phi_\chi(\varpi)|_{V_\alpha}\big)^{-1}.$$
Moreover, by \cite[Theorem 7.8]{Ga1}, one has
$$\phi_\chi \circ Ad(\varpi) (E_{\alpha_{Q,n}^\vee}) = \chi(\wt{h}_\alpha(\varpi^{n_\alpha})) \cdot E_{\alpha_{Q,n}^\vee}$$
and therefore
$$L(s, i(\chi), Ad_\alpha)=(1-q^{-s} \cdot  \chi(\wt{h}_\alpha(\varpi^{n_\alpha}))   )^{-1}= L(s, \uchi_\alpha).$$
Moreover, if we denote
$$E_\w:=\bigoplus_{\w \in \Phi_\w} \C \cdot E_{\alpha_{Q,n}^\vee},$$
and let $Ad_\w$ be the restriction of the adjoint representation $Ad$ to $E_\w$, then 
the Artin $L$-function associated to $Ad_\w$ is
$$L(s, Ad_\w \circ \phi_\chi)=\prod_{\alpha\in \Phi_\w} L(s, i(\chi), Ad_\alpha).$$
We also denote $L(s, i(\chi), Ad_\w):=L(s, Ad_\w \circ \phi_\chi)$. Thus,
\begin{equation}
\gk(w, \chi) = \frac{ L(0, i(\chi), Ad_\w)  }{ L(1, i(\chi), Ad_\w) }.
\end{equation}

\subsection{Normalization}
For $\w\in W$, we normalize the intertwining operator by
$$\msc{A}(w, \chi):= \gk(w,\chi)^{-1} \cdot A(w,\chi)=  \frac{ L(1, i(\chi), Ad_\w)  }{ L(0, i(\chi), Ad_\w) } \cdot A(w, \chi).$$
Note that in our setting $\varepsilon(s, \uchi_\alpha, \psi)=1$ for every $\alpha$, and thus the above normalization is the same as the one for linear algebraic groups, first proposed by Langlands \cite{Lds} and investigated in \cite{KS88, Sha3} for example. 

\begin{prop} \label{T:coc}
Let $\chi$ be an unramified genuine character of $Z(\wt{T})$. For every $\w_1, \w_2\in W$ (with no requirement on the length),
\begin{equation} \label{E:T-nat}
\msc{A}(w_2 w_1, \chi)= \msc{A}(w_2, {}^{w_1}\chi) \circ \msc{A}(w_1, \chi).
\end{equation}
In particular, $\msc{A}(w^{-1}, {}^w \chi) \circ \msc{A}(w, \chi)= {\rm id}$ for every $\w \in W$. Moreover, if $\chi$ is also unitary, then $\msc{A}(w, \chi)$ is holomorphic.
\end{prop}
\begin{proof}
The proof is essentially that of Winarsky \cite[Page 951-952]{Win}, which relies on an inductive argument on the length of $\w_2$ and also the basic step 
$$\msc{A}(w_\alpha, {}^{w_\alpha} \chi) \circ \msc{A}(w_\alpha, \chi)= \text{id}.$$
However, this last equality follows from \eqref{E:Plan} and the normalization of $\msc{A}(w_\alpha, \chi)$ in \eqref{E:T-nat}. See \cite[Page 313-314]{Mc1} for some details of the argument in the context of covering groups.
\end{proof}
\vskip 10pt

\section{R-group} \label{S:R-grp}
From now on, we assume that $\chi$ is a unitary unramified genuine character of $Z(\wt{T})$. Let
$$W_\chi:=\set{\w \in W: {}^w \chi = \chi} \subset W.$$
Proposition \ref{T:coc} shows that 
$$\w \mapsto \msc{A}(w, \chi)$$
gives rise to a group homomorphism
$$\begin{tikzcd} 
\tau_\chi: W_\chi \ar[r] & \text{Isom}_{\wt{G}}(I(\chi)),
\end{tikzcd} $$
where $\text{Isom}_{\wt{G}}(I(\chi))$ denotes the group of $\wt{G}$-isomorphisms of $I(\chi)$.  Let $\text{End}_{\wt{G}}(I(\chi))$ be the commuting algebra of $I(\chi)$. The group homomorphism $\tau_\chi$ gives an algebra homomorphism which, by abuse of notation, is also denoted by 
$$\begin{tikzcd}
\tau_\chi: \C[W_\chi] \ar[r] & \text{End}_{\wt{G}}( I(\chi) ).
\end{tikzcd} $$
However, $\tau_\chi$ is not an isomorphism in general.

We would like to define a subgroup $R_\chi \subset W_\chi$ such that $\tau_\chi$ induces an algebra isomorphism
$$\C[R_\chi] \simeq \text{End}_{\wt{G}}(I(\chi)).$$
For this purpose, consider the set 
$$\Phi_\chi=\set{\alpha >0: \underline{\chi}_\alpha = \mbm{1}} \subset \Phi,$$
where $\uchi_\alpha$ is as in \eqref{ud-chi}.
Let  $W_\chi^0 \subset W$ be the subgroup generated by $\set{ \w_\alpha: \alpha\in \Phi_\chi}$. It follows from \cite[Lemma 3.1]{Ga6} that $\w_\alpha \in W_\chi$ if $\alpha\in \Phi_\chi$. Therefore, 
$$W_\chi^0 \subseteq W_\chi$$
 and we have a short exact sequence:
$$\begin{tikzcd}
W_\chi^0 \ar[r, hook] &  W_\chi \ar[r, two heads] & R_\chi,
\end{tikzcd}$$
where $R_\chi:=W_\chi/W_\chi^0$. The sequence splits with a natural splitting $s: R_\chi \into W_\chi$ given as follows. Consider the group
$$W(\Phi_\chi)=\set{\w \in W: \w(\Phi_\chi) =\Phi_\chi}.$$
Then one has
$$R_\chi \simeq W_\chi \cap W(\Phi_\chi),$$
or more explicitly,
$$\begin{aligned}
R_\chi & \simeq \set{ \w\in W_\chi: \w(\Phi_\chi) = \Phi_\chi } \\
& = \set{\w \in W_\chi:  \alpha>0 \text{ and } \alpha\in \Phi_\chi \text{ imply that }  \w(\alpha)>0  } \\
& = \set{\w \in W_\chi:  \Phi_\w \cap \Phi_\chi = \emptyset}.
\end{aligned} $$
This gives that $W_\chi \simeq W_\chi^0 \rtimes R_\chi$. We always identity $R_\chi$ with  $W_\chi \cap W(\Phi_\chi)$.

Before we proceed, we first show that for $\w \in W_\chi$, the two factors $\gk(w, \chi)$ and $\gamma(w, \chi)^{-1}$ are actually equal.
\begin{lm} \label{P:gk-gm}
 Let $\wt{G}$ be an $n$-fold cover of a connected reductive group $G$, and $\chi$ a unitary unramified character of $Z(\wt{T})$. For every $\w \in W_\chi$, one has
$$L(s, Ad_\w \circ \phi_\chi)= L(s, Ad_\w^\vee \circ \phi_\chi),$$
where $Ad_\w^\vee$ is the contragredient representation of $Ad_\w$. Therefore, for every $\w \in W_\chi$,
$$\gk(w, \chi)^{-1} = \gamma(w, \chi),$$
which is nonzero if $\w \in R_\chi$.
\end{lm}
\begin{proof}
The argument is already in \cite[Lemma 4.2]{KS88}. First, since $\w \in W_\chi$, the two representations $Ad_\w \circ \phi_\chi$ and $Ad_\w^\vee \circ \phi_\chi$ are equivalent. Therefore, $L(s, Ad_\w \circ \phi_\chi)= L(s, Ad_\w^\vee \circ \phi_\chi)$. Now for $\w \in W_\chi$, we have
$$\gk(w, \chi)^{-1} = \frac{L(1, Ad_\w \circ \phi_\chi)}{L(0, Ad_\w \circ \phi_\chi)} = \frac{L(1, Ad_\w^\vee \circ \phi_\chi)}{L(0, Ad_\w \circ \phi_\chi)} = \gamma(w, \chi).$$
The proof is completed in view of the fact that $1- \uchi_\alpha(\varpi) \ne 0$ for  every $\alpha\in \Phi_\w$, if $\w \in R_\chi$.
\end{proof}

The main theorem on $R$-group is as follows:

\begin{thm}[{\cite{KnSt2, Sil1, Li5, Li3, Luo3}}] \label{T:Luo}
For a unitary  unramified genuine character $\chi$ of $Z(\wt{T})$, one has
$$W_\chi^0=\set{\w:  \msc{A}(w, \chi) \text{ is a scalar}  }.$$
Moreover, the algebra ${\rm End}_{\wt{G}}(I(\chi))$ has a basis given by $\set{ \msc{A}(w, \chi): \w \in R_\chi }$, and by restriction $\tau_\chi$ gives a natural algebra isomorphism
$$\C[R_\chi] \simeq {\rm End}_{\wt{G}}(I(\chi)).$$
\end{thm}

For a general parabolic induction of linear algebraic groups, the analogous result was first shown by Knapp--Stein \cite{KnSt2} for real groups, and by Silberger \cite{Sil1} for $p$-adic groups. The generalization to covering groups includes the work of W.-W. Li \cite{Li5, Li3}, which shows that Harish-Chandra's harmonic analysis extends to the covering setting. In particular, the Harish-Chandra $c$-function and Plancherel measure are discussed in detail in loc. cit.. Finally, in the recent work of C.-H. Luo \cite{Luo3}, the Harish-Chandra commuting algebra theorem is proved for general parabolic induction on covering groups, and in particular in the minimal parabolic case the isomorphism $\C[R_\chi] \simeq {\rm End}_{\wt{G}}(I(\chi))$ is established.
\vskip 5pt

Let
$$I(\chi) = \bigoplus_{i=1}^k m_i \pi_i$$
be the decomposition of $I(\chi)$ into irreducible subrepresentations with multiplicities $m_i$. Denote
$$\Pi(\chi):=\set{\pi_i: 1\le i \le k},$$
which constitutes the $L$-packet associated to the parameter $\phi_\chi$ corresponding to $I(\chi)$. Some immediate consequences of Theorem \ref{T:Luo}  are:
\begin{enumerate}
\item[(i)] $\dim \text{End}_{\wt{G}}(I(\chi))= \val{R_\chi}$ and $\C[R_\chi] \simeq \bigoplus_{i=1}^k M_{m_i}(\C)$.
\item[(ii)] $m_i=1$ for all $i$ if and only if $R_\chi$ is abelian.
\item[(iii)] in general, $k$ is equal to the dimension of the center of $\C[R_\chi]$, which is equal to the the number of conjugacy classes of $R_\chi$. Thus, there is a bijective correspondence
$$\text{Irr}(R_\chi)   \longleftrightarrow \Pi(\chi) .$$
\end{enumerate}
Since we are dealing with unramified $\chi$, we will show later (see Theorem \ref{T:R-abel}) that $R_\chi$ is actually abelian, and thus $I(\chi)$ is multiplicity-free.

\subsection{Parametrization of $\Pi(\chi)$}  \label{S:Pi-chi}

Denote the above correspondence 
$$ \text{Irr}(R_\chi)   \longleftrightarrow\Pi(\chi) $$
from (iii) by 
$$\sigma \leftrightarrow \pi_\sigma  , \text{ and }  \sigma_\pi \leftrightarrow  \pi.$$
The correspondence is not canonical, and can be twisted by characters of $R_\chi$.  However, in the unramified setting,  we have a natural parametrization given as follows (see \cite[Page 39]{Key3}).

Denote by $\theta_\sigma$ the character of $\sigma \in \text{Irr}(R_\chi)$. Consider the operator
\begin{equation} \label{P-sig}
P_\sigma:=\frac{\dim \sigma}{ \val{R_\chi} }  \sum_{\w \in R_\chi} \overline{\theta_\sigma(\w)} \cdot \msc{A}(w, \chi).
\end{equation}
Since the characters $\set{\theta_\sigma: \sigma \in \text{Irr}(R_\chi) }$ are orthogonal, the operators $P_\sigma$ are orthogonal projections of $I(\chi)$ onto nonzero invariant subspaces. Denote
$$V(\sigma):=P_\sigma(I(\chi)).$$
Then $V(\sigma)$'s are pairewise disjoint and we have 
$$I(\chi)=\bigoplus_{\sigma \in  \text{Irr}(R_\chi) }  V(\sigma).$$
Since the number of inequivalent constituents of $I(\chi)$ is equal to $\val{  \text{Irr}(R_\chi)  }$, it follows that $V(\sigma)$ is an isotypic sum of an irreducible representation. Thus, we write
$$V(\sigma)= m_{\pi_\sigma} \cdot \pi_\sigma,$$
and this gives a correspondence $\sigma \mapsto \pi_\sigma$. We also have
$$m_{\pi_\sigma}= \dim \sigma.$$

The inverse $\pi \mapsto \sigma_\pi $ can be described explicitly as well, see \cite[Page 39]{Key3}. The group $\text{Hom}(R_\chi, \C^\times)$ acts on $\Pi(\chi)$ by transporting the obvious action on $\text{Irr}(R_\chi)$ given by 
$$\xi \otimes \sigma, \text{ where } \xi \in \text{Hom}(R_\chi, \C^\times) \text{ and }\sigma \in \Irr(R_\chi).$$
 This action is free and transitive on the elements $\pi\in \Pi(\chi)$ which occurs with multiplicity one in $I(\chi)$. Thus the correspondence $\sigma \leftrightarrow \pi_\sigma$ is not canonical. However, the above parametrization given by \eqref{P-sig} (without any further twisting) is already a natural one in view of the following:

\begin{lm}
With notations as above, then $P_\mbm{1}(I(\chi)) = \pi_\chi^{un}$, i.e., the unramified constituent $\pi_\chi^{un}$ of $I(\chi)$ is parametrized by the trivial representation $\mbm{1} \in {\rm Irr}(R_\chi)$.
\end{lm}
\begin{proof}
It suffices to show that $P_\mbm{1}(f_0) = f_0$ where $f_0 \in \pi_\chi^{un} \subset  I(\chi)$ is the normalized  unramified vector. However, by \eqref{P-sig}, one has
$$P_\mbm{1}(f_0)= \frac{1}{ \val{R_\chi}  } \sum_{\w \in R_\chi} \msc{A}(w, \chi)(f_0).$$
By the normalization in $\msc{A}(w, \chi)$, one has $\msc{A}(w, \chi)(f_0)= f_0$ for every $\w\in R_\chi$ (in fact for every $\w \in W_\chi$). Thus, $P_\mbm{1}(f_0)= f_0$, and this gives the desired conclusion.
\end{proof}

\subsection{Some analysis on $R_\chi$} \label{SS:R-0}
From this subsection to \S \ref{S:R-sc}, we will analyze the group $R_\chi$ by relating it to the work \cite{Key2} of Keys on simply-connected Chevalley groups. More precisely, we will show that there is a group $R_\chi^{sc}$  containing $R_\chi$ such that:
\begin{enumerate}
\item[(i)] The quotient $R_\chi^{sc}/R_\chi\simeq W_\chi^{sc} / W_\chi$ can be understood from the dual side in terms of the $L$-parameter $\phi_\chi$. In particular, the size of $R_\chi^{sc}/R_\chi$ is related to the center $Z(\wt{G}^\vee)$ of the dual group $\wt{G}^\vee$.
\item[(ii)] 
The group $R_\chi^{sc}$ is equal to the R-group of a unramified principal series on a simply-connected Chevalley group, which is determined in \cite{Key2}. In particular, $R_\chi^{sc}$ is abelian, and this enforces $R_\chi$ to be abelian. However, since $R_\chi$ is not equal to $R_\chi^{sc}$ in general, this shows that unitary genuine principal series  of a covering group tends to be less reducible than the linear algebraic group.
\end{enumerate}

For this purpose, we first define a group $W_\chi^{sc}$ such that
$$W_\chi \subseteq W_\chi^{sc} \subseteq W.$$
Let $\mbf{T}_{Q,n}$ and $\mbf{T}_{Q,n}^{sc}$ be the split tori defined over $F$ associated to the two lattices $Y_{Q,n}$ and $Y_{Q,n}^{sc}$ respectively. Denote by $T_{Q,n}$ and $T_{Q,n}^{sc}$ their $F$-rational points. Let $T_{Q,n}^\dag$ and $T_{Q,n}^{sc, \dag}$ be the image of $T_{Q,n}$ and $T_{Q,n}^{sc}$ in $T$ with respect to the two maps 
$$i_{Q,n}: T_{Q,n} \to T \text{ and } i^{sc}: T_{Q,n}^{sc} \to T,$$
which are induced from $Y_{Q,n} \into Y$ and $Y_{Q,n}^{sc} \into Y$ respectively. Recalling the projection
$$\phi: \wt{G} \onto G,$$
we have from \S \ref{SS:ps} that 
$$Z(\wt{T}) =\phi^{-1}(T_{Q,n}^\dag).$$
We denote
$$Z(\wt{T})^{sc}:=\phi^{-1}(T_{Q,n}^{sc,\dag}) \subset Z(\wt{T}).$$
Let 
$$\chi^{sc}:=\chi|_{Z(\wt{T})^{sc}}$$
 be the genuine character of $Z(\wt{T})^{sc}$ obtained from the restriction of $\chi$. Consider
$$W_\chi^{sc}:=\set{\w\in W: {}^w(\chi^{sc}) =\chi^{sc}} \supset W_\chi$$
and analogously
$$R_\chi^{sc} := W_\chi^{sc}/W^0_\chi,$$
which then contains $R_\chi$. A splitting $s^{sc}$ of $R_\chi^{sc}$ into $W_\chi^{sc}$ is given by
$$R_\chi^{sc}\simeq W_\chi^{sc} \cap W(\Phi_\chi).$$
In summary, one has a commutative diagram with exact rows and compatible splittings:
$$\begin{tikzcd}
W_\chi^0 \ar[r, hook] \ar[d, equal]  &  W_\chi   \ar[r, two heads]  \ar[d, hook] & R_\chi  \ar[d, hook] \ar[l, bend right=50, "s"] \\
W_\chi^0 \ar[r, hook]  & W_\chi^{sc} \ar[r, two heads] & R_\chi^{sc} \ar[l, bend left=50, "s^{sc}"],
\end{tikzcd}$$
where $s$ and $s^{sc}$ are the aforementioned natural splittings of $R_\chi$ and $R_\chi^{sc}$ respectively.
It follows immediately that one has an isomorphism of finite groups:
$$ R_\chi^{sc} / R_\chi \simeq  W_\chi^{sc} / W_\chi.$$

\subsection{The quotient $R_\chi^{sc}/R_\chi$} \label{SS:quo}
Recall the $L$-parameter 
$$\phi_\chi: W_F \longrightarrow {}^L\wt{T}$$ 
associated to $i(\chi)$, which is a splitting of the extension
\begin{equation} \label{L-T}
\begin{tikzcd}
\wt{T}^\vee \ar[r, hook] & {}^L\wt{T} \ar[r, two heads] & W_F.
\end{tikzcd}
\end{equation}
In fact $\phi_\chi$ only depends on $\chi$, and this justifies our notation. Recall also that $\wt{T}^\vee=\Hom (Y_{Q,n}, \C^\times)$ is the dual group of $\wt{T}$ and 
$$Z(\wt{G}^\vee)\simeq \Hom(Y_{Q,n}/Y_{Q,n}^{sc}, \C^\times).$$
Denote
$$\wt{T}^\vee_{ad}:=\wt{T}^\vee/Z(\wt{G}^\vee) \simeq \Hom(Y_{Q,n}^{sc}, \C^\times).$$
Pushing out the exact sequence in (\ref{L-T}) via the map $f: \wt{T}^\vee \to \wt{T}^\vee/Z(\wt{G}^\vee)$, one obtains a short exact sequence
\begin{equation} \label{L-Tsc}
\begin{tikzcd}
\wt{T}^\vee_{ad} \ar[r, hook] & {}^L\wt{T}^{sc}:=f_*({}^L\wt{T}) \ar[r, two heads] & W_F.
\end{tikzcd}
\end{equation}
Denote by ${\rm Spl}({}^L\wt{T}, W_F)$ the set of splittings of (\ref{L-T}), and similarly $\text{Spl}({}^L\wt{T}^{sc}, W_F)$. Then $f$ induces a map
\begin{equation} \label{f-star}
\begin{tikzcd} 
f_*: \text{Spl}({}^L\wt{T}, W_F) \ar[r] & \text{Spl}({}^L\wt{T}^{sc}, W_F),
\end{tikzcd}
\end{equation}
which arises from the obvious composite. The Weyl group $W$ acts naturally on the two groups ${}^L\wt{T}$ and ${}^L\wt{T}^{sc}$, and the map 
$f_*$ in (\ref{f-star}) is $W$-equivariant. Here $f_*(\phi_\chi)$ is naturally associated with $\chi^{sc}$. 
Since the LLC for covering torus is $W$-equivariant (see \cite[\S 9.3]{GG}), we have
$$W_\chi= \text{Stab}_W(\phi_\chi),$$
and 
$$W_\chi^{sc}=\text{Stab}_W(f_*(\phi_\chi)).$$
Here $\text{Stab}_W(x)$ denotes the subgroup of stabilizers of $x$ in $W$. 

As it might be more convenient to work with parameters valued in the dual group (instead of the $L$-group), we can have the following reduction. From \S \ref{S:D-L}, we have a distinguished genuine character $\chi_\psi$ of $Z(\wt{T})$ depending on a nontrivial additive character $\psi$ of $F$. This gives a splitting $\phi_{\chi_\psi} \in \text{Spl}({}^L\wt{T}, W_F)$, which is fixed by $W$ since $\chi_\psi$ is $W$-invariant (see \cite[\S 6.5]{GG}). From this, one obtains a ``relative" version of (\ref{f-star}):
\begin{equation} \label{f-star-1}
\begin{tikzcd} 
f_*^\psi: \Hom(W_F, \wt{T}^\vee) \ar[r] & \Hom(W_F, \wt{T}^\vee_{ad}).
\end{tikzcd}
\end{equation}
If we denote
$$\phi_\chi^\psi:=\phi_\chi/\phi_{\chi_\psi} \in \Hom(W_F, \wt{T}^\vee),$$
then $f_*^\psi$ sends $\phi_\chi^\psi$ to
$$\frac{ f_*(\phi_\chi)}{f_*(\phi_{\chi_\psi})} \in \Hom(W_F, \wt{T}^\vee_{ad}).$$
The kernel of $f_*^\psi$ is given by
$$\text{Ker}(f_*^\psi)=\Hom(W_F, Z(\wt{G}^\vee)).$$
Since $f_*^\psi$ is also $W$-equivariant, we have
$$W_\chi= \text{Stab}_W(\phi_\chi^\psi), \quad W_\chi^{sc}=\text{Stab}_W(f_*^\psi(\phi_\chi^\psi)).$$

Assume henceforth that the conductor of $\psi$ is $O$. In this case, $\chi_\psi$ is unramified; thus $\phi_\chi^\psi$ is unramified and $\phi_\chi^\psi(\varpi) \in \wt{T}^\vee$ is the relative Satake parameter for $\chi$. Also, $f_*^\psi$ is determined by the map 
$$f^\vee: \wt{T}^\vee \to \wt{T}^\vee_{ad}$$
such that
$$f^\vee(\phi_\chi^\psi(\varpi))=(f_*^\psi(\phi_\chi^\psi))(\varpi).$$
For every $t\in \wt{T}^\vee$ such that $f^\vee(t)$ is fixed by $W_\chi^{sc}$, the action of $W_\chi^{sc}$ on the set $t\cdot Z(\wt{G}^\vee)$ is well-defined. For such $t$, let
$$\mca{O}^{W_\chi^{sc}}(t \cdot Z(\wt{G}^\vee))$$ 
be the set of $W_\chi^{sc}$-orbits in $t\cdot Z(\wt{G}^\vee) \subset \wt{T}^\vee$. 

\begin{prop} \label{P:bd}
Assume $G$ is a semisimple group. Let $\chi$ be a unitary unramified genuine character of $Z(\wt{T})$. Then 
$$[R_\chi^{sc}: R_\chi]=\frac{ \val{Z(\wt{G}^\vee)} }{ \val{ \mca{O}^{W_\chi^{sc}}(\phi_\chi^\psi(\varpi) \cdot Z(\wt{G}^\vee)) } }.$$
\end{prop}
\begin{proof}
For a finite group $H$ acting on a finite set $X$, the orbit counting formula reads
$$\val{X/H}= \frac{1}{\val{H}} \sum_{h\in H} \val{X^h},$$
where $X^h \subset X$ is the set of $h$-fixed points.  To apply this to the case $X= \phi_\chi^\psi(\varpi) \cdot Z(\wt{G}^\vee)$ and $H=W_\chi^{sc}$, we first note that the action of $W_\chi^{sc}$ on the set $\phi_\chi^\psi(\varpi) \cdot Z(\wt{G}^\vee)$ is well-defined, as mentioned above. Since $W$ (and thus in particular $W_\chi^{sc}$) acts trivially on $Z(\wt{G}^\vee)$, we see 
$$
X^\w = \begin{cases}
X, & \text{ if } \w \in W_\chi, \\
\emptyset & \text{ if } \w \notin W_\chi.
\end{cases}
$$
Thus,
$$\val{ \mca{O}^{W_\chi^{sc}}(\phi_\chi^\psi(\varpi) \cdot Z(\wt{G}^\vee)) }= \frac{1}{\val{W_\chi^{sc}}} \cdot \val{Z(\wt{G}^\vee)} \cdot \val{W_\chi}.$$
The result follows form the equality $[R_\chi^{sc}: R_\chi] = [W_\chi^{sc}: W_\chi]$.
\end{proof}

It is clear that the index $[R_\chi^{sc}: R_\chi]=[W_\chi^{sc}: W_\chi]$ is bounded above by $\val{Z(\wt{G}^\vee)}$ for covers of semisimple groups. In particular, if the dual group of $\wt{G}$ is of adjoint type, then $R_\chi^{sc}= R_\chi$. For example, if $G$ is a simply-connected Chevalley group and $n=1$, then there is no difference between $R_\chi^{sc}$ and $R_\chi$. For another nontrivial example, consider the $n$-fold cover $\wt{\SL}_{n+1}^{(n)}$, which has dual group $\text{PGL}_{n+1}$, or the odd-degree cover of $\Sp_{2r}$ whose dual group is $\SO_{2r+1}$.

\subsection{The group $R_\chi^{sc}$} \label{S:R-sc}

Let $\wt{G}$ be an $n$-fold cover of a general linear algebraic group. Let $\mbf{H}$ be the connected linear reductive group over $F$ such that its root datum is obtained from inverting that of $\wm{G}^\vee_{Q,n}$, i.e., the Langlands dual group of $\mbf{H}$ is isomorphic to $\wm{G}^\vee_{Q,n}$. If $n=1$, then $\mbf{H}= \mbf{G}$. Let
$$\mbf{H}^{sc} \onto \mbf{H}_{der} \into \mbf{H}$$
be the simply-connected cover of the derived subgroup $\mbf{H}_{der} \subset \mbf{H}$. Thus, $Y_{Q,n}^{sc}$ is the cocharacter (and also the coroot) lattice of $\mbf{H}^{sc}$. Denote by $H, H^{sc}$ the $F$-rational points of $\mbf{H}$ and $\mbf{H}^{sc}$ respectively. Here $H$ is the principal endoscopic group for $\wt{G}$.

We see that $T_{Q,n}^{sc}$ is just the torus of $H^{sc}$. The genuine character $\chi^{sc}$ of $Z(\wt{T})^{sc}$ gives rise to a linear unramified character
$$\uchi^{sc}: T_{Q,n}^{sc} \to \C^\times \text{ given by } \uchi^{sc}( \alpha_{Q,n}^\vee(a) ):= \chi^{sc}( \wt{h}_\alpha(a^{n_\alpha}) )$$
for all $\alpha \in \Delta$. In fact, the covering 
$$\begin{tikzcd}
\bbmu_n \ar[r, hook]  &  Z(\wt{T})^{sc}  \ar[r, two heads] & T_{Q,n}^{sc,\dag}
\end{tikzcd}$$
has a splitting $\rho^{sc}$ given by $\alpha_{Q,n}^\vee(a) \mapsto \wt{h}_\alpha(a^{n_\alpha})$ for all $\alpha\in \Delta$, and we have
$$\uchi^{sc} = \chi^{sc} \circ \rho^{sc} \circ i^{sc}.$$
One could form the unramified principal series $I(\uchi^{sc})$ of $H^{sc}$ and thus have the $R$-group $R_{\uchi^{sc}}$ for $I(\uchi^{sc})$. 

\begin{prop} \label{P:R-sc}
With notations as above, one has 
$$R_\chi^{sc} \simeq R_{\uchi^{sc}}.$$
\end{prop}
\begin{proof}
It suffices to show that $W_{\chi^{sc}} = W_{\uchi^{sc}}$ and $\Phi_{\chi^{sc}} = \Phi_{\uchi^{sc}}$. By definition of $\uchi^{sc}$,  we have
\begin{equation} \label{E:add1}
\uchi^{sc}( \alpha_{Q,n}^\vee(a) )= \chi^{sc}( \wt{h}_\alpha(a^{n_\alpha}) )
\end{equation}
for all $\alpha\in \Delta$. We claim that the equality holds for all $\alpha \in \Phi$.  By induction on the length of $\w$ such that $\w(\alpha) \in \Delta$, it suffices to prove that if 
$$\uchi^{sc}( \beta_{Q,n}^\vee(a) )= \chi^{sc}( \wt{h}_\beta(a^{n_\beta}) ), \beta \in \Phi$$
and $\gamma^\vee= \w_\alpha(\beta^\vee), \alpha\in \Delta$, then 
\begin{equation} \label{E:gam}
\uchi^{sc}( \gamma_{Q,n}^\vee(a) )= \chi^{sc}( \wt{h}_\gamma(a^{n_\gamma}) ), \beta \in \Phi
\end{equation}
As shown in \cite[Page 112]{Ga1}, one has
$$\wt{h}_\gamma(a^{n_\gamma}) = w_\alpha^{-1} \cdot \wt{h}_\beta(a^{n_\beta}) \cdot w_\alpha= w_\alpha \cdot \wt{h}_\beta(a^{n_\beta}) \cdot w_\alpha^{-1},$$
which is also equal to $\wt{h}_\beta(a^{n_\beta}) \cdot \wt{h}_\alpha(a^{n_\alpha})^{ -\angb{\alpha_{Q,n}}{\beta^\vee_{Q,n}} }$ by \eqref{F:W-act}. Thus by induction hypothesis, we have
$$\begin{aligned}
\uchi^{sc}( \gamma_{Q,n}^\vee(a) ) = &  \uchi^{sc}( \beta_{Q,n}^\vee(a) ) \cdot \uchi^{sc}( \alpha_{Q,n}^\vee(a) )^{ -\angb{\alpha_{Q,n}}{\beta^\vee_{Q,n}} } \\
= & \chi^{sc}( \wt{h}_\beta(a^{n_\beta}) ) \cdot \chi^{sc}( \wt{h}_\alpha(a^{n_\alpha}) )^{ -\angb{\alpha_{Q,n}}{\beta^\vee_{Q,n}} } \\
= & \chi^{sc}( \wt{h}_\gamma(a^{n_\gamma}) ).
\end{aligned}
$$
This proves that the equality in \eqref{E:add1} holds for all $\alpha \in \Phi$, and thus $\Phi_{\uchi^{sc}} = \Phi_{\chi^{sc}}$.  

Let $\w = \w_l \w_{l-1} ... \w_i$ be a minimal decomposition of $\w$ with $\w_i = \w_{\alpha_i}$ for some $\alpha_i \in \Delta$. Let $w= w_1 ... w_2 w_1$, where $w_i:=w_{\alpha_i}$ be the representative of $\w_{\alpha_i}$. The above argument also shows inductively that
$$w \cdot  \wt{h}_\alpha(\varpi^{n_\alpha})  \cdot w^{-1}= \wt{h}_{\w(\alpha)}(\varpi^{n_{\w(\alpha)}})$$
for all $\alpha \in \Phi$. It follows that
$$w^{-1} \cdot  \wt{h}_\alpha(\varpi^{n_\alpha})  \cdot w= \wt{h}_{\w^{-1}(\alpha)}(\varpi^{n_{\w^{-1}(\alpha)}})$$
for all $\alpha\in \Phi$, and thus
$${}^w (\uchi^{sc})( \alpha_{Q,n}^\vee(a) ) = {}^w (\chi^{sc})( \wt{h}_\alpha(a^{n_\alpha}) )$$
for all $\w \in W$ and $\alpha \in \Delta$. Therefore, $W_{\uchi^{sc}}= W_{\chi^{sc}}$. Thus by the definition of $R$-groups, we have $R_{\uchi^{sc}} \simeq R_\chi^{sc}$.
\end{proof}

In \cite[\S 3]{Key2}, all possibilities for the group $R_\chi^{sc} \simeq R_{\uchi^{sc}}$ are tabulated as in Table 1 and Table 2.

\begin{table}[!htbp]  \label{T1}
\caption{The group $R_\chi^{sc}$ for classical group $\mbf{H}^{sc}$}
\vskip 5pt
\begin{tabular}{|c|c|c|c|c|c|c|c|c|}
\hline
$\mbf{H}^{sc}$ & $A_r$  &  $B_r$ & $C_r$  & $D_r, r $ even &  $D_r, r $ odd   \\
\hline
$ R_\chi^{sc} $ & $\Z/d\Z, d|(r+1)$ & $\Z/2\Z$  & $\Z/2\Z$ & $\Z/2\Z$ or $(\Z/2\Z)^2$  & $\Z/2\Z$ or $\Z/4\Z$  \\ 
\hline
\end{tabular}
\end{table}

\begin{table}[!htbp]  \label{T2}
\caption{The group $R_\chi^{sc}$ for exceptional group $\mbf{H}^{sc}$}
\vskip 5pt
\begin{tabular}{|c|c|c|c|c|c|c|c|c|}
\hline
$\mbf{H}^{sc}$ & $E_6$  &  $E_7$ & $E_8$  & $F_4$ &  $G_2$ \\
\hline
$ R_\chi^{sc} $ & $\Z/3\Z$ & $\Z/2\Z$  & $\set{1}$ & $\set{1}$  & $\set{1}$  \\ 
\hline
\end{tabular}
\end{table}

Immediately, one has
\begin{thm} \label{T:R-abel}
Let $\wt{G}$ be an n-fold cover of a connected reductive group $G$, and $\chi$ a unitary unramified genuine character of $Z(\wt{T})$. Then $R_\chi \subset R_\chi^{sc}$ is abelian and therefore
$$I(\chi)= \bigoplus_{ \sigma \in \Irr(R_\chi) }  \pi_\sigma,$$
i.e., the decomposition of $I(\chi)$ is multiplicity-free.
\end{thm}

 \begin{cor} \label{C:Tcoef}
For every $\w \in R_\chi$, one has
 $$\msc{A}(w, \chi)= \bigoplus_{\sigma \in {\rm Irr}(R_\chi) }  \sigma(\w) \cdot {\rm id}_{\pi_\sigma}.$$
 Therefore, for every $f\in C_c^\infty(\wt{G})$, 
 $${\rm Trace }\ \msc{A}(w, \chi) I(\chi)(f)  = \sum_{\sigma \in \Irr(R_\chi)} \sigma(\w) \cdot {\rm Trace }\  \pi_\sigma(f).$$ 
 \end{cor}
 \begin{proof}
 It suffices to verify the first equality, as the second will follow from it.  For $\sigma \in \Irr(R_\chi)$, recall the projection $P_\sigma$ of $I(\chi)$ on $\pi_\sigma$. By Schur's Lemma, $\msc{A}(w, \chi)$ acts on $\pi_\sigma$ as a homothety given by $c_{\pi_\sigma}(\w) \in \C$.  However, 
 $$\begin{aligned}
P_{\sigma'} & = \frac{1}{ \val{R_\chi} } \sum_{\w \in R_\chi} \overline{\sigma'(\w)}  \Bigg(  \bigoplus_{\sigma  \in \Irr(R_\chi)} c_{\pi_\sigma}(\w) \cdot \text{id}_{\pi_\sigma} \Bigg) \\
& = \bigoplus_{\sigma  \in \Irr(R_\chi)} \Bigg(  \frac{1}{ \val{R_\chi} }  \sum_{\w \in R_\chi} \overline{\sigma'(\w)}  \cdot c_{\pi_\sigma}(\w) \Bigg) \cdot \text{id}_{\pi_\sigma}.
 \end{aligned} $$
 That is, the pairing of the class function $c_{\pi_\sigma}(-)$ on $R_\chi$ with $\sigma'(-)$ is equal to the Kronecker delta function $\delta_{\sigma, \sigma'}$. Since the characters $\Irr(R_\chi)$ form an orthonormal basis for class functions on $R_\chi$, this shows that $c_{\pi_\sigma}= \sigma$ for every $\sigma \in \Irr(R_\chi)$.
 \end{proof}

\subsection{Covers of $\SL_2$} \label{S:SL2}

We illustrate the previous discussion on R-group by considering the  $n$-fold cover $\wt{G}=\wt{\SL}_2^{(n)}$ arising from $Q(\alpha^\vee)=1$. We show that our analysis agrees with certain results in \cite{Szp6}.

If $n$ is even, then the dual group $\wt{G}^\vee$ is $\SL_2$. Write 
$$s_\zeta:=\phi_\chi^\psi(\varpi)=\left(\begin{matrix}
\zeta & \\
 & \zeta^{-1}
\end{matrix}\right) \in \wt{G}^\vee
$$
for the relative Satake parameter of $\chi$ discussed in \S \ref{SS:quo}, where $\zeta \in \C^\times$ depends on $\chi$. There are three cases:
\begin{enumerate}
\item[$\bullet$] $\zeta^4 \ne 1$. In this case, $\Phi_\chi=\emptyset$ and thus $W(\Phi_\chi)=W$. Also $W_\chi^{sc}=W_\chi=\set{1}$. Therefore $R_\chi^{sc}=R_\chi=\set{1}$, and in particular $\val{ \mca{O}^{W_\chi^{sc}}(s_\zeta \cdot Z(\wt{G}^\vee))}=2$.
\item[$\bullet$] $\zeta^4 = 1$ but $\zeta^2 \ne 1$. In this case, $\Phi_\chi=\emptyset$ and thus $W(\Phi_\chi)=W$. However, we have $W_\chi^{sc}=W$, while $W_\chi=\set{1}$. Therefore $R_\chi^{sc}=W$ and $R_\chi=\set{1}$. Indeed, we have $\val{ \mca{O}^{W_\chi^{sc}}(s_\zeta \cdot Z(\wt{G}^\vee))}=1$ in this case.
\item[$\bullet$] $\zeta^2= 1$. In this case, $\Phi_\chi=\set{\alpha}$ and thus $W(\Phi_\chi)=\set{1}$. Also $W_\chi^{sc}=W_\chi=W$. Therefore $R_\chi^{sc}=R_\chi=\set{1}$, and also $\val{ \mca{O}^{W_\chi^{sc}}(s_\zeta \cdot Z(\wt{G}^\vee))}=2$.
\end{enumerate}
Hence for $n$ even, $I(\chi)$ is always irreducible for a unitary unramified genuine character $\chi$. 

Now we consider $n$ odd, in which case the dual group of $\wt{\SL}_2^{(n)}$ is $\wt{G}^\vee=\PGL_2$ and thus one always has $R_\chi^{sc}= R_\chi$ by Proposition \ref{P:bd}. Write 
$$s_\zeta:=\rho_\chi^\psi(\varpi)=\left(\begin{matrix}
\zeta & \\
 & 1
\end{matrix} \right) \in \wt{G}^\vee
$$
for the relative Satake parameter for $I(\chi)$. There are three cases:
\begin{enumerate}
\item[$\bullet$] $\zeta^2 \ne 1$. In this case, $\Phi_\chi=\emptyset$ and thus $W(\Phi_\chi)=W$. Also $W_\chi=\set{1}$. Therefore $R_\chi=\set{1}$.
\item[$\bullet$] $\zeta = -1$. Then $\Phi_\chi=\emptyset$ and thus $W(\Phi_\chi)=W$. But $W_\chi=W$. Therefore $R_\chi=W$.
\item[$\bullet$] $\zeta = 1$. Then $\Phi_\chi=\set{\alpha}$ and thus $W(\Phi_\chi)=\set{1}$. But $W_\chi=W$. In this case, $R_\chi=\set{1}$.
\end{enumerate}

Combining the above even and odd cases, we see that the only reducibility point for $I(\chi)$ is when $n$ is odd and $\chi$ is such that $\uchi_\alpha$ is a nontrivial quadratic character; in this case
$$I(\chi) = \pi_\chi^{un} \oplus \pi.$$
The  result agrees with \cite[Proposition 5.1]{Szp6}.

\subsection{Comparison of $R_\chi$ and $\mca{S}_{\phi_\chi}$}  \label{L-para}
For linear algebraic groups, it was shown by Keys \cite{Key3} that the R-group $R_\chi$ is naturally identified with the component group of the centralizer of the parameter $\phi_\chi$. In this subsection, we establish the same identification for covering groups.

Let 
$$\phi: W_F \to {}^L\wt{G}$$ 
be an $L$-parameter. Let $S_\phi \subset \wt{G}^\vee$ be the centralizer of $\phi(W_F)$ in $\wt{G}^\vee$, i.e.,
$$S_\phi:=\set{ g\in \wt{G}^\vee:  g\cdot \phi(a) \cdot g^{-1} = \phi(a) \text{ for every } a\in W_F  };$$
it is a reductive subgroup of $\wt{G}^\vee$ but not necessarily connected.
Let $S_\phi^0$ be the connected component of $S_\phi$. Define
$$\mca{S}_\phi:= \frac{ S_\phi }{Z(\wt{G}^\vee) \cdot S_\phi^0}.$$
There is a dual group (instead of $L$-group) relative version which is more closely related to linear algebraic groups. Recall that depending on the choice of a distinguished genuine character  $\chi_\psi$, one has an isomorphism
$${}^L\wt{G} \simeq_{\chi_\psi}  \wt{G}^\vee \times W_F,$$
see \eqref{L-iso}. One obtains an unramified parameter
$$\begin{tikzcd}
\phi[\chi_\psi]:  W_F \ar[r, "\phi"] & {}^L\wt{G}  \ar[r, two heads] & \wt{G}^\vee,
\end{tikzcd}$$
where the second map is the projection depending on $\chi_\psi$.  Local Langlands correspondence gives the $L$-parameter $\phi_{\chi_\psi}$  associated to $\chi_\psi$, which takes values in $Z(\wt{G}^\vee)$. One has
$$\phi[\chi_\psi]=\phi \cdot \phi_{\chi_\psi}^{-1}.$$
Analogously, let $S_{\phi[\chi_\psi]} \subset \wt{G}^\vee$ be the centralizer of $\phi[\chi_\psi](W_F)$ in $\wt{G}^\vee$. Define
$$\mca{S}_{\phi[\chi_\psi]} := \frac{ S_{\phi[\chi_\psi]} }{ Z(\wt{G}^\vee) \cdot S_{\phi[\chi_\psi]}^0 }.$$

\begin{lm} \label{L:rel-S}
With notations as above, $S_\phi = S_{\phi[\chi_\psi]}$ for every distinguished genuine character $\chi_\psi$. Hence, $\mca{S}_\phi = \mca{S}_{\phi[\chi_\psi]}$ as well.
\end{lm}
\begin{proof}
It follows from the isomorphism ${}^L\wt{G} \simeq_{\chi_\psi}  \wt{G}^\vee \times W_F$ that
$$\phi(a) = (\phi[\chi_\psi](a), a) \in \wt{G}^\vee \times W_F.$$
Thus, $g\in \wt{G}^\vee$ centralizes $\phi(a)$ if and only if it centralizes $\phi[\chi_\psi](a)$; this gives the equality $S_\phi = S_{\phi[\chi_\psi]}$ and completes the proof.
 \end{proof}

\begin{thm} \label{T:R=S}
Let $\wt{G}$ be an $n$-fold cover of a connected reductive group $G$. Let $\chi$ be a unitary unramified genuine  character of $Z(\wt{T})$, and let $\phi_\chi$ be the $L$-parameter associated to $\chi$. One has an isomorphism 
$$R_\chi \simeq \mca{S}_{\phi_\chi}.$$
\end{thm}
\begin{proof}
The idea is to reduce it to the linear algebraic case where the isomorphism is proved in \cite[Page 42]{Key3}.

As in \S \ref{S:R-sc}, let $\mbf{H}$ be the connected split linear algebraic group whose dual group is $\wt{\mbf{G}}^\vee_{Q,n}$. Let $\mbf{T}_{Q,n}$ be its split torus whose cocharacter lattice is $Y_{Q,n}$. We have $H, T_{Q,n}$ denoting the $F$-rational points of $\mbf{H}, \mbf{T}_{Q,n}$ respectively.

Let $\chi$ be a unitary unramified genuine character of $Z(\wt{T})$. For simplicity, denote by $\phi$ (instead of $\phi_\chi$) the associated $L$-parameter valued in ${}^L\wt{T}$. By choosing a distinguished genuine character $\chi_\psi$, we have the unramified parameter
$$\phi[\chi_\psi]: W_F \to \wt{T}^\vee \into \wt{G}^\vee,$$
which is just $\phi_\chi^\psi$ in the notation of \S \ref{SS:quo}.
Note that $\wt{T}^\vee= T_{Q,n}^\vee$ and $\wt{G}^\vee = H^\vee$. The parameter $\phi[\chi_\psi]$ is thus associated to an unramified character 
$$\uchi: T_{Q,n} \to \C^\times$$
such that
$$\phi_{\uchi} = \phi[\chi_\psi],$$
where $\phi_{\uchi}$ is the $L$-parameter associated to $\uchi$ by the local Langlands correspondence for linear torus. We can explicate $\uchi$ as follows. First, $\chi\cdot \chi_\psi^{-1}: T_{Q,n}^\dag \to \C^\times$ is a linear character, where $T_{Q,n}^\dag$ is the image of the isogeny 
$$i_{Q,n}:   T_{Q,n} \to T,$$
see \S \ref{SS:R-0}. Then, $\uchi$ is just the pull-back of $\chi\cdot \chi_\psi^{-1}$ via $i_{Q,n}$; that is, $\uchi= (\chi\cdot \chi_\psi^{-1}) \circ i_{Q,n}$.
In any case, we have
$$\mca{S}_{\phi_{\uchi}} = \mca{S}_{\phi[\chi_\psi]}.$$

Now we consider $R_{\uchi}$ and $R_{\chi}$. Since $\chi_\psi$ is Weyl-invariant, we have 
\begin{equation} \label{RC1}
W_\chi = W_{\chi \cdot \chi_\psi^{-1}} = W_{\uchi}.
\end{equation} 
By the construction of distinguished genuine character, we have 
$$\chi_\psi( \wt{h}_\alpha( a^{n_\alpha} ) ) =1$$
for all $\alpha\in \Delta$ (see \cite[\S 6.1]{GG}). However, from the proof of Proposition \ref{P:R-sc} one has
$w \cdot \wt{h}_\alpha(\varpi^{n_\alpha}) \cdot w^{-1} = \wt{h}_{\w(\alpha)}  (\varpi^{n_{\w(\alpha)}})$
for all $\w \in W$ and $\alpha \in \Delta$. It follows that $\chi_\psi( \wt{h}_\alpha( a^{n_\alpha} ) ) =1$ for every $\alpha \in \Phi$.  Thus, 
\begin{equation} \label{RC2}
\begin{aligned}
 \Phi_{\uchi}= & \set{\alpha_{Q,n}>0:  \uchi(a^{\alpha_{Q,n}^\vee}) =1}  \\
=&  \set{\alpha >0:  (\chi \cdot \chi_\psi^{-1}) ( \wt{h}_\alpha(a^{n_\alpha}) ) =1} \\
=&  \set{\alpha >0:  \chi ( \wt{h}_\alpha(a^{n_\alpha}) ) =1} \\
= & \set{\alpha>0: \uchi_\alpha =\mbm{1}  } \\
= &  \Phi_\chi.
\end{aligned}
\end{equation} 
We deduce from \eqref{RC1} and \eqref{RC2} that 
$$R_{\uchi} \simeq R_{\chi}.$$
It is proved in \cite[Proposition 2.6]{Key3} that one has the isomorphism $R_{\uchi} \simeq \mca{S}_{\phi_{\uchi}}$. From this we see that $R_\chi \simeq \mca{S}_{\phi[\chi_\psi]}$, which is also isomorphic to $\mca{S}_\phi$ by Lemma \ref{L:rel-S}. This completes the proof.
\end{proof}

In view of the isomorphism $R_\chi \simeq \mca{S}_{\phi_\chi}$, the character relation in Corollary \ref{C:Tcoef} can be interpreted in terms of $\Irr(\mca{S}_{\phi_\chi})$ as well.

\begin{cor} \label{C:R=1}
If the dual group $\wt{G}^\vee$ is semisimple simply-connected, then $R_\chi=\set{1}$ for every unitary unramified genuine character $\chi$ of $Z(\wt{T})$.
\end{cor}
\begin{proof}
This follows from the isomorphism $R_\chi \simeq \mca{S}_{\phi_\chi}$ and the well-known result of Steinberg that the centralizer of a semisimple element inside a simply-connected group (such as $\wt{G}^\vee$ by our assumption) is connected.

Alternatively (and equivalently), from the proof of the preceding theorem, we have $R_\chi = R_{\uchi}$, where $I(\uchi)$ is an unramified principal series of $H$. If $\wt{G}^\vee$ is simply-connected, then $H$ is of adjoint type and it can be argued directly (see \cite[Corollary 1.6]{LiJS1} or \cite{BM1}) that  $R_{\uchi}=\set{1}$ in this case. 
\end{proof}

\begin{eg} \label{E:Sp-e}
Let $\mbf{G}=\Sp_{2r}$ and let $Q$ be a Weyl-invariant quadratic form on its cocharacter lattice. If $n_\alpha = 0 \ ({\rm mod} \ 2)$ for the unique short simple coroot $\alpha^\vee$ of $\Sp_{2r}$, then
$$\wt{G}^\vee=\Sp_{2r};$$
in this case $R_\chi=\set{1}$ and $I(\chi)$ is always irreducible. This applies in particular to the case of metaplectic-type group $\wt{\Sp}_{2r}^{(n)}$, i.e., when $n_{\alpha} = 2 \mod 4$, see Definition \ref{D:meta}. For $\wt{\SL}_2^{(n)}$, this is compatible with the discussion in \S  \ref{S:SL2}.
\end{eg}

\section{Whittaker space and the main conjecture}  \label{S:conj}

The main goal of this section is to investigate $\dim \Wh(\pi)$, where $\pi \in \Pi(\chi)$ is any irreducible constituent of $I(\chi)$. 

\subsection{The Whittaker space}
Let $\psi: F\to \C^\times$ be an additive character of conductor $O$. Let
$$\psi_U: U\to \C^\times$$
be the character of $U$ such that its restriction to every $U_\alpha, \alpha\in \Delta$ is given by $\psi \circ e_\alpha^{-1}$. We may write $\psi$ instead of $\psi_U$ for simplicity.

\begin{dfn}
For a genuine representation $(\pi, V_\pi)$ of $\wt{G}$, a linear functional $\ell: V_\pi \to \C$ is called a $\psi$-Whittaker functional if $\ell(\pi(u) v)= \psi(u) \cdot v$ for all $u\in U$ and $v\in V_\pi$.
Write $\Wh(\pi)$ for the space of $\psi$-Whittaker functionals for $\pi$.
\end{dfn}

The space $\Wh(I(\chi))$ for an unramified principal series $I(\chi)$ could be described as follows.
\begin{enumerate}
\item[$\bullet$] First, let $\Ftn(i(\chi))$ be the vector space of  functions $\cc$ on $\wt{T}$  satisfying
$$\cc(\wt{t} \cdot \wt{z}) =  \cc(\wt{t}) \cdot \chi(\wt{z}) \text{ for } \wt{t} \in \wt{T} \text{ and } \wt{z} \in \wt{A}.$$
The support of any $\cc \in \Ftn(i(\chi))$ is a disjoint union of cosets in $\wt{T}/\wt{A}$. We have
$$\dim \Ftn(i(\chi))=\val{\msc{X}_{Q,n}},$$
since $\wt{T}/\wt{A} \simeq Y/Y_{Q,n} = \msc{X}_{Q,n}$. 
\item[$\bullet$] Second, for every $\gamma \in \wt{T}$, let $\cc_\gamma \in \Ftn(i(\chi))$ be the unique element satisfying
$$\text{supp}(\cc_{\gamma})=\gamma \cdot \wt{A} \text{ and } \cc_{\gamma}(\gamma)=1.$$
Clearly, $\cc_{\gamma \cdot a} = \chi(a)^{-1} \cdot \cc_\gamma$ for every $a\in \wt{A}$. If $\set{\gamma_i}\subset \wt{T}$ is a chosen set of representatives of $\wt{T}/\wt{A}$, then $\set{\cc_{\gamma_i}}$ forms a basis for $\Ftn(i(\chi))$. Let $i(\chi)^\vee$ be the vector space of functionals of $i(\chi)$.
 The set $\set{\gamma_i}$ gives rise to linear functionals $l_{\gamma_i} \in i(\chi)^\vee$ such that $l_{\gamma_i}(\phi_{\gamma_j})=\delta_{ij}$, where $\phi_{\gamma_j}\in i(\chi)$ is the unique element such that 
$$\text{supp}(\phi_{\gamma_j})=\wt{A}\cdot \gamma_j^{-1} \text{ and  } \phi_{\gamma_j}(\gamma_j^{-1})=1.$$
It is easy to see that for any $\gamma\in \wt{T}$ and $a\in \wt{A}$, one has
$$\phi_{\gamma a}= \chi(a)\cdot \phi_\gamma, \quad  l_{\gamma a} = \chi(a)^{-1} \cdot l_\gamma.$$
Moreover, there is a natural isomorphism of vector spaces 
$$\Ftn(i(\chi)) \simeq i(\chi)^\vee$$
 given by
$$ \cc  \mapsto   l_\cc:= \sum_{\gamma_i \in \wt{T}/\wt{A}} \cc(\gamma_i) \cdot l_{\gamma_i}.
$$
It can be checked easily that this isomorphism does not depend on the choice of representatives for $\wt{T}/\wt{A}$.
\item[$\bullet$] Third, there is an isomorphism between $i(\chi)^\vee$ and the space $\Wh(I(\chi))$ of $\psi$-Whittaker functionals of $I(\chi)$ (see \cite[\S 6]{Mc2}), given by $\lambda \mapsto W_\lambda$ with
$$W_\lambda:  I(\chi) \to \C, \quad f \mapsto \lambda \left( \int_{U} f(\wt{w}_G^{-1}u) \psi(u)^{-1} \mu(u) \right),$$
where $f\in I(\chi)$ is an $i(\chi)$-valued function on $\wt{G}$. Here $\wt{w}_G=\wt{w}_{\alpha_1} \wt{w}_{\alpha_2} ... \wt{w}_{\alpha_k}\in K$ is the representative of $\w_G$, where $\w_G=\w_{\alpha_1} \w_{\alpha_2} ... \w_{\alpha_k}$ is a minimum decomposition of $\w_G$. 
\end{enumerate}
Thus, we have a chain of natural isomorphisms of vector spaces all of dimension $\val{\msc{X}_{Q,n}}$:
$$\Ftn(i(\chi))\simeq i(\chi)^\vee \simeq \Wh(I(\chi)).$$
For every $\cc\in \Ftn(i(\chi))$, by abuse of notation, we will write $\lambda_\cc^{\chi} \in \Wh(I(\chi))$ for the resulting $\psi$-Whittaker functional of $I(\chi)$ obtained from the above isomorphism.

The operator $A(w,\chi): I(\chi) \to I({}^w \chi)$ induces a homomorphism of vector spaces
$$A(w, \chi)^*: \Wh(I({}^w \chi)) \to \Wh(I(\chi)) $$
given by 
$$\angb{\lambda_\cc^{{}^w \chi} }{-} \mapsto \angb{\lambda_\cc^{{}^w \chi} }{A(w,\chi)(-)},$$
where $\cc \in \Ftn(i({}^w \chi))$.

\subsection{The scattering matrix $\mca{S}_\mfr{R}(w, i(\chi))$}
Let $\mfr{R}, \mfr{R}' \subset \wt{T}$ be two ordered sets of representatives of $\wt{T}/\wt{A}=\msc{X}_{Q,n}$.
Let 
$$\set{\lambda_{\gamma }^{^w \chi}: \gamma \in \mfr{R}}$$
be the ordered basis for  $\Wh(I({}^w \chi))$, and 
$$\set{ \lambda_{\gamma'}^{\chi}: \gamma'\in \mfr{R}' }$$
 the ordered basis for $\Wh(I(\chi))$. The map $A(w,\chi)^*$ is
then determined by the so-called scattering matrix
$$\mca{S}_{\mfr{R}, \mfr{R}'}(w, i(\chi)) = [\tau(w, \chi, \gamma, \gamma')]_{\gamma \in \mfr{R}, \gamma'\in \mfr{R}'}$$
such that
\begin{equation} \label{tau-Mat}
A(w, \chi)^*(\lambda_{\gamma}^{^w \chi}) = \sum_{\gamma'\in \mfr{R}'} \tau(w, \chi, \gamma, \gamma') \cdot \lambda_{\gamma'}^{\chi}.
\end{equation}

We briefly describe the matrix $\mca{S}_{\mfr{R}, \mfr{R}'}(w, i(\chi))$. First, we have:
\begin{enumerate}
\item[$\bullet$] For $\w\in W$ and $\wt{z}, \wt{z}'\in \wt{A}$, the identity 
\begin{equation} \label{SLCM1}
\tau(w, \chi, \gamma \cdot \wt{z}, \gamma' \cdot \wt{z}')=({}^w \chi)^{-1}(\wt{z}) \cdot \tau(w, \chi, \gamma, \gamma') \cdot \chi(\wt{z}')
\end{equation}
holds.
\item[$\bullet$] For $\w_1, \w_2 \in W$ such that $l(\w_2\w_1)=l(\w_2) + l(\w_1)$, one has
\begin{equation} \label{SLCM2}
\tau( w_2 w_1, \chi,  \gamma, \gamma')=\sum_{\gamma''\in \wt{T}/\wt{A}} \tau(w_2, {}^{w_1}\chi, \gamma, \gamma'') \cdot \tau(w_1, \chi, \gamma'', \gamma'),
\end{equation}
which is referred to as the cocycle relation. By \eqref{SLCM1}, the above sum is independent of the choice of representatives $\gamma''$.
\end{enumerate}

The above cocycle relation implies that in principle it suffices to understand $\tau(w_\alpha, \chi, \gamma, \gamma')$ for $\gamma, \gamma'\in \wt{T}$ for $\alpha\in \Delta$. For this purpose, let $du$ be the self-dual Haar measure of $F$ with respect to $\psi$ such that $du(O)=1$; thus, $du(O^\times)=1-q^{-1}$. Consider the Gauss sum
$$G_{\psi}(a, b):=\int_{O^\times} (u, \varpi)_n^a \cdot \psi(\varpi^b u) du \text{ for } a, b\in \Z.$$
We also write
$$\mathbf{g}_{\psi}(k):=G_{\psi}(k, -1),$$
where $k\in \Z$ is any integer. Denote henceforth
$$\vep:= (-1, \varpi)_n \in \C^\times.$$ 
It is known that
\begin{equation} \label{F:gauss}
\mbf{g}_{\psi}(k)=
\begin{cases}
\vep^k \cdot \overline{\mbf{g}_{\psi}(-k)} & \text{ for every  } k\in \Z, \\
-q^{-1} & \text{ if } n| k,\\
\mbf{g}_{\psi}(k) & \text{ with  }  \val{\mbf{g}_{\psi}(k)}=q^{-1/2}  \text{ if } n \nmid k.\\
\end{cases}
\end{equation}
Here $\overline{z}$ denotes the complex conjugation of a complex number $z$.

It is shown in \cite{KP, Mc2} (with some refinement from \cite{Ga2}) that $\tau(w_\alpha, \chi, \gamma, \gamma')$ is determined as follows.

\begin{thm} \label{T:tau}
Suppose that $\gamma=\s_{y_1}$ and $\gamma'=\s_y$ with $y_1, y\in Y$.  First, we can write $\tau(w_\alpha, \chi, \gamma, \gamma')=\tau^1(w_\alpha, \chi, \gamma, \gamma') + \tau^2(w_\alpha, \chi, \gamma, \gamma')$ satisfying the following properties:
\begin{enumerate} 
\item[$\bullet$] 
$\tau^i(w_\alpha, \chi, \gamma \cdot \wt{z}, \gamma' \cdot \wt{z}')=({}^{w_\alpha} \chi)^{-1}(\wt{z}) \cdot \tau^i( w_\alpha, \chi, \gamma, \gamma') \cdot \chi(\wt{z}') \text{ for all } \wt{z}, \wt{z}'\in \wt{A} $;
\item[$\bullet$] 
$\tau^1(w_\alpha, \chi, \gamma, \gamma')=0$  unless  $y_1 \equiv y \mod Y_{Q,n}$;
\item[$\bullet$] $\tau^2(w_\alpha, \chi, \gamma, \gamma')=0$   unless $y_1 \equiv \w_\alpha[y] \mod Y_{Q,n}$.
\end{enumerate}
Second, 
\begin{enumerate}
\item[$\bullet$] if $y_1= y$, then 
$$\tau^1(w_\alpha, \chi, \gamma, \gamma')=(1-q^{-1}) \frac{\chi (\wt{h}_\alpha(\varpi^{n_\alpha}))^{k_{y,\alpha}}}{1-\chi (\wt{h}_\alpha(\varpi^{n_\alpha}))}, \text{ where } k_{y,\alpha}=\ceil{\frac{\angb{y}{\alpha}}{n_\alpha}};$$
\item[$\bullet$] if $y_1=\w_\alpha[y]$, then
$$\tau^2(w_\alpha, \chi, \gamma, \gamma') = \vep^{ \angb{y_\rho}{\alpha} \cdot D(y, \alpha^\vee) } \cdot \g(\angb{y_\rho}{\alpha}Q(\alpha^\vee)).$$
\end{enumerate}
\end{thm}

\subsection{The main conjecture}
If $\mfr{R}=\mfr{R}'$, then 
$$\mca{S}_\mfr{R}(w, i(\chi)):= \mca{S}_{\mfr{R}, \mfr{R}}(w, i(\chi)),$$
is a square matrix with entires indexed by a single ordered set $\mfr{R}$. 
Since $W$ acts on $\msc{X}_{Q,n}=Y/Y_{Q,n}$ with respect to $\w[-]$, one has a decomposition
$$\msc{X}_{Q,n} = \bigsqcup_{ \mca{O}_y \in \mca{O}_{\msc{X}} } \mca{O}_y$$
into $W$-orbits. This gives a natural partition
$$\mfr{R} = \bigsqcup_{ \mca{O}_y \in \mca{O}_{\msc{X}} } \mfr{R}_y,$$
where $\mfr{R}_y \subset \mfr{R}$ is the subset of representatives of $\mca{O}_y$.
For each $W$-orbit $\mca{O}_y$,  denote
$$\mca{S}_\mfr{R}(w, i(\chi))_{\mca{O}_y} := [\tau(w, \chi, \s_{z}, \s_{z'})]_{z, z' \in \mfr{R}_y}.$$
It follows from the cocycle relation \eqref{SLCM2} and Theorem \ref{T:tau} that $\mca{S}_\mfr{R}(w, i(\chi))$ is a block-diagonal matrix with blocks $\mca{S}_\mfr{R}(w, i(\chi))_{\mca{O}_y}$ for $\mca{O}_y \in \mca{O}_{\msc{X}}$, which we write as
\begin{equation} \label{DB}
\mca{S}_\mfr{R}(w, i(\chi)) = \bigoplus_{ \mca{O}_y \in \mca{O}_{\msc{X}}  } \mca{S}_\mfr{R}(w, i(\chi))_{\mca{O}_y}.
  \end{equation}
  
In fact, for $y\in \msc{X}_{Q,n}$, let
$$\Wh (I(\chi))_{\mca{O}_y}= \text{Span}\set{\lambda^{\chi}_{\s_z}: z\in \mfr{R}_y   } \subset \Wh(I(\chi))$$
be the ``$\mca{O}_y$-subspace" of the Whittaker space of $I(\chi)$.
It is well-defined and independent of the representatives for $y\in \msc{X}_{Q,n}$, and moreover
$$\dim \Wh (I(\chi))_{\mca{O}_y}= \val{ \mca{O}_y }.$$
One has a decomposition
$$\Wh (I(\chi)) = \bigoplus_{\mca{O}_y \in \mca{O}_\msc{X}}  \Wh (I(\chi))_{\mca{O}_y}.$$
For every $\sigma \in {\rm Irr}(R_\chi)$, the inclusion  $h_\sigma: \pi_\sigma \into I(\chi)$ induces a surjection of vector spaces
$$h_\sigma^*: \Wh(I(\chi)) \onto \Wh(\pi_\sigma).$$
Denote 
\begin{equation} \label{Oy-Wh}
\Wh(\pi_\sigma)_{\mca{O}_y}:=h_\sigma^* \left(  \Wh(I(\chi))_{\mca{O}_y}  \right).
\end{equation}

Consider the natural permutation representation
$$\begin{tikzcd}
\sigma^\msc{X}: W \ar[r]  &  \perm(\msc{X}_{Q,n})
\end{tikzcd} $$
of $W$ given by $\sigma_{\msc{X}}(\w)(y) = \w[y]$. Clearly, one has a decomposition
$$\sigma^\msc{X} = \bigoplus_{ \mca{O}_y \in \mca{O}_\msc{X}  } \sigma_{\mca{O}_y}^\msc{X}, $$
where 
$$\sigma^\msc{X}_{\mca{O}_y}:  W \longrightarrow \perm(\mca{O}_y)$$
is the permutation representation on the $W$-orbit  $\mca{O}_y$. As we always identify $R_\chi$ as a subgroup of $W_\chi \subset W$, thus $\sigma_{\mca{O}_y}^\msc{X}$ can be viewed as a representation of $R_\chi$ by restriction.

\begin{conj} \label{MConj}
Let $\wt{G}$ be a saturated $n$-fold cover of a semisimple simply-connected group $G$. Let $\chi$ be a unitary unramified genuine character of $Z(\wt{T})$. Then for the natural correspondence (as discussed in \S \ref{S:Pi-chi})
$$\begin{tikzcd}
{\rm Irr}(R_\chi) \ar[r] & \Pi(\chi), \quad \sigma \mapsto \pi_\sigma,
\end{tikzcd}$$ 
one has 
$$\dim \Wh(\pi_\sigma)_{\mca{O}_y}= \angb{\sigma}{\sigma_{\mca{O}_y}^\msc{X}}_{R_\chi}$$
for every $W$-orbit  $\mca{O}_y$, where $\angb{\sigma_1}{\sigma_2}_{R_\chi}$ is the pairing of the two representations $\sigma_1, \sigma_2$ of $R_\chi$. In particular, $\dim \Wh(\pi_\sigma) = \angb{\sigma}{\sigma^\msc{X}}_{R_\chi}$ for every $\sigma \in \Irr(R_\chi)$.
\end{conj}

Note that the conjecture is trivially true for arbitrary $\wt{G}$ if $R_\chi=\set{1}$. This applies in particular to all Brylinski--Deligne covers of $\GL_r$, for which $R_\chi$ is always trivial. We also recall that if $G$ is simply-connected, then by Definition \ref{D:meta}, $\wt{G}$ is saturated if and only if the dual group $\wt{G}^\vee$ is of adjoint type, i.e., $Y_{Q,n}= Y_{Q,n}^{sc}$ in this case.

\begin{rmk} \label{R:comp}
 Conjecture \ref{MConj} is compatible with the decomposition 
$$I(\chi)=\bigoplus_{\sigma \in \text{Irr}(R_\chi)}  \pi_\sigma.$$
Indeed, one has 
$$\sum_{\sigma \in  \text{Irr}(R_\chi) }  \angb{ \sigma }{ \sigma^\msc{X}_{\mca{O}_y} }_{R_\chi} = \angb{ \C[R_\chi] }{ \sigma^\msc{X}_{\mca{O}_y} }_{R_\chi}= \val{\mca{O}_y},$$
which is equal to $\dim \Wh(I(\chi))_{\mca{O}_y}$.
\end{rmk}

\begin{rmk}
Ginzburg proposed in  \cite[Conjecture 1]{Gin4}   that if $\pi$ is an irreducible unramified representation of $\wt{G}$ which is non-generic, then there exists a non-generic theta representation $\Theta$ of a Levi subgroup $\wt{M} \subset \wt{P} \subset \wt{G}$ such that $\pi \into \text{Ind}_{ \wt{P}  }^{ \wt{G} } (\Theta)$. On the other hand, Conjecture \ref{MConj} implies that for simply-connected $G$, we have $\dim \Wh(\pi_\chi^{un})= \val{\mca{O}^{R_\chi}_{\msc{X}}  }$ (the number of $R_\chi$-orbits in $\msc{X}_{Q,n}$), since $\pi_\chi^{un} = \pi_\mbm{1}$. That is, Ginzburg's conjecture is vacuously true for such unramified representation $\pi_\chi^{un}$. For a comparison with the case of regular unramified $\chi$, see \cite[Remark 7.4]{Ga6}.
\end{rmk}

  \subsection{A formula for $\dim \Wh(\pi_\sigma)$} \label{S:formu}
In this subsection, let  $G$ be a general connected reductive group and let $\wt{G}$ be an $n$-fold cover, unless specified otherwise.  Consider the decomposition 
$$I(\chi)= \bigoplus_{\sigma \in \Irr(R_\chi)} \pi_\sigma.$$
Let $\w \in W_\chi$. We have ${}^w \chi \simeq \chi$, and thus an endomorphism
$$\msc{A}(w, \chi)^*: \Wh (I(\chi)) \to  \Wh(I(\chi))$$
induced from $\msc{A}(w, \chi)=\gamma(w, \chi) \cdot A(w, \chi)$, see Lemma \ref{P:gk-gm}. In fact, if $\w \in R_\chi$, then the normalizing factor is nonzero, and thus $A(w, \chi)$ is already holomorphic for every $\w \in R_\chi$. In any case, it follows from Corollary \ref{C:Tcoef} that for $\w \in R_\chi$
 $$
 \msc{A}(w, \chi)^*= \bigoplus_{\sigma \in {\rm Irr}(R_\chi) }  \sigma(\w) \cdot {\rm id}_{\Wh (\pi_\sigma)},
$$ 
and therefore the characteristic polynomial of $\msc{A}(w, \chi)^*$ is
 $$\begin{aligned}
\det \left(X\cdot \text{id} -  \msc{A}(w, \chi)^* \right) = &  \det \left( X \cdot I_{\val{\msc{X}_{Q,n}}} - \msc{A}(w, \chi)^* \right) \\ 
 =& \prod_{\sigma \in \Irr(R_\chi)} (X- \sigma(\w) )^{ \dim \Wh (\pi_\sigma)  } 
 \end{aligned} $$

For every $\w \in W$ and $\mca{O}_y \in \mca{O}_\msc{X}$, the restriction of $\msc{A}(w, \chi)^*: \Wh(I({}^w \chi)) \to \Wh(I(\chi))$ to $\Wh(I({}^w \chi))_{\mca{O}_y}$ gives a well-defined homomorphism
$$\msc{A}(w, \chi)^*_{\mca{O}_y}: \Wh(I({}^w \chi))_{\mca{O}_y} \to \Wh(I(\chi))_{\mca{O}_y}$$
represented by $\mca{S}_\mfr{R}(w, i(\chi))_{\mca{O}_y}$, see \cite[Proposition 4.4]{Ga6}. Moreover,
\begin{equation} \label{E:dcT}
\msc{A}(w, \chi)^*= \bigoplus_{\mca{O}_y \in \mca{O}_\msc{X}} \msc{A}(w, \chi)^*_{\mca{O}_y},
\end{equation}
where the sum is taken over all $W$-orbits in $\msc{X}_{Q,n}$. For $\w \in R_\chi$, this gives that
\begin{equation} \label{AO-dec}
 \msc{A}(w, \chi)^*_{\mca{O}_y}= \bigoplus_{\sigma \in {\rm Irr}(R_\chi) }  \sigma(\w) \cdot {\rm id}_{\Wh (\pi_\sigma)_{\mca{O}_y}}.
\end{equation}

For every $W$-orbit $\mca{O}_y \subset \msc{X}_{Q,n}$, consider the map
$$\sigma^{\rm Wh}_{\mca{O}_y}: R_\chi \longrightarrow  \GL(\Wh(I(\chi))_{\mca{O}_y} )$$
given by 
$$\sigma^{\rm Wh}_{\mca{O}_y}(\w):= \msc{A}(w, \chi)^*_{\mca{O}_y}.$$ 
Also let 
$$ \sigma^{\rm Wh}: R_\chi \longrightarrow  \GL(\Wh(I(\chi)) )$$
be the map given by $$\sigma^{\rm Wh}(\w):= \msc{A}(w, \chi)^*.$$
It is then clear from \eqref{E:dcT} that
$$\sigma^{\rm Wh} = \bigoplus_{\mca{O}_y \in \mca{O}_\msc{X} } \sigma^{\rm Wh}_{\mca{O}_y}.$$

\begin{thm} \label{T:dW}
Let $\wt{G}$ be an $n$-fold cover of a connected reductive group $G$. Then for every $\mca{O}_y \in \mca{O}_\msc{X}$, the map $\sigma^{\rm Wh}_{\mca{O}_y}$ is a well-defined group homomorphism and 
$$\dim \Wh(\pi_\sigma)_{\mca{O}_y} = \angb{\sigma}{ \sigma^{\rm Wh}_{\mca{O}_y} }_{R_\chi}$$
for every $\sigma\in \Irr(R_\chi)$. Hence, $\dim \Wh(\pi_\sigma)= \angb{\sigma}{ \sigma^{\rm Wh} }_{R_\chi}$.
\end{thm}
\begin{proof}
For $\w, \w' \in R_\chi$, we have
$$\begin{aligned}
\sigma^{\rm Wh}_{\mca{O}_y}(\w) \circ \sigma^{\rm Wh}_{\mca{O}_y}(\w') = &  \msc{A}(w, \chi)^*_{\mca{O}_y} \circ \msc{A}(w', \chi)^*_{\mca{O}_y} \\
= & \msc{A}(w, \chi)^*_{\mca{O}_y} \circ \msc{A}(w', {}^{w}\chi)^*_{\mca{O}_y} \\
= &  \left( \msc{A}(w', {}^{w}\chi) \circ \msc{A}(w, \chi) \right)^*_{\mca{O}_y} \\
= & \msc{A}(w' w, \chi)^*_{\mca{O}_y} \\
= & \sigma^{\rm Wh}_{\mca{O}_y}(\w' \w) = \sigma^{\rm Wh}_{\mca{O}_y}(\w \w'),
\end{aligned}$$
where the last equality follows from the fact that $R_\chi$ is abelian (see Theorem \ref{T:R-abel}). This shows that $\sigma^{\rm Wh}_{\mca{O}_y}$ is a representation of $R_\chi$ on $\Wh(I(\chi))_{\mca{O}_y}$. Clearly, the equality in \eqref{AO-dec} gives
$$\sigma^{\rm Wh}_{\mca{O}_y}= \bigoplus_{\sigma \in \Irr(R_\chi)} \dim \Wh(\pi_\sigma)_{\mca{O}_y} \cdot \sigma$$
and thus it follows that
$$\dim \Wh(\pi_\sigma)_{\mca{O}_y} = \angb{\sigma}{ \sigma^{\rm Wh}_{\mca{O}_y} }_{R_\chi} .$$
This completes the proof.
\end{proof}

If $\w\in R_\chi$, then $\gk(w, \chi)^{-1}= \gamma(w, \chi)$ (see Lemma \ref{P:gk-gm}) is actually holomorphic and nonzero at $\chi$. Denote by $\theta_{\sigma^{\rm Wh}_{\mca{O}_y}}$ the character of $\sigma^{\rm Wh}_{\mca{O}_y}$. For $\w \in R_\chi$, one has
$$\theta_{\sigma^{\rm Wh}_{\mca{O}_y}}(\w)= \gamma(w, \chi) \cdot \Tr\left( A(w, \chi)^*_{ \mca{O}_y }  \right)$$
and thus
$$\dim \Wh(\pi_\sigma)_{\mca{O}_y} = \frac{ \gamma(w, \chi) }{\val{R_\chi}} \cdot \sum_{\w \in R_\chi} \overline{ \sigma(\w) }  \cdot \Tr\left( A(w, \chi)^*_{ \mca{O}_y } \right).$$

Recall the permutation representation 
$$\sigma^\msc{X}_{\mca{O}_y}:  W \longrightarrow \perm(\mca{O}_y)$$
associated to the $W$-orbit  $\mca{O}_y$. Let $\theta_{\sigma^\msc{X}_{\mca{O}_y}}$ be the character of $\sigma^\msc{X}_{\mca{O}_y}$. Then for every $\w \in W$, one has
$$\theta_{\sigma^\msc{X}_{\mca{O}_y}}(\w)= \val{ (\mca{O}_y)^\w },$$
where 
$$(\mca{O}_y)^\w= \set{ y\in \mca{O}_y:  \w[y] = y  }.$$
By restriction, we view both $\sigma^\msc{X}$ and $\sigma^\msc{X}_{\mca{O}_y}$ as representations of $R_\chi$.
Conjecture \ref{MConj} could be reformulated as follows.

\begin{conj} \label{MConj2}
Let $\wt{G}$ be a saturated $n$-fold cover of a semisimple simply-connected group $G$. Let $\chi$ be a unitary unramified genuine character of $Z(\wt{T})$. Then for every $W$-orbit $\mca{O}_y$, one has
$$\sigma^{\rm Wh}_{\mca{O}_y} \simeq \sigma^\msc{X}_{\mca{O}_y};$$
equivalently but more explicitly, for every $\w \in R_\chi$, 
$$ \Tr \big( A(w, \chi)^*_{ \mca{O}_y } \big) = \val{ (\mca{O}_y)^\w  } \cdot \gamma(w, \chi)^{-1}.$$
\end{conj}

 The equivalence between Conjecture \ref{MConj} and Conjecture \ref{MConj2} follows from Theorem \ref{T:dW}. 
The computation of $\Tr \big( A(w, \chi)^*_{ \mca{O}_y } \big)$ is equivalent to that of  $\Tr\left(\mca{S}_\mfr{R}(w, i(\chi))_{\mca{O}_y}\right)$ for any representative set $\mfr{R} \subset Y$ of $\msc{X}_{Q,n}$.

\begin{rmk}
Conjecture \ref{MConj2} answers in a special case for the scattering matrix the analogous question raised in \cite[\S 3.2]{GSS2} regarding the trace of a local coefficients matrix. We note that if ${}^w \chi = \chi$, then both the local coefficients matrix and the scattering matrix associated to the operator $A(w, \chi)^*$ give invariants of the operator, albeit different. In this paper, it is the latter that is used and plays a crucial role in determining $\dim \Wh(\pi_\sigma)$. We hope that this phenomenon also helps justify our viewpoint in \cite{GSS2} that both the local coefficients matrix and the scattering matrix are important objects and should be studied together.
\end{rmk}

\subsection{Double cover of $\GSp_{2r}$} \label{SS:GSp}
In this subsection, we apply Theorem \ref{T:dW} to the double cover of $\GSp_{2r}$ and show that it recovers \cite[Corollary 6.6]{Szp5}. Meanwhile, we also show that the analogue of Conjecture \ref{MConj2} fails for such covers. This example shows that one can not extend the conjecture in a naive way to covers of a reductive group whose derived subgroup is simply-connected.

Let $\GSp_{2r}$ be the group of similitudes of symplectic type, and let $(X, \Delta, Y, \Delta^\vee)$ be its root data given as follows. The character group $X\simeq \Z^{r+1}$ has a standard basis 
$$\{e_i^*: 1\le i\le r \} \cup \{e_0^* \},$$
where the simple roots are 
$$\Delta=\{e_i^*-e_{i+1}^*: 1\le i \le r-1 \} \cup \{ 2e_r^*-e_0^* \}.$$

The cocharacter group $Y\simeq \Z^{r+1}$ is given with a basis 
$$\{e_i: 1\le i\le r \} \cup \{e_0 \}.$$
The simple coroots are 
$$\Delta^\vee=\{ e_i-e_{i+1}: 1\le i \le r-1 \} \cup \{ e_r \}.$$
Write $\alpha_i=e_i^* - e_{i+1}^*$, $\alpha^\vee_i=e_i-e_{i+1}$ for $1\le i\le r-1$, and also $\alpha_r=2e_r^*-e_0^*$, $\alpha_r^\vee=e_r$. Consider the covering $\overline{\GSp}_{2r}$ incarnated by $(D, \mbm{1})$. We are interested in those $\overline{\GSp}_{2r}$ whose restriction to $\Sp_{2r}$ is the one with $Q(\alpha_r^\vee)=1$. That is, we assume
$$Q(\alpha_i^\vee)=2 \text{ for } 1\le i \le r-1, \text{ and } Q(\alpha_r^\vee)=1.$$

Since $\Delta^\vee \cup \{e_0\}$ gives a basis for $Y$, to determine $Q$ it suffices to specify $Q(e_0)$. For $n=2$, we will obtain a double cover $\overline{\GSp}_{2r}$ which restricts to the classical metaplectic double cover $\overline{\Sp}_{2r}$. The number $Q(e_0) \in \Z/2\Z$ determines whether the similitude factor $F^\times$ corresponding to the cocharacter $e_0$ splits into $\overline{\GSp}_{2r}$ or not. To recover the classical double cover of $\GSp_{2r}$ (see \cite{Szp5}), we take $Q(e_0)$ to be an even number in this subsection.

In this case, one has
$$Y_{Q,2}=\left\{\sum_{i=1}^r k_i \alpha_i^\vee +k e_0 \in Y: k_i\in \Z \text{ for } 1\le i\le r-1, k_r, k \in 2\Z \right\}. $$
The sublattice $Y_{Q,2}^{sc}$ is spanned by $\{\alpha_{i, Q, 2}^\vee\}_{1\le i\le r}$, i.e.,
$$\{\alpha_1^\vee, \alpha_2^\vee, ..., \alpha_{r-1}^\vee, 2\alpha_r^\vee \}.$$
Regarding the dual group of the double cover $\wt{\GSp}_{2r}$, one has
$$  \overline{\GSp}_{2r}^{\vee} = \begin{cases}
\GSp_{2r}(\C), \text{  if $r$ is odd;} \\
{\rm PGSp}_{2r}(\C)  \times \GL_1(\C),  \text{  if $r$ is even.}  \end{cases} 
$$
Thus, $\mbf{H}$ is $\GSp_{2r}$ (resp. ${\rm Spin}_{2r+1} \times \GL_1$) if $r$ is odd (resp. even). 
Note that 
$$\msc{X}_{Q,2} = Y/Y_{Q,2} =\set{0, e_r, e_0, e_r + e_0},$$
which is isomorphic to $\Z/2\Z \times \Z/2\Z$.

If $r$ is odd, then the torus $T_{Q,2}$ of $\mbf{H}$ acts transitively on all non-degenerate characters of the unipotent subgroup of the Borel subgroup of $\mbf{H}$. Thus, $R_{\uchi} = \set{1}$ for every unramified unitary character $\uchi$ (see \cite[Lemma 2.5]{LiJS1}). Therefore, the R-group $R_\chi$ for $\wt{\GSp}_{2r}$ with $r$ odd is trivial for every unitary unramified genuine character $\chi$. 

We assume that $r$ is even. For every $\alpha\in \Phi$, recall that we have the notation (see \S \ref{SS:Pg})
$$\chi_\alpha:= \uchi_\alpha(\varpi) = \chi(\wt{h}_\alpha(\varpi^{n_\alpha})).$$
It follows from \cite[Page 399]{Key2} that the only nontrivial $R_\chi$ is $\set{1, \w} \simeq \Z/2\Z$ which is generated by
$$\w:= \w_{\alpha_1} \w_{\alpha_3} ... \w_{\alpha_{r-1}},$$
with the character $\chi$ satisfying
$$\chi_{\alpha_i} = -1 \text{ for all } i=2k-1, 1\le k \le r/2.$$

\begin{prop} \label{P:GSp}
Assume $r$ even and $\chi$ is an unramified character satisfying the above condition. Then as a representation of $R_\chi$,
$$\sigma^{\rm Wh} \simeq 2\cdot \mbm{1} \oplus 2 \cdot \varepsilon,$$
where $\varepsilon$ denotes the non-trivial character of $R_\chi$.
\end{prop}
\begin{proof}
It suffices to compute the trace of $\sigma^{\rm Wh}$. It is easy to see that for every $y\in \msc{X}_{Q,2}$ and $\w_{\alpha_i}, i = 1, 3, ..., r-1$, one has
$$\w_{\alpha_i}[y] = y \in \msc{X}_{Q,2}.$$
Thus, by \cite[Proposition 4.12]{GSS2}, the $(y, y)$-entry of $A(w, y)^*$ is given by
$$\tau(w, \chi, \s_y, \s_y)=\gamma(w, \chi)^{-1} \cdot \prod_{i= 1, 3, ..., r-1}  \chi_{\alpha_i}^{\angb{y}{\alpha_i}}= \gamma(w, \chi)^{-1} \cdot (-1)^{\angb{y}{\alpha_i}}.$$
We see
$$  \gamma(w, \chi) \cdot \tau(w, \chi, \s_y, \s_y) =
\begin{cases}
1 & \text{ if } y = 0, \\
-1 & \text{ if } y = e_r, \\
1 & \text{ if } y=e_0, \\
-1 & \text{ if } y=e_r + e_0.
\end{cases} $$
This shows that ${\rm Tr}(\sigma^{\rm Wh})(w) = 0$. Thus, $\sigma^{\rm Wh} = 2\cdot \mbm{1} \oplus 2\cdot \varepsilon$, as claimed.
\end{proof}

\begin{thm}[{\cite[Corollary 6.6]{Szp5}}] \label{T:GSp}
If $r$ is odd, then every unitary unramified genuine principal series $I(\chi)$ for the double cover $\wt{\GSp}_{2r}$ is irreducible. If $r$ is even, then the only reducibility of $I(\chi)$ occurs when $R_\chi \simeq \Z/2\Z$; in this case, $I(\chi) = \pi_\chi^{un} \oplus \pi_\varepsilon$, and 
$$\dim \Wh(\pi_\chi^{un}) = \dim \Wh(\pi_\varepsilon) = 2.$$
\end{thm}
\begin{proof}
We only need to show the last two equalities, which follow from combining Theorem \ref{T:dW} and Proposition \ref{P:GSp}.
\end{proof}

\begin{rmk}
It follows from the proof of Proposition \ref{P:GSp} that $\w[y] =y$ for every $y\in \msc{X}_{Q,2}$. Thus, 
$$\sigma^{\msc{X}} = 4 \cdot \mbm{1};$$
in particular, it is not isomorphic to $\sigma^{\rm Wh}$. We see that the (naive) analogue of Conjecture \ref{MConj2} fails for such $\wt{\GSp}_{2r}$.
\end{rmk}

\section{On the dimension of $\Wh(\pi_\chi^{un})$} \label{S:Wh-un}

\subsection{Lower bound for $\dim \Wh(\pi_\chi^{un})$}
In this subsection, we will prove Conjecture \ref{MConj2} for $\pi_\chi^{un}$ in a special case, which gives a lower bound of $\dim \Wh(\pi_\chi^{un})$.

Let $\wt{G}$ be an $n$-fold covering group of a connected reductive group $G$. Assume that $\wt{G}$ is not of metaplectic-type (see Definition \ref{D:meta}). We call $z\in Y$ an exceptional point (see \cite[Definition 5.1]{GSS2}) if
$$\angb{z_\rho}{\alpha}= -n_\alpha$$
for every $\alpha\in \Delta$, that is,
$$\w_\alpha[z] = z + \alpha_{Q,n}^\vee \text{ for every } \alpha \in \Delta.$$
Note that the definition here is the same as \cite[Definition 5.1]{GSS2}, since we have assumed that $G$ is not of metaplectic type.

For $\wt{G}$ not of metaplectic type, denote by $Y_n^{\rm exc} \subset Y$ the set of exceptional points. Let
$$f:  Y \onto \msc{X}_{Q,n}$$
be the quotient map, and denote 
$$\msc{X}_{Q,n}^{\rm exc}:= f(Y_n^{\rm exc}).$$
If $y\in Y$ is exceptional, then $y  \in f^{-1}(  \msc{X}_{Q,n}^W)$; that is,
$$\msc{X}_{Q,n}^{\rm exc} \subset (\msc{X}_{Q,n})^W.$$
Denoting 
$$\rho_{Q,n}:=\frac{1}{2} \sum_{\alpha >0} \alpha_{Q,n}^\vee \in Y\otimes \Q,$$
we always have
$$(\set{\rho - \rho_{Q,n} } \cap Y) \subseteq Y_n^{\rm exc}.$$
If $G$ is a semisimple group and $\wt{G}$ is not of metaplectic type, then (see \cite[Lemma 5.2]{GSS2})
$$Y_n^{\rm exc} = \set{\rho - \rho_{Q,n} } \cap Y;$$
that is, $Y_n^{\rm exc}$ contains the unique element $\rho - \rho_{Q,n}$ if it lies in $Y$. It is also determined explicitly in \cite[\S6--\S7]{GSS2} the dependence of $Y_n^{\rm exc}$ on $\wt{G}$ for covers of simply-connected groups.

\begin{thm} \label{T:unWh}
Let $\wt{G}$ be an $n$-fold cover of a connected reductive group $G$. Assume that $\wt{G}$ is not of metaplectic type. Then for every $z \in \msc{X}_{Q,n}^{\rm exc}$ and $\w \in R_\chi$, one has $ \mca{S}_{\mfr{R}}(w, i(\chi))_{\mca{O}_z} = \gamma(w, \chi)^{-1}$. Therefore,
$$\sigma_{\mca{O}_z}^{\rm Wh } = \sigma_{\mca{O}_z}^\msc{X} = \mbm{1}_{R_\chi},$$
and thus
$$\dim \Wh(\pi_\chi^{un}) \ge \val{ \msc{X}_{Q,n}^{\rm exc}  }.$$
In particular, if $\rho - \rho_{Q,n} \in Y$, then $\pi_\chi^{un}$ is generic. Moreover, Conjecture \ref{MConj2} holds for such $\mca{O}_z$. 
\end{thm}
\begin{proof}
Since $\mca{O}_z=\set{z}$, we have 
$$\mca{S}_{\mfr{R}}(w, i(\chi))_{\mca{O}_z} = \tau(w, \chi, \s_z, \s_z).$$
First, we note that by \eqref{SLCM1} and the fact that $\chi$ is fixed by $\w\in R_\chi$, the entry $\tau(w, \chi, \s_z, \s_z)$ is independent of the representative for $z\in \msc{X}_{Q,n}^{\rm exc}$. It follows from \cite[Proposition 4.12]{GSS2} (as $\wt{G}$ is not of metaplectic type) that 
$$\tau(w, \chi, \s_z, \s_z)= \gamma(w, \chi)^{-1}$$
for every $z\in \msc{X}_{Q,n}^{\rm exc}$, and Conjecture \ref{MConj2} holds for such $\mca{O}_z$. In fact, $\sigma_{\mca{O}_z}^{\rm Wh } = \sigma_{\mca{O}_z}^\msc{X} = \mbm{1}_{R_\chi}$ and thus
$$\dim \Wh(\pi_\chi^{un})_{\mca{O}_z}=  \angb{ \mbm{1} }{ \sigma^{\rm Wh}_{\mca{O}_z}  }_{R_\chi}= 1.$$
Therefore, $\dim \Wh(\pi_\chi^{un}) \ge \val{ \msc{X}_{Q,n}^{\rm exc}  }$. This completes the proof.
\end{proof}

For covers of a semisimple and simply-connected group $G$, it follows from \cite[Theorem 6.3]{GSS2} that
$$0\le \val{ \msc{X}_{Q,n}^{\rm exc}  } \le \val{ (\msc{X}_{Q,n})^W } \le 1.$$
In fact, in \cite[\S 7]{GSS2}, we determined explicitly the size of the two sets $\msc{X}_{Q,n}^{\rm exc}$ and $(\msc{X}_{Q,n})^W$. On the other hand, if $G$ is semisimple but not simply-connected, then it is possible to have
$$\val{ \msc{X}_{Q,n}^{\rm exc}  } \le 1 < \val{ (\msc{X}_{Q,n})^W }.$$
See \cite{GSS2} for details.

We note that the equality $\sigma_{\mca{O}_z}^{\rm Wh} = \sigma_{\mca{O}_z}^\msc{X}$ in Theorem \ref{T:unWh} might fail for $z \in (\msc{X}_{Q,n})^W- \msc{X}_{Q,n}^{\rm exc}$ for covers of a semisimple group: we will consider such an example from $n$-fold covers of $\SO_3$ in the next section. This example shows that the naive analogue of Conjecture \ref{MConj} does not hold for general semisimple groups. 

\begin{rmk}
If $G$ is almost simple and $\wt{G}$ is of metaplectic type, then it follows from the discussion after Definition \ref{D:meta} that $\wt{G}=\wt{\Sp}_{2r}$ with $n_\alpha \equiv 2 \ ({\rm mod}\ 4)$. Moreover, Example \ref{E:Sp-e} 
shows that $R_\chi=\set{1}$ in this case, and thus the (in-)equalities $\dim \Wh (\pi_\chi^{un}) = \val{\msc{X}_{Q,n}} \ge \val{\msc{X}_{Q,n}^{\rm exc}}$ hold trivially.
\end{rmk}

\subsection{Unramified Whittaker function} \label{S:UWf}
We note that Theorem \ref{T:unWh} applied to the case $n=1$ shows that $\pi_\chi^{un}$ is the only generic constituent of $I(\chi)$ for the linear algebraic group $G$. This fact also follows from the Casselman--Shalika formula \cite{CS}. Motivated by this, for covering groups, we consider in this subsection the relation between $\dim \Wh(\pi_\chi^{un})$ and the unramified Whittaker function, which is also the approach taken in \cite{Gin4} but for a general unramified genuine character.

First, by restriction, one obtains a surjection of vector spaces
$$h^{un}: \Wh(I(\chi)) \onto \Wh(\pi_\chi^{un}).$$
Let $f^0\in \pi_\chi^{un} \subset I(\chi)$ be the normalized unramified function. For $\cc \in \Ftn(i(\chi))$, let $\lambda_\cc \in \Wh(I(\chi))$ be the associated Whittaker functional, also viewed as an element in $\Wh(\pi_\chi^{un})$; that is, we may still use $\lambda_\cc$ for $h^{un}(\lambda_\cc)$ if no confusion arises. The unramified Whittaker function on $\wt{G}$ associated with $\cc$ is given by
$$\mca{W}_\cc(\wt{g}) = \lambda_\cc( \pi_\chi^{un}(\wt{g}) f^0   ) = \lambda_\cc( I(\chi)(\wt{g})f^0).$$
We have a decomposition $\wt{G}= \wt{B} K= U\wt{T} K$ and 
$$\mca{W}_\cc(u\wt{t} k)= \psi(u) \cdot \mca{W}_\cc(\wt{t}) \text{ for } u\in U, \wt{t}\in \wt{T}, k\in K.$$
Thus, the value of $\mca{W}_\cc$ is determined by its restriction to $\wt{T}$. If $\cc = \cc_\gamma$ for $\gamma \in \wt{T}$, then we denote
$$\mca{W}_\gamma:= \mca{W}_{\cc_\gamma}.$$

Recall that $\wt{t} \in \wt{T}$ is called dominant if 
$$\wt{t} \cdot (U\cap K) \cdot \wt{t}^{-1} \subset K.$$
Let 
$$Y^+=\set{y\in Y: \angb{y}{\alpha}\ge 0 \text{ for all } \alpha\in \Delta}.$$
Then an element $\s_y\in \wt{T}$ is dominant if and only if $y\in Y^+$.
The following result regarding $\mca{W}_\gamma(\wt{t})$ is shown in \cite{KP, Pat, CO} for coverings of $\GL_r$. For a general covering group, the idea is the same; it is implicit in \cite{Mc2} and explicated in \cite{Ga5}.

\begin{prop}
One has $\mca{W}_\gamma(\wt{t})=0$ unless $\wt{t} \in \wt{T}$ is dominant. Moreover, for dominant $\wt{t}$,
$$ \mca{W}_\gamma(\wt{t})=\delta_B^{1/2}( \wt{t} ) \cdot \sum_{\w \in W} \gk(w_G w^{-1}, \chi)  \cdot \tau(w, {}^{w^{-1}}\chi, \gamma, w_G \cdot \wt{t} \cdot w_G^{-1}),$$
where $\delta_B$ is the modular character of $B$.
\end{prop}

It follows that for $z\in Y$ and $y\in Y$,
$$ \mca{W}_{\s_z}(\s_y)=
\begin{cases}
\delta_B^{1/2}( \s_y ) \cdot \sum_{\w \in W} \gk(w_G w^{-1}, \chi)  \cdot \tau(w, {}^{w^{-1}}\chi, \s_z, w_G \cdot \s_y \cdot w_G^{-1}) & \text{ if } y\in Y^+, \\
0 & \text{ otherwise.}
\end{cases}$$
 We define for every $\gamma, \wt{t} \in \wt{T}$,
$$ \mca{W}^*_{\gamma}(\wt{t}):=\delta_B^{1/2}( \wt{t} )^{-1} \cdot \sum_{\w \in W} \gk(w_G w^{-1}, \chi)  \cdot \tau(w, {}^{w^{-1}}\chi, \gamma, \wt{t}).$$
We emphasize that here $\wt{t}$ is not required to be dominant. In particular, 
\begin{equation} \label{E:W*}
 \mca{W}^*_{\s_z}(\s_y)=\delta_B^{1/2}( \s_y )^{-1} \cdot \sum_{\w \in W} \gk(w_G w^{-1}, \chi)  \cdot \tau(w, {}^{w^{-1}}\chi, \s_z, \s_y)
 \end{equation}
 for every $z, y\in Y$. One can extend by linearity and define $\mca{W}_{\cc}^*(\wt{t})$ for every $\cc \in \Ftn(i(\chi))$. If $w_G^{-1}\wt{t} w_G$ is dominant, then
$$\mca{W}_{\cc}^*(\wt{t})= \mca{W}_{\cc}(w_G^{-1} \wt{t} w_G).$$

We denote 
$$\mca{W}^*:=\set{ \mca{W}^*_{\cc}: \cc \in \Ftn(i(\chi))  } .$$
If $\mca{W}^*_\cc=0$ as a function of $\wt{T}$, then $\mca{W}_\cc=0$ and it follows that $\lambda_\cc(v)=0$ for every $v\in \pi_\chi^{un}$ as the unramified vector $f^0$ generates $\pi_\chi^{un}$. Thus, one has an injection of vector spaces
\begin{equation} \label{emb-W}
\begin{tikzcd}
\Wh(\pi_\chi^{un}) \ar[r, hook] & \mca{W}^*
\end{tikzcd}
\end{equation}
given by 
$$\lambda_\cc \mapsto \mca{W}^*_\cc.$$

For $z\in \msc{X}_{Q,n}$, let $\mca{W}^*_{\mca{O}_z} \subset \mca{W}^*$ be the subspace spanned by $\set{\mca{W}^*_{\s_{z'}}: z' \in \mfr{R}_z }$, where $\mfr{R}_z \subset Y$ is a set of representatives of $\mca{O}_z$. Note that $\mca{W}^*_{\s_{z'}}$ depends on the representative $z'$; however, the space $\mca{W}^*_{\mca{O}_z}$ depends only on $\mca{O}_z$.  From the  restriction of \eqref{emb-W}, we obtain
\begin{equation} \label{Emb}
\begin{tikzcd}
 \Wh(\pi_\chi^{un})_{\mca{O}_z}  \ar[r,hook] & \mca{W}^*_{\mca{O}_z}.
 \end{tikzcd}
 \end{equation}
Let $\C^{  \val{\mca{O}_z} }$ be the $\val{ \mca{O}_z }$-dimensional complex vector space. 
Endow the space $\C^{  \val{\mca{O}_z} }$ with the coordinates indexed by $\mfr{R}_z$. Thus we write $(c_y)_{y\in \mfr{R}_z}$ for a general vector in $\C^{  \val{\mca{O}_z } }$.  Depending on $\mfr{R}_z$, there is an evaluation map
$$\begin{tikzcd}
\nu_{\mfr{R}_z}:  \mca{W}^*_{\mca{O}_z} \ar[r]  & \C^{  \val{\mca{O}_z} }
\end{tikzcd} $$
with the $y$-th coordinate of $\nu_{\mfr{R}_z}(\mca{W}_{\s_{z'}}), y, z' \in \mfr{R}_z$ given by
$$\nu_{\mfr{R}_z}(\mca{W}^*_{\s_{z'}})_y = \mca{W}^*_{\s_{z'}}(\s_y);$$
that is, 
$$\nu_{\mfr{R}_z}(\mca{W}^*_{\s_{z'}})= \left( \mca{W}^*_{\s_{z'}}(\s_y) \right)_{y\in \mfr{R}_z} \in \C^{\val{ \mca{O}_z}}.$$
If $\mfr{R} \subset Y$ is a set of representatives for $\msc{X}_{Q,n}$, then we have a unique subset $\mfr{R}_z \subset \mfr{R}$ representing $\mca{O}_z$. By combining all the $\nu_{\mfr{R}_z}$, one obtains an evaluation map
$$\begin{tikzcd}
\nu:  \mca{W}^* \ar[r]  & \C^{  \val{\msc{X}_{Q,n}} },
\end{tikzcd} $$
which depends on the chosen $\mfr{R}$. In particular, for a general $\cc \in \Ftn(i(\chi))$, we have
$$\nu(\mca{W}^*_\cc)_y= \mca{W}^*_\cc(\s_y).$$

For every $\mca{O}_z \in \mca{O}_{\msc{X}}$, composing \eqref{Emb} with $\nu_{\mfr{R}_z}$ gives a vector space homomorphism
$$\begin{tikzcd}
\nu^\chi_{\mfr{R}_z}:  \Wh(\pi_\chi^{un})_{\mca{O}_z}  \ar[r]  &  \C^{ \val{ \mca{O}_z } }.
\end{tikzcd} $$
Similarly, we have 
$$\begin{tikzcd}
\nu^\chi: \Wh(\pi_\chi^{un}) \ar[r]  & \C^{ \val{ \msc{X}_{Q,n} } }
\end{tikzcd} $$
with
$$ \nu^\chi = \bigoplus_{ \mca{O}_z \in \mca{O}_{\msc{X}}  }  \nu^\chi_{ \mfr{R}_z }.$$

\begin{conj} \label{C:inj}
For every $W$-orbit $\mca{O}_z$ and every choice of $\mfr{R}_z$, the homomorphism $\nu^\chi_{\mfr{R}_z}:  \Wh(\pi_\chi^{un})_{\mca{O}_z}  \to    \C^{ \val{ \mca{O}_z } }$ is injective. Thus, $\dim \Wh(\pi_\chi^{un})_{ \mca{O}_z  } =  {\rm rank}(\nu^\chi_{ \mfr{R}_z }  )$ and 
$$\dim \Wh( \pi_\chi^{un} )= \sum_{ \mca{O}_z\in \mca{O}_\msc{X} }  {\rm rank}(\nu^\chi_{ \mfr{R}_z }  ).$$
\end{conj}
For application purpose (i.e., to determine $\dim\Wh(\pi_\chi^{un})_{\mca{O}_z}$), it is sufficient to consider one $\mfr{R}_z$. More precisely, for low-rank groups, or low-degree covering groups, one can verify the injectivity in Conjecture \ref{C:inj} for a particular $\mfr{R}_z$ and compute the above ${\rm rank}(\nu^\chi_{ \mfr{R}_z })$ explicitly.

\begin{thm} \label{T:un-gen}
Let $\wt{G}$ be an $n$-fold cover of a linear algebraic group $G$, which is not of metaplectic type (see Definition \ref{D:meta}). For every $z\in Y_n^{\rm exc}$, taking $\mfr{R}_z =\set{z}$, we have 
$${\rm rank}(\nu_{ \mfr{R}_z }^\chi) =1.$$
Hence,  $\dim \Wh(\pi_\chi^{un}) \ge \val{ \msc{X}_{Q,n}^{\rm exc}  }$. 
\end{thm} 
\begin{proof}
If $z\in Y_n^{\rm exc}$, then $-z$ is dominant and it follows from \cite[Theorem 5.6]{GSS2} that we have a Casselman--Shalika formula for $\wt{G}$ which reads
$$\mca{W}^*_{\s_z}(\s_{z}) = \delta_B^{1/2}(\s_{z})^{-1} \cdot \prod_{\alpha > 0} (1-q^{-1} \chi_\alpha).$$
Since $\chi$ is unitary, we see that 
$$\nu_{\mfr{R}_z}^\chi(\lambda_{\s_z}) = \mca{W}^*_{\s_z}(\s_{z}) \ne 0.$$
This shows that $\nu_{\mfr{R}_z}^\chi$ is an isomorphism between the one-dimensional vector spaces. The result follows.
\end{proof}
Theorem \ref{T:un-gen} is compatible with Theorem \ref{T:unWh}, though the approach in the former highlights the role of unramified Whittaker functions.

\section{Covers of symplectic groups} \label{S:eg}
The goal of this section is to show that Conjecture \ref{MConj2} (equivalently Conjecture \ref{MConj}) holds for $\wt{\Sp}_{2r}^{(n)}, r\ge 1$ and $\wt{\SL}_3^{(2)}$. Recall that for every $\alpha\in \Phi$ we denote
$$\chi_\alpha:= \uchi_\alpha(\varpi)= \chi(\wt{h}_\alpha(\varpi^{n_\alpha})).$$

\subsection{Covers of $\Sp_{2r}, r\ge 1$}
Consider the Dynkin diagram for the simple coroots of $\Sp_{2r}$:

$$ \qquad 
\begin{picture}(4.7,0.2)(0,0)
\put(1,0){\circle{0.08}}
\put(1.5,0){\circle{0.08}}
\put(2,0){\circle{0.08}}
\put(2.5,0){\circle{0.08}}
\put(3,0){\circle{0.08}}
\put(1.04,0){\line(1,0){0.42}}
\multiput(1.55,0)(0.05,0){9}{\circle*{0.02}}
\put(2.04,0){\line(1,0){0.42}}
\put(2.54,0.015){\line(1,0){0.42}}
\put(2.54,-0.015){\line(1,0){0.42}}
\put(2.74,-0.04){$>$}
\put(1,0.1){\footnotesize $\alpha_1^\vee$}
\put(1.5,0.1){\footnotesize $\alpha_2^\vee$}
\put(2,0.1){\footnotesize $\alpha_{r-2}^\vee$}
\put(2.5,0.1){\footnotesize $\alpha_{r-1}^\vee$}
\put(3,0.1){\footnotesize $\alpha_r^\vee$}
\end{picture}
$$
\vskip 10pt

Let $Y=Y^{\sct}=\langle \alpha_1^\vee, \alpha^\vee_2, ..., \alpha_{r-1}^\vee, \alpha_r^\vee \rangle$ be the cocharacter lattice of $\Sp_{2r}$, where $\alpha_r^\vee$ is the short coroot as shown in the above diagram. For simplicity, let $Q$ be the Weyl-invariant quadratic form on $Y$ such that $Q(\alpha_r^\vee)=1$. The bilinear form $B_Q$ is given by
$$
B_Q(\alpha_i^\vee, \alpha_j^\vee) =
\begin{cases}
2 & \text{if } i=j=r, \\
4& \text{if } 1\le i=j \le r-1, \\
-2 & \text{if } j=i+1,\\
0 & \text{if $\alpha_i^\vee, \alpha_j^\vee$ are not adjacent}.
\end{cases}
$$
Let $\wt{G}:=\wt{\Sp}_{2r}^{(n)}$ be the $n$-fold cover of $\Sp_{2r}$. We have 
$$\wt{G}^\vee =
\begin{cases}
\Sp_{2r} & \text{ if $n$ is even}; \\
\SO_{2r+1} & \text{ if $n$ is odd}.
\end{cases}
$$
By Corollary \ref{C:R=1}, one has $R_\chi=\set{1}$ if $n$ even. For odd $n$, it is clear that $n_{\alpha_i}=n$ for all $\alpha_i \in \Delta$ and
$$Y_{Q,n}=Y_{Q,n}^{sc}=  nY.$$
Following notations in \cite[Page 267]{Bou}, we consider the map 
$$\bigoplus_{i=1}^r \Z \alpha_i^\vee \to \bigoplus_{i=1}^r \Z e_i$$
 given by
$$(x_1, x_2, x_3 ..., x_r) \mapsto (x_1, x_2-x_1, x_3-x_2, ..., x_{r-1} -x_{r-2}, x_r -x_{r-1}),$$
which is an isomorphism. The Weyl group is $W= S_r \rtimes (\Z/2\Z)^r$, where $S_r$ is the permutation group on $\bigoplus_i \Z e_i$ and each $(\Z/2\Z)_i$ acts by $e_i \mapsto \pm e_i $. In particular, $\w_{\alpha_i}, 1\le i\le r-1,$ acts on $(y_1, y_2, ..., y_r) \in \bigoplus_i \Z e_i$ by exchanging $y_i$ and $y_{i+1}$, while $\w_{\alpha_r}$ acts by $(-1)$ on $\Z e_r$.

For odd $n$, it follows from Proposition \ref{P:bd}, Proposition \ref{P:R-sc} and \cite[\S 3]{Key2} that the only possible non-trivial $R$-group (up to isomorphism) for $\wt{G}$ is 
$$R_\chi=\set{1, \w_{\alpha_r}},$$
where $\chi$ is the unramified genuine character of $Z(\wt{T}) \subset \wt{G}$ such that
\begin{enumerate}
\item[$\bullet$] $\uchi_{\alpha_i}$  is any unitary unramified linear character for all   $1\le i \le r-1$, and
\item[$\bullet$] furthermore
$$\uchi_{\alpha_r}^2 = \mbm{1} \text{ and }   \uchi_{\alpha_r} \ne \mbm{1}.$$
\end{enumerate}
In particular, $\chi_{\alpha_r} = -1$. Denote by $\varepsilon$ the non-trivial character of $R_\chi$. We have a decomposition 
$$I(\chi) = \pi_\chi^{un} \oplus \pi_\varepsilon,$$
where $\pi_\varepsilon$ is non-isomorphic to $\pi_\chi^{un} =\pi_{\mathbbm{1}}$.

\begin{thm} \label{T:Sp2r}
For odd $n$ and $\chi$ as above, Conjecture \ref{MConj2} holds for $\wt{\Sp}_{2r}^{(n)}$; in this case,
$$\dim \Wh(\pi_\chi^{un}) = \frac{n^r + n^{r-1}}{ 2 }, \quad \dim \Wh(\pi_\varepsilon) = \frac{n^r - n^{r-1}}{ 2 }.$$
\end{thm}
\begin{proof}
For $n$ odd
$$\msc{X}_{Q,n} = Y/Y_{Q,n} \simeq (\Z/n\Z)^r.$$
For every $W$-orbit $\mca{O}_y \subset \msc{X}_{Q,n}$, we will compute and check explicitly that
\begin{equation} \label{MC-C1}
\Tr\left( A(w_{\alpha_r}, \chi)^*_{ \mca{O}_y } \right) = \val{ (\mca{O}_y)^{\w_{\alpha_r} }  } \cdot \gamma(w_{\alpha_r}, \chi)^{-1}.
\end{equation}
One has a decomposition 
$$\mca{O}_y = \bigsqcup_{i\in I} \mca{O}_{z_i}^{R_\chi}$$
of $\mca{O}_y$ into $R_\chi$-orbits, where
$$\mca{O}_z^{R_\chi}= \set{z} \text{ or }    \mca{O}_z^{R_\chi} = \set{z, \w_{\alpha_r}[z]}.$$
To show \eqref{MC-C1}, it suffices to prove that for every $R_\chi$-orbit $\mca{O}_z^{R_\chi} \subset \msc{X}_{Q,n}$, one has
\begin{equation}  \label{MC-C2}
\sum_{z' \in \mca{O}_z^{R_\chi}}  \tau(w_{\alpha_r}, \chi, \s_{z'}, \s_{z'}) =  \val{ ( \mca{O}_z^{R_\chi} )^{\w_{\alpha_r}}  } \cdot \gamma(w_{\alpha_r}, \chi)^{-1}.
\end{equation}

First, if $\mca{O}_z^{R_\chi}= \set{z}$, then $\w_{\alpha_r}[z] = z \in \msc{X}_{Q,n}$. Write $z=\sum_{i=1}^r z_i e_i \in \msc{X}_{Q,n}$ with $0\le z_i \le n-1$. The equality $\w_{\alpha_r}[z] = z $ is equivalent to that $z_r = (n+1)/2$. It follows from \cite[Proposition 4.12]{GSS2} that
$$  \tau(w_{\alpha_r}, \chi, \s_{z}, \s_{z}) = \chi_{\alpha_r}^2 \cdot \gamma(w_{\alpha_r}, \chi)^{-1} = 1 \cdot \gamma(w_{\alpha_r}, \chi)^{-1} .$$
That is, \eqref{MC-C2} holds for such $\mca{O}_z^{R_\chi}$. In fact, we also see that the character $\theta_{\sigma^\msc{X}}$ of the representation $\sigma^\msc{X}:  R_\chi \to \text{Perm}(\msc{X}_{Q,n})$ is given by
$$\theta_{\sigma^\msc{X}}(1)= n^r, \quad \theta_{\sigma^\msc{X}}(\w_{\alpha_r})= n^{r-1}.$$

Second, assume $\mca{O}_z^{R_\chi}= \set{z, \w_{\alpha_r}[z]}$, then $n \nmid \angb{z_\rho}{\alpha_r}$. It follows from the proof of \cite[Lemma 3.9]{Ga2} that
$$k_{z, \alpha_r} +  k_{\w_{\alpha_r}[z], \alpha_r} = 1.$$
We obtain
$$\begin{aligned}
& \tau(w_{\alpha_r}, \chi, \s_{z}, \s_{z}) + \tau(w_{\alpha_r}, \chi, \s_{\w_{\alpha_r}[z]}, \s_{ \w_{\alpha_r}[z]}) \\
=& \frac{1-q^{-1}}{ 1-\chi_{\alpha_r} }  ( \chi_{\alpha_r})^{ k_{z, \alpha_r}   } + \frac{1-q^{-1}}{ 1-\chi_{\alpha_r} }  ( \chi_{\alpha_r})^{ k_{\w_{\alpha_r}[z], \alpha_r}  } \\
= & \frac{1-q^{-1}}{ 2}  ( -1 )^{ k_{z, \alpha_r}   } + \frac{1-q^{-1}}{ 2 }  ( -1 )^{ 1- k_{z, \alpha_r}   } \\
= & 0.
\end{aligned}$$
On the other hand, $\val{ ( \mca{O}_z^{R_\chi} )^{\w_{\alpha_r}}  } =0$; this shows that \eqref{MC-C2} holds in this case.

Therefore Conjecture \ref{MConj2} holds for $\wt{\Sp}_{2r}^{(n)}$. The desired dimension formula for the Whittaker spaces of $\pi_\chi^{un}$ and $\pi_\varepsilon$ follows from Theorem \ref{T:dW} and the character $\theta_{\sigma^\msc{X}}$ computed above. This completes the proof.
\end{proof}

\begin{rmk} \label{R:obsat}
The equality in \eqref{MC-C2} might fail for covers of a general simply-connected group. More precisely, for a general $\wt{G}$, we have the decomposition of a $W$-orbit $\mca{O}_y$ into $R_\chi$-orbits
$$\mca{O}_y = \bigsqcup_{i\in I} \mca{O}_{z_i}^{R_\chi}.$$
Conjecture \ref{MConj2} predicts that for every $\w \in R_\chi$, one has
\begin{equation}  \label{E:sumO}
\sum_{i\in I} \sum_{z' \in \mca{O}_{z_i}^{R_\chi}}  \tau(w, \chi, \s_{z'}, \s_{z'}) = \sum_{i\in I}  \val{ ( \mca{O}_{z_i}^{R_\chi} )^{\w}  } \cdot \gamma(w, \chi)^{-1}.
\end{equation}
However, the inner summands indexed by $I$ on the two sides may not be equal. A counter-example arises from considering $\wt{\SL}_4^{(3)}$ with $y=\sum_{\alpha \in \Delta} \alpha^\vee$, in which case $\val{\mca{O}_y}=6$. 

This subtlety is of the main difficulty with verifying Conjecture \ref{MConj2} by a direct computation. 
Indeed, it follows from Table 1 and Table 2 (or more precisely \cite[\S3]{Key2}) that the nontrivial unramified group $R_\chi$ for covers of simply-connected groups of type $B_r, D_r, E_6$ and $E_7$ is 
small; thus the orbits $\mca{O}_{z_i}^{R_\chi}$ are all small. However, as just noted, one needs to consider the whole $W$-orbit $\mca{O}_y$, whose size could be large depending on $W$ and $n$. This hinders one from a direct computation in the general situation.
\end{rmk}
\vskip 10pt

Theorem \ref{T:Sp2r} could also be obtained from the consideration in \S  \ref{S:UWf}, especially Conjecture \ref{C:inj}. We illustrate this by considering the case of $\wt{\SL}_2^{(n)}$. Write $n=2d+1$ and $\Delta=\set{\alpha}$. The twisted Weyl action on $\msc{X}_{Q,n}= \Z/n\Z$ is given by
$$\w_\alpha[k\alpha^\vee]= (1-k) \alpha^\vee \in \msc{X}_{Q,n}.$$
In total there are $(d+1)$-many $W$-orbits. Every orbit except that of $-d \alpha^\vee$ is free. We choose a set of representatives
of $\msc{X}_{Q,n}$ as
$$\mfr{R}=\set{i\cdot \alpha^\vee:  -d \le i \le d }.$$
The $W$-orbits in $\msc{X}_{Q,n}$ are
$$\mca{O}_{\msc{X}}= \set{ \mca{O}_{i\alpha^\vee}:  1\le i\le d  } \cup \set{ \mca{O}_{-d\alpha^\vee} }$$
with 
$$ \mfr{R}_{i\alpha^\vee}=\set{i\alpha^\vee, (1-i)\alpha^\vee} \text{ for } 1\le i\le d \text{ and } \mfr{R}_{-d\alpha^\vee}= \set{-d\alpha^\vee}.$$ 

\begin{prop}  \label{P:inj2}
Conjecture \ref{C:inj} holds for $\mfr{R}_{i\alpha^\vee}$ and moreover
 $${\rm rank}(\nu^\chi_{\mfr{R}_{i\alpha^\vee}})= 1$$
for every $1\le i \le d$ and $i=-d$.
\end{prop}
\begin{proof}
Since $-d\alpha^\vee \in Y$ is an exceptional point, in view of Theorem \ref{T:un-gen}, it suffices to deal with the case $1\le i \le d$. We denote $z:=i\alpha^\vee$. Let $\lambda_\cc \in \Wh(I(\chi))_{\mca{O}_z}$, viewed as an element in $\Wh (\pi_\chi^{un})_{\mca{O}_z}$ by restriction. Then 
$${\rm supp}(\cc) \subset \s_z \cdot \wt{A} \cup \s_{\w[z]} \cdot \wt{A}.$$
We have
$$\mca{W}_\cc^* = \cc(\s_z) \cdot \mca{W}_{\s_z}^* + \cc(\s_{\w[z]}) \cdot \mca{W}_{\s_{\w[z]}}^*.$$
Recall the projection map $h^{un}: \Wh(I(\chi))_{\mca{O}_z} \onto \Wh(\pi_\chi^{un})$ from \S \ref{S:UWf}. Assume $\cc$ is such that
\begin{equation} \label{temp0}
\nu_{\mfr{R}_z}^{\chi}(\lambda_\cc) = (\mca{W}_\cc^*(\s_z), \mca{W}_\cc^*(\s_{\w[z]})) = (0, 0).
\end{equation}
We want to show that $h^{un}(\lambda_\cc) = 0$. For each $z=i\alpha^\vee$, 
we denote
$$ \mca{M}_{\mfr{R}_z}:= \left(\begin{matrix}
\mca{W}^*_{\s_z}(\s_z)   &  \mca{W}^*_{\s_z}(\s_{\w[z]})  \\
\mca{W}^*_{\s_{\w[z]}}(\s_z)   &  \mca{W}^*_{\s_{\w[z]}}(\s_{\w[z]})
\end{matrix}\right).
$$
It is easy to see that
\begin{equation} \label{tempM}
\nu_{\mfr{R}_z}^{\chi}(\lambda_\cc)=(\cc(\s_z), \cc(\s_{\w[z]})) \mca{M}_{\mfr{R}_z}.
\end{equation}
Now by \eqref{E:W*}, one has
$$\mca{W}^*_{\s_z}(\s_y)\cdot \delta_B^{1/2}( \s_y ) = \gk(w_\alpha, \chi)  \cdot \tau(\text{id}, \chi, \s_z,  \s_y) +  \tau(w_\alpha, {}^{w_\alpha} \chi, \s_z,  \s_y).$$
Since we are in the case where $\uchi_\alpha^2=\mbm{1}$ but $\uchi_\alpha \ne \mbm{1}$, we have 
$$\uchi_\alpha(\varpi)= \chi(\wt{h}_\alpha(\varpi^n))= -1.$$
We also note that ${}^{w_\alpha} \chi= \chi$. Thus a straightforward computation gives that
$$ \mca{M}_{\mfr{R}_{i\alpha^\vee}}= 
 \left(\begin{matrix}
\mca{W}^*_{\s_{i\alpha^\vee}}(\s_{ i\alpha^\vee  })   &  \mca{W}^*_{\s_{i\alpha^\vee}}(\s_{ (1-i)\alpha^\vee  })  \\
\mca{W}^*_{\s_{(1-i)\alpha^\vee}}(\s_{ i\alpha^\vee  })  & \mca{W}^*_{\s_{(1-i)\alpha^\vee}}(\s_{(1- i)\alpha^\vee  }) 
\end{matrix}\right)
=
\left(\begin{matrix}
q^{i-1} &  q^{-i-1} \g(1-2i) \\
q^i \g(2i-1) & q^{-i-1}.
\end{matrix}\right).$$
Combining \eqref{temp0} and \eqref{tempM}, we get
\begin{equation} \label{rel0}
q^{-1} \cdot \cc(\s_{i\alpha^\vee}) + \g(2i-1) \cdot \cc(\s_{(1-i)\alpha^\vee})=0.
\end{equation}
Consider the map
$$\begin{tikzcd}
P_{\mathbbm{1}}^*: \Wh(I(\chi)) \ar[r, hook] & \Wh(I(\chi))
\end{tikzcd}
$$
induced from the projection $P_{\mathbbm{1}}: I(\chi) \to I(\chi)$. To show $h^{un}(\lambda_\cc)=0$ is equivalent to proving $P_{\mathbbm{1}}^*(\lambda_\cc)=0$. Now
$$\begin{aligned}
& P_{\mathbbm{1}}^*(\lambda_\cc) \\
= & \frac{1}{2} \cdot (\lambda_\cc + \msc{A}(w, \chi)^*(\lambda_\cc)) \\
= & \frac{1}{2} \big(\cc(\s_z) + \cc(\s_z) \gamma(w, \chi) \tau(w, \chi, z, z) + \cc(\s_{\w[z]}) \gamma(w, \chi) \tau(w, \chi, \w[z], z) \big)\cdot \lambda_{\s_z}  \\
& \  + \frac{1}{2} \big(\cc(\s_{\w[z]}) + \cc(\s_z) \gamma(w, \chi) \tau(w, \chi, z, \w[z]) + \cc(\s_{\w[z]}) \gamma(w, \chi) \tau(w, \chi, \w[z], \w[z]) \big)\cdot \lambda_{\s_{\w[z]}}.
\end{aligned}
$$
A simplification gives that the above coefficient in front of $\lambda_{\s_z}$ is
$$\frac{1}{1+q^{-1}} \cdot \big(\cc(\s_{i\alpha^\vee}) q^{-1} + \cc(\s_{(1-i)\alpha^\vee}) \g(2i-1) \big),$$
which is equal to 0 by \eqref{rel0}. Similarly, it can be checked easily that the coefficient in front of $\lambda_{\s_{\w[z]}}$ is also 0. This shows that Conjecture \ref{C:inj} holds.

It is clear that for every $1\le i \le d$, we have
$${\rm rank}(\nu^\chi_{\mfr{R}_{i\alpha^\vee}})= {\rm rank}(\mca{M}_{\mfr{R}_{i\alpha^\vee}})=1.$$
The proof is now completed.
\end{proof}

It follows from Proposition \ref{P:inj2} that
$$\dim \Wh(\pi_\chi^{un})= \angb{\mathbbm{1}}{ \sigma^\msc{X} }_{R_\chi}= \val{ \mca{O}_{\msc{X}}  }=d + 1= \frac{n+1}{2}.$$
Consequently, 
$$\dim \Wh(\pi_\varepsilon)= \angb{\varepsilon}{\sigma^\msc{X}}_{R_\chi}= d = \frac{n-1}{2} ,$$
the number of free $R_\chi$-orbits in $\msc{X}_{Q,n}$.

\begin{rmk}
Let $n$ be odd and $\chi$ be the above non-trivial quadratic genuine character of $Z(\wt{T}) \subset \wt{\Sp}_{2r}$. The Whittaker dimension for the constituents in $\pi_\chi^{un} \oplus \pi_\chi' = \text{Ind}_{\wt{B}}^{ \wt{\Sp}_{2r} }(i(\chi))$ can be deduced from that of $\wt{\SL}_2^{(n)}$ as follows. Here we write $\pi_\chi'$ for $\pi_\varepsilon$. Each of the rank-one lattice $(\Z e_j) \subset Y$ gives rise to an $n$-fold covering $\wt{\GL}_{1,e_j}$  of the torus $\GL_1 \simeq F^\times$, by restriction from $\wt{T}$. One has an isomorphism (see \cite[\S 5.1.3]{GW})
\begin{equation} \label{F:BC}
\prod_{1\le j \le r} \wt{\GL}_{1,e_j}\big/ H \simeq  \wt{T},
\end{equation}
where $H=\set{(\zeta_j) \in (\bbmu_n)^r:   \prod_j \zeta_j =1 }$; that is, block-commutativity holds for coverings of Levi-subgroups of $\wt{\Sp}_{2r}$. Thus, we can write
$$i(\chi)= \prod_{j=1}^r i(\chi_j),$$
where $i(\chi_j) \in \Irr( \wt{\GL}_{1,e_j} )$ is of dimension $n$. Note also $\wt{T}_0:=\wt{\GL}_{1,e_r}$ is just the covering torus of $\wt{\SL}_2$ associated to $\alpha_r$. Let $\wt{M}= \prod_{1\le j \le r-1}  \wt{\GL}_{1,e_j} \times \wt{\SL}_2$ be the Levi subgroup of the parabolic subgroup $\wt{P} \subset \wt{\Sp}_{2r}^{(n)}$ associated to $\alpha_r$. The character $\chi_r$ is a non-trivial quadratic character of $Z(\wt{T}_0)$ and thus we have
$$\text{Ind}_{\wt{T}_0}^{  \wt{\SL}_2 }   ( i(\chi_r))  =  \pi_{\chi_r}^{un}  \oplus \pi_{\chi_r}'.$$
By induction in stages, one has
$$\pi_\chi^{un} = \text{Ind}_{\wt{P}}^{ \wt{\Sp}_{2r}  } (\boxtimes_{1\le j \le r-1}  i(\chi_j)) \boxtimes  \pi_{ \chi_r }^{un} $$
and similarly,
$$\pi_\chi' = \text{Ind}_{\wt{P}}^{ \wt{\Sp}_{2r}  } (\boxtimes_{1\le j \le r-1}  i(\chi_j)) \boxtimes  \pi_{ \chi_r }'.$$
Now it follows from the equalities
$$\dim i(\chi_j) = n, \quad \dim \Wh( \pi_{ \chi_r }^{un} ) = \frac{n + 1}{ 2}, \quad \dim \Wh( \pi_{ \chi_r }') = \frac{n - 1}{ 2} $$
and Rodier's heredity that
$$\dim \Wh(\pi_\chi^{un}) = n^{r-1}\cdot \frac{n + 1}{ 2 }, \quad \dim \Wh(\pi_\chi') = n^{r-1}  \cdot \frac{n -1}{ 2 },$$
which agrees with Theorem \ref{T:Sp2r}
\end{rmk}

\vskip 5pt

\subsection{Double cover of $\SL_3$} 
Before we proceed, we recall some observations from \S \ref{S:formu}. By the tables in  \S \ref{T1} and  Corollary \ref{T:R-abel}, the group $R_\chi$ is always cyclic for all semisimple type except the $D_r$ case when  $r$ is even. Recall that 
$$\msc{A}(w, \chi)^*= \gamma(w, \chi) \cdot A(w, \chi)^*.$$
Assume $R_\chi$ is cyclic and let $\w$ be a generator; then $\sigma(\w), \sigma \in \Irr(R_\chi)$ are distinct and $\dim \Wh(\pi_\sigma), \sigma \in \Irr(R_\chi)$ are just the multiplicities of the distinct eigenvalues $\sigma(\w) \cdot \gamma(w, \chi)^{-1}$ of the polynomial 
$$\det \left( X \cdot \text{id} - A(w, \chi)^*  \right)= \det \left( X \cdot I_{\val{\msc{X}_{Q,n}}} - \mca{S}_\mfr{R}(w, i(\chi)) \right).$$
This will be the observation we apply to  the double cover $\wt{\SL}_3^{(2)}$ in this subsection. 

Let $\alpha_1^\vee, \alpha_2^\vee$ be the two simple co-roots of $\SL_3$:
$$\qquad 
\begin{picture}(2.7,0.2)(0,0)
\put(1,0){\circle{0.08}}
\put(1.5,0){\circle{0.08}}
\put(1.04,0){\line(1,0){0.42}}
\put(0.95,0.1){\footnotesize $\alpha_{1}^\vee$}
\put(1.45,0.1){\footnotesize $\alpha_{2}^\vee$}
\end{picture}
$$
\vskip 10pt
\noindent For convenience, we write $\w_i= \w_{\alpha_i}$ for $i=1, 2$.  Let $\alpha_3=\alpha_1^\vee + \alpha_2^\vee \in \Phi^+$. Let $Q: Y\to \Z$ be the unique Weyl-invariant quadratic form such that $Q(\alpha_1^\vee)=1$. Taking $n=2$, we get
$$Y_{Q,2}= Y_{Q,2}^{sc}= 2Y$$
and thus
$$\msc{X}_{Q,2} \simeq (\Z/2\Z) \oplus (\Z/2\Z).$$
The ordered set
$$\mfr{R}= \set{0, \alpha_1^\vee, \alpha_2^\vee, \alpha_3^\vee } \subset Y$$
is a set of representatives of $\msc{X}_{Q,2}$. There are two $W$-orbits
$$\mca{O}_0 =\set{0, \alpha_1^\vee, \alpha_2^\vee}, \quad \mca{O}_{\alpha_3^\vee}=\set{\alpha_3^\vee}.$$
 Let $\chi$ be a unitary unramified genuine character of $Z(\wt{T})$. Since the dual group of $\wt{\SL}_3^{(2)}$ is $\PGL_3$,  we see that $R_\chi^{sc} = R_\chi$ by Proposition \ref{P:bd}; moreover, $R_\chi$ is either trivial or $\Z/3\Z$.

Assume $R_\chi = \Z/3\Z$, we see that $\Phi_\chi = \emptyset$ and $R_\chi=W_\chi= \langle \w_1 \w_2 \rangle = \langle \w_2 \w_1 \rangle \subset W$, i.e. 
$${}^{\w_2 \w_1} \chi= \chi.$$
This implies that
\begin{equation} \label{val-chi}
\chi_{\alpha_1}= \zeta = \chi_{\alpha_2}, \text{ and } \chi_{\alpha_3} = \zeta^2,
\end{equation}
where $\zeta \in \C^\times$ is a primitive third-root of unity. For such $\chi$, one has the decomposition
$$I(\chi) = \pi_\chi^{un} \oplus \pi_1 \oplus \pi_2$$
of $I(\chi)$ into non-isomorphic irreducible components. We have $\dim \Wh(I(\chi))=4$, and the permutation representation 
$$\sigma^\msc{X}: R_\chi \longrightarrow  \perm(\msc{X}_{Q,2})$$
is such that in $\msc{X}_{Q,2}$:
$$\w_2\w_1[\alpha_3^\vee]= \alpha_3^\vee, \quad \w_2\w_1[0]= \alpha_1^\vee, \quad \w_2\w_1[\alpha_1^\vee]= \alpha_2^\vee, \quad  \w_2\w_1[\alpha_2^\vee]= \alpha_0^\vee.$$
We have $\Irr(R_\chi)=\set{\mathbbm{1}, \sigma, \sigma^2}$ where $\sigma$ is the generator given by
$$\sigma(\w_2 \w_1) =\zeta.$$
It then follows easily that
$$\sigma^\msc{X}= (2\cdot \mbm{1}) \oplus \sigma \oplus \sigma^2.$$
Thus, we could label constituents of $I(\chi)$ as
 $$\pi_\chi^{un} = \pi_{\mbm{1}} = \pi_{\sigma^0}, \quad \pi_1= \pi_{\sigma} \text{ and }  \pi_2= \pi_{\sigma^2}.$$
Since $R_\chi$ is cyclic and $\w_2 \w_1$ is a generator of $\Irr(R_\chi)$, to determine $\dim \Wh(\pi_{\sigma^i})$, it suffices to compute the characteristic polynomial of $A(w_2 w_1, \chi)^*$ which takes the form
$$\det \left( X\cdot {\rm id}- A(w_2 w_1, \chi)^* \right)=
\prod_{0 \le i \le 2} \left( X - \zeta^i \cdot \gamma(w_2 w_1, \chi)^{-1} \right)^{\dim \Wh(\pi_{\sigma^i})},$$
where 
$$\gamma(w_2 w_1, \chi)^{-1} = \frac{(1-q^{-1} \chi_{\alpha_1}^{-1}) (1- q^{-1} \chi_{\alpha_3})^{-1}}{ (1-\chi_{\alpha_1})(1-\chi_{\alpha_3}) } = \frac{1+ q^{-1}  +  q^{-2}}{3}.$$

Let $\mca{S}_\mfr{R}(w_2w_1, i(\chi))$ be the scattering matrix associated to the ordered set $\mfr{R}$ above. For simplicity of computation, we assume $\bbmu_4 \subset F^\times$, and hence $\varepsilon =1$. By using \eqref{SLCM1}, \eqref{SLCM2} and Theorem \ref{T:tau}, we obtain in this case an explicit form (again, using the short-hand notation $\chi_{\alpha_1}, \chi_{\alpha_2}, \chi_{\alpha_3}$):
$$\begin{aligned}
& \mca{S}_\mfr{R}(w_2w_1, i(\chi)) \\
=&
\left(
\begin{matrix}
\frac{(1-q^{-1})^2}{(1-\chi_{\alpha_1})(1-\chi_{\alpha_3}) }  & \g(-1) \gamma(\uchi_{\alpha_3})^{-1} &   \g(-1)  \frac{1-q^{-1}}{1-\chi_{\alpha_1}  }  & 0 \\
\g(1) \frac{1-q^{-1}}{ 1-\chi_{\alpha_3} }  &   \chi_{\alpha_1} \frac{1-q^{-1}}{1-\chi_{\alpha_1}  } \gamma(\uchi_{\alpha_3})^{-1} &  q^{-1} & 0  \\
\g(1) \gamma(\uchi_{\alpha_1})^{-1} & 0 &  \chi_{\alpha_3} \frac{1-q^{-1}}{1-\chi_{\alpha_3}  } \gamma(\uchi_{\alpha_1})^{-1} & 0 \\
0 &  0 & 0 & \chi_{\alpha_1} \chi_{\alpha_3} \gamma(\uchi_{\alpha_1})^{-1} \gamma(\uchi_{\alpha_3})^{-1} 
\end{matrix}
\right)
\end{aligned}
$$
A straightforward computation gives that
$$\begin{aligned}
&  \det\left( X \cdot I_4- \mca{S}_\mfr{R}(w_2w_1, i(\chi)) \right) \\
= & \left( X- \frac{1+q^{-1} + q^{-2}}{3} \right) \cdot \left( X^3- \left( \frac{1+q^{-1} + q^{-2}}{3} \right)^3 \right),
\end{aligned} $$
and thus
$$\det\left( X \cdot I_4 -  \msc{A}(w_2 w_1, \chi)^* \right)= (X- 1)^2 \cdot (X- \zeta ) \cdot (X- \zeta^2).$$
Therefore,
$$\dim \Wh(\pi_{\mbm{1}}) =2 \text{ and } \dim \Wh(\pi_{\sigma^i})=1 \text{ for } i =1, 2.$$
Clearly,
$$\sigma^{\rm Wh}  = (2\cdot \mbm{1}) \oplus \sigma \oplus \sigma^2.$$

\begin{prop} \label{T:SL3}
For $\wt{\SL}_3^{(2)}$ we have
$$\sigma^{\rm Wh} = \sigma^\msc{X} = (2\cdot \mbm{1}) \oplus \sigma \oplus \sigma^2.$$
Moreover, Conjecture \ref{MConj2}  holds.
\end{prop}
\begin{proof}
The equalities are clear. It suffices to prove that Conjecture \ref{MConj2} holds for the two orbits $\mca{O}_0$ and $\mca{O}_{\alpha_3^\vee}$, which is a priori stronger than the equalities. This either follows from a direct computation, or we argue alternatively by using the fact that $ \mca{O}_\msc{X} = \set{\mca{O}_0, \mca{O}_{\alpha_3^\vee}   }$ with $\alpha_3^\vee \in \msc{X}_{Q,2}^{\rm exc}$.  Indeed, by Theorem \ref{T:unWh}, $\sigma^{\rm Wh}_{ \mca{O}_{\alpha_3^\vee} } = \sigma^{\rm Wh}_{ \mca{O}_{\alpha_3^\vee} } =\mbm{1}$. However, since
$$\sigma^{\rm Wh}= \sigma^{\rm Wh}_{ \mca{O}_{\alpha_3^\vee} } \oplus \sigma^{\rm Wh}_{ \mca{O}_0 } = \sigma^\msc{X}_{ \mca{O}_{\alpha_3^\vee} } \oplus \sigma^\msc{X}_{ \mca{O}_0 } =\sigma^\msc{X},$$
it enforces that
$$\sigma^{\rm Wh}_{ \mca{O}_0 } = \sigma^{\rm Wh}_{ \mca{O}_0 } = \mbm{1} \oplus \sigma \oplus \sigma^2.$$
The proof is completed.
\end{proof}

\section{Two remarks} \label{S:2rmk}
In this section, we consider two examples to justify the necessary constraints imposed on $G$ and $\wt{G}$ in Conjecture \ref{MConj2}. First,  we  consider $\wt{\SO}_3^{(n)}$ and show that a naive analogous formula does not hold for general covers of semisimple groups which are not simply connected. Second, we consider the double cover of the simply-connected $\Spin_6 \simeq  \SL_4$, whose dual group is $\SL_4/\mu_2$, and show that analogous Conjecture \ref{MConj2} does not hold. This shows that it is necessary to require the cover $\wt{G}$ to be saturated.

\subsection{Covers of $\SO_3$} Let $Y=\Z\cdot e$ be the cocharacter lattice of $\SO_3$ with $\alpha^\vee=2e$  generating the co-root lattice $Y^{sc}$. Let $Q: Y\to \Z$ be the Weyl-invariant quadratic form such that $Q(e)=1$. Thus, $Q(\alpha^\vee)=4$. We get
$$\alpha_{Q,n}^\vee= \frac{n}{\text{gcd}(4, n)} \cdot \alpha^\vee.$$
On the other hand, 
$$Y_{Q,n} = \Z \cdot  \frac{n}{\text{gcd}(2, n)} e .$$
The equality ${}^{w_\alpha} \chi = \chi$ is equivalent to that
$$\chi\left( \wt{h}_\alpha(a^{n/\text{gcd}(2,n)})   \right)=1$$
for all $a\in F^\times$. 

\begin{lm}
Let $\chi$ be a unitary unramified genuine character of $Z(\wt{T})$. Then $R_\chi = W$ if and only if $4|n$ and $\uchi_\alpha$ is a non-trivial quadratic character. 
\end{lm}
\begin{proof}
Clearly, $R_\chi =W$ if and only if $\Phi_\chi = \emptyset$ and $W_\chi= W$. We discuss case by case. First, if $4\nmid n$, then $\text{gcd}(4,n) = \text{gcd}(2,n)$, and in this case, $R_\chi=\set{1}$. Second, if $n=4m$, then $\alpha_{Q,n}^\vee=m \alpha^\vee$ and $n/\text{gcd}(2, n) =2m$. In this case, if $\uchi_\alpha$ is a non-trivial quadratic character, then we have $R_\chi= W$. 
\end{proof}

\begin{rmk}
For $G$ of adjoint type, if $\chi$ is a unitary unramified character of $T$, then $I(\chi)$ is always irreducible (see Corollary \ref{C:R=1}). The above result shows that this may fail for covers of groups of adjoint type.
\end{rmk}

\vskip 5pt

Now we assume that $n=4m$ and $\uchi_\alpha$ is a non-trivial quadratic character, i.e.
$$\chi_\alpha = \chi( \wt{h}_\alpha(\varpi^{n_\alpha}) )= -1.$$
In this case, $R_\chi = W$ and
$$I(\chi) = \pi_\chi^{un} \oplus \pi.$$
We have
$$Y_{Q,n}=Y_{Q,n}^{sc} =\Z \cdot \alpha_{Q,n}^\vee = \Z \cdot (m\alpha^\vee) = \Z \cdot (2me).$$
Therefore, $\wt{G}^\vee \simeq \SO_3$ is of adjoint type. It is clear that 
$$\msc{X}_{Q,n} \simeq \Z/ (2m)\Z$$
with the twisted Weyl action given by
$$\w_\alpha[ie]= (-i)e + 2e = (2-i)e.$$
One has $\val{ \mca{O}_\msc{X} }= m+1$, i.e., there are $m+1$ many $W$-orbits in $\msc{X}_{Q,n}$. Let $\mfr{R} \subset Y$ be the following set of representatives of $\msc{X}_{Q,n}$:
$$\mfr{R}= \set{ ie: -m + 1 \le i\le m}.$$
The two trivial $W$-orbits are
$$\mca{O}_{e}= \set{e}, \quad \mca{O}_{(-m+1)e}= \set{(-m+1)e};$$
while for all other $ie\in \mfr{R}$ with $2\le i \le m$, the orbit
$$\mca{O}_{ie}= \set{ie, (2-i)e} \subset \msc{X}_{Q,n}$$
is $W$-free. We thus have $\mfr{R}_{e}, \mfr{R}_{(-m+1)e}$ and $\mfr{R}_{ie} \subset \mfr{R}$ to represent the above three families of orbits.

\begin{prop} \label{P:SO3}
Assume $n=4m$ and $\uchi_\alpha$ is a non-trivial quadratic character. Then 
$$\sigma^{\rm Wh} = m \cdot \mbm{1}_W \oplus m\cdot \varepsilon_W, \text{ and } \sigma^\msc{X}= (m+1) \cdot \mbm{1}_W \oplus (m-1)\cdot \varepsilon_W.$$
Hence,
$$\dim \Wh(\pi_\chi^{un})= m = \dim \Wh(\pi).$$
\end{prop}
\begin{proof}
Choosing $\mfr{R}$ as above, the scattering matrix $\mca{S}_\mfr{R}(w_\alpha, i(\chi))$ is the block-diagonal matrix with blocks $\mca{S}_\mfr{R}(w_\alpha, i(\chi))_{\mca{O}_{ie}}$ for $i=-m+1$ and $1\le i \le m$. Here,
$$\mca{S}_\mfr{R}(w_\alpha, i(\chi))_{\mca{O}_{(-m+1)e}}= \gamma(w_\alpha, \chi)^{-1}= \frac{1-q^{-1} \chi_\alpha}{1-\chi_\alpha}= \frac{1+q^{-1}}{2}.$$
and
$$\mca{S}_\mfr{R}(w_\alpha, i(\chi))_{\mca{O}_{e}}=\chi_\alpha \cdot \gamma(w_\alpha, \chi)^{-1} = - \frac{1+q^{-1}}{2};$$
also, for $2\le i \le m$,
$$
\mca{S}_\mfr{R}(w_\alpha, i(\chi))_{\mca{O}_{ie}}=
\left(\begin{matrix}
\frac{1-q^{-1}}{1-\chi_\alpha} \chi_\alpha &  \g((i-1)4) \\
\g((1-i)4) & \frac{1-q^{-1}}{1-\chi_\alpha} 
\end{matrix}\right).
$$
It follows that for $2\le i \le m$,
$$
\sigma_{\mca{O}_{ie}}^{\rm Wh} = \sigma_{\mca{O}_{ie}}^\msc{X} = \mbm{1} \oplus \varepsilon,
$$
and
$$\sigma_{\mca{O}_{(-m+1)e}}^{\rm Wh} = \sigma_{\mca{O}_{(-m+1)e}}^\msc{X} =\mbm{1};$$
however,
$$\sigma_{\mca{O}_{e}}^{\rm Wh}  =  \varepsilon_W, \quad  \sigma_{\mca{O}_{e}}^\msc{X} =\mbm{1}.$$
The last result on Whittaker dimension follows from Theorem \ref{T:dW}.
\end{proof}

One may also prove Proposition \ref{P:SO3} by using the method described in \S \ref{S:formu}. That is, we can compute $\dim \Wh(\pi_\chi^{un})$ by showing that
\begin{enumerate}
\item[(i)] Conjecture \ref{C:inj} holds for $\wt{\SO}_3$, and
\item[(ii)] we have
$${\rm rank}(\nu_{\mfr{R}_{ie}}^\chi)=
\begin{cases}
1 & \text{ if } 2\le i \le m \text{ or } i=-m+1;\\
0 & \text{ if } i=1.
\end{cases}
$$
\end{enumerate}

Here (i) can be verified exactly in the same way as Proposition \ref{P:inj2}, and thus we omit the details. We discuss (ii) for the three cases $i=-m + 1$, $i=1$ and $2\le i \le m$ separately.
\begin{enumerate}
\item[$\bullet$]  First, since $(-m+1)e = \rho - \rho_{Q,n}$, it is the unique element in $Y_n^{\rm exc}$. In this case, the equality ${\rm rank}(\nu_{\mfr{R}_{(1-m)e}}^\chi)=1$ follows from Theorem \ref{T:un-gen}. 
\item[$\bullet$] Second, for $\mfr{R}_{e}= \set{e}$, we have ${\rm rank}(\nu_{\mfr{R}_{e}}^\chi)=1$ if and only if $\mca{W}^*_{\s_e}(\s_e) \ne 0$. A straightforward computation from \eqref{E:W*} gives that 
$$\mca{W}^*_{\s_e}(\s_e)= 0.$$
Thus, ${\rm rank}(\nu_{\mca{O}_{e}}^\chi)=0$.
\item[$\bullet$] Third, we deal with free $W$-orbits $\mca{O}_{ie}, 2\le i \le m$. Similar to the case of  $\wt{\SL}_2^{(n)}$,  we have ${\rm rank}(\nu^\chi_{\mfr{R}_{ie}})= {\rm rank}(\mca{M}_{\mfr{R}_{ie}})$ where
$$ \mca{M}_{\mfr{R}_{ie}}:= \left(\begin{matrix}
\mca{W}^*_{\s_{ie}}(\s_{ ie  })   &  \mca{W}^*_{\s_{ie}}(\s_{ (2-i)e  })  \\
\mca{W}^*_{\s_{(2-i)e}}(\s_{ ie  })  & \mca{W}^*_{\s_{(2-i)e}}(\s_{(2- i)e  }) 
\end{matrix}\right).$$
Again, since ${}^{w_\alpha} \chi = \chi$ and $\chi_\alpha= -1$, it follows from \eqref{E:W*} that
$$ \mca{M}_{\mfr{R}_{ie}}= \left(\begin{matrix}
q^{(i-2)/2}  &  q^{-(2+i)/2} \g((1-i)4) \\
q^{i/2} \g((i-1)4) & q^{-(2+i)/2}.
\end{matrix}\right).$$
Note that we have $1\le i-1 \le m-1$ and thus $\det(\mca{M}_{\mfr{R}_{ie}} )=0$. Clearly, this implies that ${\rm rank}( \mca{M}_{\mfr{R}_{ie}} )=1$. 
\end{enumerate}

Combining the above gives that $\dim \Wh(\pi_\chi^{un})= m$.  It follows from this example of $\wt{\SO}_3^{(n)}$ that a naive analogue of Conjecture \ref{MConj2} does not hold for coverings of a general semisimple group. Here the difference between $\msc{X}_{Q,n}^{\rm exc}$ and $(\msc{X}_{Q,n})^W$ plays a sensitive role and accounts for the failure. Indeed, in the case of $\wt{\SO}_3^{(n)}$, one has
$$\val{ \msc{X}_{Q,n}^{\rm exc} }=1 \text{ and } \val{ (\msc{X}_{Q,n})^W } =2.$$

\subsection{Double cover of $\Spin_6$} 
We consider in this subsection only double cover of $\Spin_6\simeq \SL_4$, though the phenomenon appears for 
general $2m$-fold cover of $\Spin_{2k}$ with $m$ and $k$ being both odd. For this reason, we would like to consider the situation from the perspective of spin groups.

Consider the Dynkin diagram of simple coroots for the simply-connected $G = \Spin_6$:
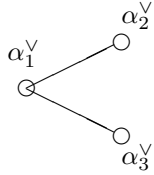
$$
\begin{picture}(6.7,0.4)(0,0)
\put(3,0){\circle{0.08}}
\put(3.5, 0.25){\circle{0.08}}
\put(3.5, -0.25){\circle{0.08}}
%
\put(3.00,0){\line(2,1){0.46}}
\put(3.00,0){\line(2,-1){0.46}}
%
\put(2.9,0.15){\footnotesize $\alpha_1^\vee$}
\put(3.5,0.35){\footnotesize $\alpha_2^\vee$}
\put(3.5,-0.4){\footnotesize $\alpha_3^\vee$}
\end{picture}
$$

\vskip 30pt
\noindent Let $Q$ be the Weyl-invariant quadratic form $Q$ of $Y$ such that $Q(\alpha_i^\vee)=1$ for all $1\le i \le 3$. Let $\wt{G}$ be the double cover of $G$ arising from $Q$. One has $Y_{Q,2}^{sc}= 2 \cdot Y$ and 
$$Y_{Q,2}= \set{ \sum_{i=1}^3 y_i \alpha_i^\vee:  2| y_i \text{ for all $i$ and } 2(y_1 + y_2 + y_3) }.$$
Thus, we get
$$\wt{G}^\vee = \SO_6$$
and that the principal endoscopic group $H$ for $\wt{G}$ is $\SO_6$. We have 
$$\msc{X}_{Q,2}= \set{0, \alpha_1^\vee, \alpha_2^\vee, \alpha_1^\vee + \alpha_2^\vee  }.$$
Note that $\alpha_2^\vee = \alpha_3^\vee \in \msc{X}_{Q,n}$. There are two $W$-orbits of $\msc{X}_{Q,2}$ represented by the following graph:
$$
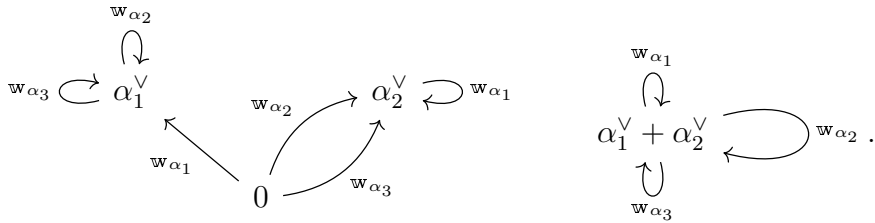
\begin{tikzcd}
\alpha_1^\vee  \ar[loop above, "{\w_{\alpha_2}}"]  \ar[loop left, "{\w_{\alpha_3}}"] & & \alpha_2^\vee \ar[loop right, "{\w_{\alpha_1}}"] \\
&  0 \ar[lu, "{\w_{\alpha_1}}"]  \ar[ru, bend left=30, "{\w_{\alpha_2}}"] \ar[ru, bend right=30, "{\w_{\alpha_3}}"']
\end{tikzcd}
\qquad 
\begin{tikzcd}
\alpha_1^\vee + \alpha_2^\vee \ar[loop above, "{\w_{\alpha_1}}"]   \ar[loop right, "{\w_{\alpha_2}}"]   \ar[loop below, "{\w_{\alpha_3}}"] 
\end{tikzcd}
.
$$
We have $\msc{X}_{Q,n}= \mca{O}_0 \cup \mca{O}_{\alpha_1^\vee + \alpha_2^\vee}$.
\vskip 10pt

It then follows from \cite[Theorme 6.8]{Gol1} that the only nontrivial unramified $R_\chi$ is  $\set{1, \w=\w_{\alpha_{r-1}} \w_{\alpha_r}  }$ with 
$$\chi_{\alpha_2} = \chi_{\alpha_3} = -1.$$
A direct computation using \eqref{SLCM1}, \eqref{SLCM2} and Theorem \ref{T:tau} gives that
$$\begin{aligned}
\tau(w, \chi, \s_0, \s_0) + \tau(w , \chi, \s_{\alpha_2^\vee}, \s_{\alpha_2^\vee}) & = \gamma(w, \chi) - \gamma(w, \chi) = 0 \\
\tau(w, \chi, \s_{\alpha_1^\vee}, \s_{\alpha_1^\vee})  & = \gamma(w, \chi) \\
\tau(w, \chi, \s_{\alpha_1 + \alpha_2^\vee}, \s_{\alpha_1 + \alpha_2^\vee})  & = -  \gamma(w, \chi) .
\end{aligned} $$
Denoting $\Irr(R_\chi)= \set{\mbm{1}, \varepsilon}$, it then follows that
$$\sigma_{\mca{O}_0}^{\rm Wh}= (2 \cdot \mbm{1}) \oplus \varepsilon, \quad \sigma_{\mca{O}_{\alpha_1^\vee + \alpha_2^\vee}}^{\rm Wh} = \varepsilon.$$
In particular, writing $I(\chi) = \pi_\chi^{un} \oplus \pi$, we have
$$\dim \Wh(\pi_\chi^{un}) = 2, \text{ and }   \dim \Wh(\pi) = 2.$$
On the other hand, it is clear from above graph that
$$\sigma_{\mca{O}_0}^\msc{X}= 3 \cdot \mbm{1} , \quad \sigma_{\mca{O}_{\alpha_1^\vee + \alpha_2^\vee}}^\msc{X} = \mbm{1}.$$
We see that analogous Conjecture \ref{MConj2} does not hold in this case. The constraint that $\wt{G}^\vee$ is of adjoint type seems to be necessary.
\vskip 10pt

For a low-rank group and ``small" $R_\chi$, one should be able to compute and verify explicitly Conjecture \ref{MConj2}. However, for general $n$-fold covers of a simply-connected group $G$, in view of the difficulty highlighted in Remark \ref{R:obsat}, it is desirable to approach the problem from a more uniform and conceptual perspective. In any case, we will leave the investigation of this to a future work, as a continuation of the present paper.


\end{document}